\documentclass[reqno]{amsart}
\usepackage{mathrsfs}
\usepackage{color}
\usepackage{amsmath}
\usepackage{amsfonts}
\usepackage{amssymb}
\usepackage{graphicx}
\usepackage{hyperref}

 \newtheorem{Theorem}{Theorem}[section]
 \newtheorem{Corollary}[Theorem]{Corollary}
 \newtheorem{Lemma}[Theorem]{Lemma}
 \newtheorem{Proposition}[Theorem]{Proposition}
 
 \newtheorem{Definition}[Theorem]{Definition}

 \newtheorem{Remark}[Theorem]{Remark}

 \numberwithin{equation}{section}

\begin{document}

\title[Boundary points, minimal $L^2$ integrals and concavity property]
 {Boundary points, minimal $L^2$ integrals and concavity property \uppercase\expandafter{\romannumeral2}: on weakly pseudoconvex K\"ahler manifolds}


\author{Qi'an Guan}
\address{Qi'an Guan: School of
Mathematical Sciences, Peking University, Beijing 100871, China.}
\email{guanqian@math.pku.edu.cn}

\author{Zhitong Mi}
\address{Zhitong Mi: Institute of Mathematics, Academy of Mathematics
and Systems Science, Chinese Academy of Sciences, Beijing, China
}
\email{zhitongmi@amss.ac.cn}

\author{Zheng Yuan}
\address{Zheng Yuan: School of
Mathematical Sciences, Peking University, Beijing 100871, China.}
\email{zyuan@pku.edu.cn}

\thanks{}



\date{\today}

\dedicatory{}

\commby{}


\begin{abstract}
In this article, we consider minimal $L^2$ integrals on the sublevel sets of plurisubharmonic functions on weakly pseudoconvex K\"ahler manifolds with Lebesgue measurable gain related to modules at boundary points of the sublevel sets,
and establish a concavity property of the minimal $L^2$ integrals.
As applications,
we present a necessary condition for the concavity degenerating to linearity,
a concavity property related to modules at inner points of the sublevel sets,
an optimal support function related to the modules,
a strong openness property of the modules and a twisted version,
an effectiveness result of the strong openness property of the modules.

\end{abstract}

\maketitle
\section{Introduction}
Let $\varphi$ be a plurisubharmonic function on a complex manifold $M$ (see \cite{Demaillybook}). Recall that the multiplier ideal sheaf $\mathcal{I}(\varphi)$ is the sheaf of germs of holomorphic functions $f$ such that $|f|^2e^{-\varphi}$ is locally integrable,
which plays an important role in several complex variables, complex algebraic geometry and complex differential geometry (see e.g. \cite{Tian,Nadel,Siu96,DEL,DK01,DemaillySoc,DP03,Lazarsfeld,Siu05,Siu09,DemaillyAG,Guenancia}).

The strong openness property of multiplier ideal sheaves \cite{GZSOC}, i.e. $\mathcal{I}(\varphi)=\mathcal{I}_+(\varphi):=\mathop{\cup} \limits_{\epsilon>0}\mathcal{I}((1+\epsilon)\varphi)$
(conjectured by Demailly \cite{DemaillySoc}) is an important feature of multiplier ideal sheaves and has opened the door to new types of approximation technique (see e.g. \cite{GZSOC,McNeal and Varolin,K16,cao17,cdM17,FoW18,DEL18,ZZ2018,GZ20,ZZ2019,ZhouZhu20siu's,FoW20,KS20,DEL21}).
Guan-Zhou \cite{GZSOC} proved the strong openness property (the 2-dimensional case was proved by Jonsson-Musta\c{t}\u{a} \cite{JonssonMustata}).
After that, Guan-Zhou \cite{GZeff} established an effectiveness result of the strong openness property by considering the minimal $L^2$ integral on the pseudoconvex domain $D$.

Considering the minimal $L^2$ integrals on the sublevel sets of the weight $\varphi$, Guan \cite{G16} obtained a sharp version of Guan-Zhou's effectiveness result, and established a concavity property of the minimal $L^2$ integrals (see also 				\cite{G2018},\cite{GM},\cite{GM_Sci},\cite{GY-concavity},\cite{GMY}).

Recall that the minimal $L^2$ integrals on the sublevel sets of $\varphi$ in \cite{G16} (see also \cite{G2018},\cite{GM},\cite{GM_Sci},\cite{GY-concavity},\cite{GMY}) are related to the modules (ideals) at the inner points of the sublevel sets.
In \cite{BGY-boundary}, Bao-Guan-Yuan considered the minimal $L^2$ integrals related to modules at boundary points of the sublevel sets of plurisubharmonic functions on pseudoconvex domains,
and gave a concavity property of the minimal $L^2$ integrals,
which deduced a sharp effectiveness result related to a conjecture posed by Jonsson-Musta\c{t}\u{a} \cite{JonssonMustata},
and completed the approach from the conjecture to the strong openness property.

In the present article, we consider minimal $L^2$ integrals on sublevel sets of plurisubharmonic functions on weakly pseudoconvex K\"{a}hler manifolds with Lebesgue measurable gain related to  modules at boundary points of the sublevel sets,
and establish a concavity property of the minimal $L^2$ integrals.

\subsection{Main result}\label{sec:Main result}
Let $M$ be a complex manifold. Let $X$ and $Z$ be closed subsets of $M$. We call that a triple $(M,X,Z)$ satisfies condition $(A)$, if the following two statements hold:

$\uppercase\expandafter{\romannumeral1}.$ $X$ is a closed subset of $M$ and $X$ is locally negligible with respect to $L^2$ holomorphic functions; i.e., for any local coordinated neighborhood $U\subset M$ and for any $L^2$ holomorphic function $f$ on $U\backslash X$, there exists an $L^2$ holomorphic function $\tilde{f}$ on $U$ such that $\tilde{f}|_{U\backslash X}=f$ with the same $L^2$ norm;

$\uppercase\expandafter{\romannumeral2}.$ $Z$ is an analytic subset of $M$ and $M\backslash (X\cup Z)$ is a weakly pseudoconvex K\"ahler manifold.

Let $M$ be an $n-$dimensional complex manifold, and let $(M,X,Z)$  satisfy condition $(A)$. Let $K_M$ be the canonical line bundle on $M$. Let $F$ be a holomorphic function on $M$. We assume that $F$ is not identically zero. Let $\psi$ be a plurisubharmonic function on $M$. Let $\varphi_{\alpha}$ be a Lebesgue measurable function on $M$ such that $\varphi_{\alpha}+\psi$ is a plurisubharmonic function on $M$.

Let $T\in [-\infty,+\infty)$.
Denote that
$$\Psi:=\min\{\psi-2\log|F|,-T\}.$$
For any $z \in M$ satisfying $F(z)=0$,
we set $\Psi(z)=-T$. Note that for any $t\ge T$,
the holomorphic function $F$ has no zero points on the set $\{\Psi<-t\}$.
Hence $\Psi=\psi-2\log|F|=\psi+2\log|\frac{1}{F}|$ is a plurisubharmonic function on $\{\Psi<-t\}$.

\begin{Definition}
We call that a positive measurable function $c$ (so-called ``\textbf{gain}") on $(T,+\infty)$ is in class $P_{T,M,\Psi}$ if the following two statements hold:
\par
$(1)$ $c(t)e^{-t}$ is decreasing with respect to $t$;
\par
$(2)$ There exist $T_1> T$ and a closed subset $E$ of $M$ such that $E\subset Z\cap \{\Psi(z)=-\infty\}$ and for any compact subset $K\subset M\backslash E$, $e^{-\varphi_\alpha}c(-\Psi_1)$ has a positive lower bound on $K$, where $\Psi_1:=\min\{\psi-2\log|F|,-T_1\}$.
\end{Definition}

Let $z_0$ be a point in $M$. Denote that $\tilde{J}(\Psi)_{z_0}:=\{f\in\mathcal{O}(\{\Psi<-t\}\cap V): t\in \mathbb{R}$ and $V$ is a neighborhood of $z_0\}$. We define an equivalence relation $\backsim$ on $\tilde{J}(\Psi)_{z_0}$ as follows: for any $f,g\in \tilde{J}(\Psi)_{z_0}$,
we call $f \backsim g$ if $f=g$ holds on $\{\Psi<-t\}\cap V$ for some $t\gg T$ and open neighborhood $V\ni o$.
Denote $\tilde{J}(\Psi)_{z_0}/\backsim$ by $J(\Psi)_{z_0}$, and denote the equivalence class including $f\in \tilde{J}(\Psi)_{z_0}$ by $f_{z_0}$.

If $z_0\in \cap_{t>T} \{\Psi<-t\}$, then $J(\Psi)_{z_0}=\mathcal{O}_{M,z_0}$ (the stalk of the sheaf $\mathcal{O}_{M}$ at $z_0$), and $f_{z_0}$ is the germ $(f,z_0)$ of holomorphic function $f$.

Let $f_{z_0},g_{z_0}\in J(\Psi)_{z_0}$ and $(h,z_0)\in \mathcal{O}_{M,z_0}$. We define $f_{z_0}+g_{z_0}:=(f+g)_{z_0}$ and $(h,z_0)\cdot f_{z_0}:=(hf)_{z_0}$.
Note that $(f+g)_{z_0}$ and $(hf)_{z_0}$ ($\in J(\Psi)_{z_0}$) are independent of the choices of the representatives of $f,g$ and $h$. Hence $J(\Psi)_{z_0}$ is an $\mathcal{O}_{M,z_0}$-module.

For $f_{z_0}\in J(\Psi)_{z_0}$ and $a,b\ge 0$, we call $f_{z_0}\in I\big(a\Psi+b\varphi_\alpha\big)_{z_0}$ if there exist $t\gg T$ and a neighborhood $V$ of $z_0$,
such that $\int_{\{\Psi<-t\}\cap V}|f|^2e^{-a\Psi-b\varphi_\alpha}<+\infty$.
Note that $I\big(a\Psi+b\varphi_\alpha\big)_{z_0}$ is an $\mathcal{O}_{M,z_0}$-submodule of $J(\Psi)_{z_0}$.
If $z_0\in \cap_{t>T} \{\Psi<-t\}$, then $I_{z_0}=\mathcal{O}_{M,z_0}$, where $I_{z_0}:=I\big(0\Psi+0\varphi_\alpha\big)_{z_0}$.

Let $Z_0$ be a subset of $M$. Let $f$ be a holomorphic $(n,0)$ form on $\{\Psi<-t_0\}\cap V$, where $V\supset Z_0$ is an open subset of $M$ and $t_0>T$
is a real number.
Let $J_{z_0}$ be an $\mathcal{O}_{M,z_0}$-submodule of $J(\Psi)_{z_0}$ such that $I\big(\Psi+\varphi_\alpha\big)_{z_0}\subset J_{z_0}$,
where $z_0\in Z_0$.
Denote
\begin{equation}
\label{def of g(t) for boundary pt}
\begin{split}
\inf\Bigg\{ \int_{ \{ \Psi<-t\}}|\tilde{f}|^2e^{-\varphi_\alpha}c(-\Psi): \tilde{f}\in
H^0(\{\Psi<-t\},\mathcal{O} (K_M)  ) \\
\&\, (\tilde{f}-f)_{z_0}\in
\mathcal{O} (K_M)_{z_0} \otimes J_{z_0},\text{for any }  z_0\in Z_0 \Bigg\}
\end{split}
\end{equation}
by $G(t;c,\Psi,\varphi_\alpha,J,f)$, where $t\in[T,+\infty)$, $c$ is a nonnegative function on $(T,+\infty)$ and $|f|^2:=\sqrt{-1}^{n^2}f\wedge \bar{f}$ for any $(n,0)$ form $f$.
Without misunderstanding, we simply denote $G(t;c,\Psi,\varphi_\alpha,J,f)$ by $G(t)$. For various $c$, $\psi$ and $\varphi_\alpha$, we simply denote $G(t;c,\Psi,\varphi_\alpha,J,f)$ by $G(t;c)$, $G(t;\Psi)$ and $G(t;\varphi_\alpha)$ respectively.

In this article, we obtain the following concavity property of $G(t)$.

\begin{Theorem}
\label{main theorem}
Let $c\in\mathcal{P}_{T,M,\Psi}$. If there exists $t \in [T,+\infty)$ satisfying that $G(t)<+\infty$, then $G(h^{-1}(r))$ is concave with respect to  $r\in (\int_{T_1}^{T}c(t)e^{-t}dt,\int_{T_1}^{+\infty}c(t)e^{-t}dt)$, $\lim\limits_{t\to T+0}G(t)=G(T)$ and $\lim\limits_{t \to +\infty}G(t)=0$, where $h(t)=\int_{T_1}^{t}c(t_1)e^{-t_1}dt_1$ and $T_1 \in (T,+\infty)$.
\end{Theorem}

When $M$ is a pseudoconvex domain $D$ in $\mathbb{C}^n$, $\varphi_\alpha\equiv0$, $c(t)\equiv 1$ and $T=0$, Theorem \ref{main theorem} degenerates to the concavity property in \cite{BGY-boundary}.

\begin{Remark}
	\label{infty2}Let $c\in\mathcal{P}_{T,M,\Psi}$.	If  $\int_{T_1}^{+\infty}c(t)e^{-t}dt=+\infty$ and $f_{z_0}\notin
\mathcal{O} (K_M)_{z_0} \otimes J_{z_0}$ for some  $ z_0\in Z_0$, then $G(t)=+\infty$ for any $t\geq T$. Thus, when there exists $t \in [T,+\infty)$ satisfying that $G(t)\in(0,+\infty)$, we have $\int_{T_1}^{+\infty}c(t)e^{-t}dt<+\infty$ and $G(\hat{h}^{-1}(r))$ is concave with respect to  $r\in (0,\int_{T}^{+\infty}c(t)e^{-t}dt)$, where $\hat{h}(t)=\int_{t}^{+\infty}c(l)e^{-l}dl$.
\end{Remark}

Let $c(t)$ be a nonnegative measurable function on $(T,+\infty)$. Set
\begin{equation}\nonumber
\begin{split}
\mathcal{H}^2(c,t)=\Bigg\{\tilde{f}:\int_{ \{ \Psi<-t\}}|\tilde{f}|^2e^{-\varphi_\alpha}c(-\Psi)<+\infty,\  \tilde{f}\in
H^0(\{\Psi<-t\},\mathcal{O} (K_M)  ) \\
\& (\tilde{f}-f)_{z_0}\in
\mathcal{O} (K_M)_{z_0} \otimes J_{z_0},\text{for any }  z_0\in Z_0  )\Bigg\},
\end{split}
\end{equation}
where $t\in[T,+\infty)$.

As a corollary of Theorem \ref{main theorem}, we give a necessary condition for the concavity property degenerating to linearity.
\begin{Corollary}
\label{necessary condition for linear of G}
Let $c\in\mathcal{P}_{T,M,\Psi}$.
Assume that $G(t)\in(0,+\infty)$ for some $t\ge T$, and $G(\hat{h}^{-1}(r))$ is linear with respect to $r\in[0,\int_T^{+\infty}c(s)e^{-s}ds)$, where $\hat{h}(t)=\int_{t}^{+\infty}c(l)e^{-l}dl$.

Then there exists a unique holomorphic $(n,0)$ form $\tilde{F}$ on $\{\Psi<-T\}$
such that $(\tilde{F}-f)_{z_0}\in\mathcal{O} (K_M)_{z_0} \otimes J_{z_0}$ holds for any  $z_0\in Z_0$,
and $G(t)=\int_{\{\Psi<-t\}}|\tilde{F}|^2e^{-\varphi_{\alpha}}c(-\Psi)$ holds for any $t\ge T$.

Furthermore
\begin{equation}
\begin{split}
  \int_{\{-t_1\le\Psi<-t_2\}}|\tilde{F}|^2e^{-\varphi_\alpha}a(-\Psi)=\frac{G(T_1;c)}{\int_{T_1}^{+\infty}c(t)e^{-t}dt}
  \int_{t_2}^{t_1}a(t)e^{-t}dt
  \label{other a also linear}
\end{split}
\end{equation}
holds for any nonnegative measurable function $a$ on $(T,+\infty)$, where $T\le t_2<t_1\le+\infty$ and $T_1 \in (T,+\infty)$.
\end{Corollary}

\begin{Remark}
\label{rem:linear}
If $\mathcal{H}^2(\tilde{c},t_0)\subset\mathcal{H}^2(c,t_0)$ for some $t_0\ge T$, we have
\begin{equation}
\begin{split}
  G(t_0;\tilde{c})=\int_{\{\Psi<-t_0\}}|\tilde{F}|^2e^{-\varphi_\alpha}\tilde{c}(-\Psi)=
  \frac{G(T_1;c)}{\int_{T_1}^{+\infty}c(t)e^{-t}dt}
  \int_{t_0}^{+\infty}\tilde{c}(s)e^{-s}ds,
  \label{other c also linear}
\end{split}
\end{equation}
 where $\tilde{c}$ is a nonnegative measurable function on $(T,+\infty)$ and $T_1 \in (T,+\infty)$.
\end{Remark}

\subsection{Applications}

In this section, we give some applications of Theorem \ref{main theorem}.

\subsubsection{Concavity property at inner points}

Let $M$ be a complex manifold. Let $X$ and $Z$ be closed subsets of $M$. Assume that the triple $(M,X,Z)$ satisfies condition $(A)$. Let $\psi$ be a plurisubharmonic function on $M$. Let $\varphi$ be a Lebesgue measurable function on $M$, such that $\psi+\varphi$ is a plurisubharmonic function on $M$. Denote $T=-\sup\limits_M \psi$.

\begin{Definition}
We call a positive measurable function $c$ (so-called ``\textbf{gain}") on $(T,+\infty)$ in class $P_{T,M}$ if the following two statements hold:
\par
$(1)$ $c(t)e^{-t}$ is decreasing with respect to $t$;
\par
$(2)$ there is a closed subset $E$ of $M$ such that $E\subset Z\cap \{\psi(z)=-\infty\}$ and for any compact subset $K\subset M\backslash E$, $e^{-\varphi}c(-\psi)$ has a positive lower bound on $K$.
\end{Definition}

Let $Z_0$ be a subset of $\{\psi=-\infty\}$ such that $Z_0 \cap
Supp(\mathcal{O}/\mathcal{I}(\varphi+\psi))\neq \emptyset$. Let $U \supset Z_0$ be
an open subset of $M$, and let $f$ be a holomorphic $(n,0)$ form on $U$. Let $\mathcal{F}_{z_0} \supset \mathcal{I}(\varphi+\psi)_{z_0}$ be an ideal of $\mathcal{O}_{z_0}$ for any $z_{0}\in Z_0$.

Denote
\begin{equation}\label{def of g(t) in concavity II}
\begin{split}
\inf\Bigg\{\int_{ \{ \psi<-t\}}|\tilde{f}|^2e^{-\varphi}c(-\psi): \tilde{f}\in
H^0(\{\psi<-t\},\mathcal{O} (K_M)  ) \\
\&\, (\tilde{f}-f)\in
H^0(Z_0 ,(\mathcal{O} (K_M) \otimes \mathcal{F})|_{Z_0})\Bigg\}
\end{split}
\end{equation}
by $G(t;c,\psi,\varphi,f,\mathcal{F})$, where $t\in[T,+\infty)$, $c$ is a nonnegative function on $(T,+\infty)$, $|f|^2:=\sqrt{-1}^{n^2}f\wedge \bar{f}$ for any $(n,0)$ form $f$,
and $(\tilde{f}-f)\in
H^0(Z_0 ,(\mathcal{O} (K_M) \otimes \mathcal{F})|_{Z_0} )$ means $(\tilde{f}-f,z_0)\in(\mathcal{O}(K_M)\otimes \mathcal{F})_{z_0}$ for all $z_0\in Z_0$.

Using Theorem \ref{main theorem}, we obtain the following concavity property of $G(t;c,\psi,\varphi,f,\mathcal{F})$.
\begin{Corollary}
\label{generalization for concavity II}

Assume that $c\in\mathcal{P}_{T,M}$. If there exists $t_0 \ge T$ satisfying that $G(t_0;c,\psi,\varphi,f,\mathcal{F})<+\infty$, then $G(h^{-1}(r);c,\psi,\varphi,f,\mathcal{F})$ is concave with respect to  $r\in (\int_{T_1}^{T}c(t)e^{-t}dt,\int_{T_1}^{+\infty}c(t)e^{-t}dt)$, $\lim\limits_{t\to T+0}G(t;c,\psi,\varphi,f,\mathcal{F})=G(T;c,\psi,\varphi,f,\mathcal{F})$ and \\ $\lim\limits_{t \to +\infty}G(t;c,\psi,\varphi,f,\mathcal{F})=0$, where $h(t)=\int_{T_1}^{t}c(t_1)e^{-t_1}dt_1$ and $T_1 \in (T,+\infty)$.
\end{Corollary}

It follows from Corollary \ref{necessary condition for linear of G} and Corollary \ref{generalization for concavity II} that we have the following remark.

\begin{Remark}
\label{necessary condition for linear of concavity II}
Let $c\in\mathcal{P}_{T,M}$.
Assume that $G(t;c,\psi,\varphi,f,\mathcal{F})\in(0,+\infty)$ for some $t\ge T$, and $G(\hat{h}^{-1}(r);c,\psi,\varphi,f,\mathcal{F})$ is linear with respect to $r\in[0,\int_T^{+\infty}c(s)e^{-s}ds)$, where $\hat{h}(t)=\int_{t}^{+\infty}c(l)e^{-l}dl$.

Then there exists a unique holomorphic $(n,0)$ form $\tilde{F}$ on $M$
such that $(\tilde{F}-f)\in
H^0(Z_0 ,(\mathcal{O} (K_M) \otimes \mathcal{F})|_{Z_0})$
and $G(t;c,\psi,\varphi,f,\mathcal{F})=\int_{\{\psi<-t\}}|\tilde{F}|^2e^{-\varphi}c(-\psi)$ holds for any $t\ge T$.

Furthermore
\begin{equation}\nonumber
\begin{split}
  \int_{\{-t_1\le\psi<-t_2\}}|\tilde{F}|^2e^{-\varphi}a(-\psi)=\frac{G(T_1;c,\psi,\varphi,f,\mathcal{F})}{\int_{T_1}^{+\infty}c(t)e^{-t}dt}
  \int_{t_2}^{t_1}a(t)e^{-t}dt
\end{split}
\end{equation}
holds for any nonnegative measurable function $a$ on $(T,+\infty)$, where $T\le t_2<t_1\le+\infty$ and $T_1 \in (T,+\infty)$.
\end{Remark}

When $\psi$ is a plurisubharmonic function on $M$, and $\{\psi<-t\}\backslash (X\cup Z)$ is a weakly pseudoconvex K\"ahler manifold for any $t\in\mathbb{R}$, Corollary \ref{generalization for concavity II} and Remark \ref{necessary condition for linear of concavity II} can be referred to \cite{GMY} (see also \cite{GM_Sci} and \cite{BGMY7}).

\subsubsection{An optimal support function related to $I(a\Psi)$}
Let $M$ be an $n-$dimensional complex manifold. Let $X$ and $Z$ be closed subsets of $M$, and let $(M,X,Z)$  satisfy condition $(A)$. Let $K_M$ be the canonical line bundle on $M$. Let $F$ be a holomorphic function on $M$. Assume that $F$ is not identically zero. Let $\psi$ be a plurisubharmonic function on $M$.

 Denote that
$$\Psi:=\min\{\psi-2\log|F|,0\}.$$
For any $z \in M$ satisfying $F(z)=0$, we set $\Psi(z)=0$.
Denote that $M_t:=\{z\in M:-t\le\Psi(z)<0\}$.
Let $Z_0$ be a subset of $M$, and let  $f$ be a holomorphic $(n,0)$ form on $\{\Psi<0\}$.
Denote
\begin{displaymath}\begin{split}
	\inf\bigg\{\int_{M_t}|\tilde f|^2:&f\in H^0(\{\Psi<0\},\mathcal{O}(K_M))\\
	&\&\,(\tilde f-f)_{z_0}\in \mathcal{O}(K_M)\otimes I(\Psi)_{z_0}\,\text{for any }  z_0\in Z_0\bigg\}	
\end{split}\end{displaymath}
by $C_{f,\Psi,t}(Z_0)$ for any $t\ge0$. When $C_{f,\Psi,t}(Z_0)=0$ or $+\infty$, we set $\frac{\int_{M_t}|f|^2e^{-\Psi}}{C_{f,\Psi,t}(Z_0)}=+\infty$.

We obtain the following optimal support function of $\frac{\int_{M_t}|f|^2e^{-\Psi}}{C_{f,\Psi,t}(Z_0)}$.
\begin{Proposition}
	\label{p:1}
	Assume that $\int_{\{\Psi<-l\}}|f|^2<+\infty$ holds for any $l>0$. Then the inequality
	\begin{equation}
		\label{eq:0304a}
		\frac{\int_{M_t}|f|^2e^{-\Psi}}{C_{f,\Psi,t}(Z_0)}\ge\frac{t}{1-e^{-t}}
	\end{equation}
	holds for any $t\ge0$, where $\frac{t}{1-e^{-t}}$ is the optimal support function.
\end{Proposition}
 Proposition \ref{p:1} can be referred to \cite{GY-support} when $M$ is a pseudoconvex domain in $\mathbb{C}^n$, $F\equiv1$, $Z_0=\{z_0\}\subset M$  and $\psi(z_0)=-\infty$.

Take $M=\Delta\subset\mathbb{C}$, $Z_{0}=o$ the origin of $\mathbb{C}$, $F\equiv1$, $\psi=2\log{|z|}$ and $f\equiv dz$.
It is clear that $\int_M|f|^2<+\infty$. Note that $C_{f,\Psi,t}(Z_0)=2\pi(1-e^{-t})$ and $\int_{M_t}|f|^2e^{-\Psi}=2t\pi$.
Then $\frac{\int_{M_t}|f|^2e^{-\Psi}}{C_{f,\Psi,t}(Z_0)}=\frac{t}{1-e^{-t}}$, which shows the optimality of the support function $\frac{t}{1-e^{-t}}$.

\subsubsection{Strong openness property of $I(a\Psi+\varphi)_o$ and a twisted version}

Let $D\subseteq\mathbb{C}^n$ be a pseudoconvex domain containing the origin $o$, and let $\psi$ be a plurisubharmonic function on $D$. Let $F\not\equiv0$ be a holomorphic function on $D$, and let $\varphi$ be a plurisubharmonic function on $D$.
Denote that
$$\Psi:=\min\{\psi-2\log|F|,0\}.$$
For any $z \in M$ satisfying $F(z)=0$, we set $\Psi(z)=0$.

Recall that $f_{o}\in I(a\Psi+b\varphi)_{o}$ if and only if there exist $t\gg 0$ and a neighborhood $V$ of $o$ such that $\int_{\{\Psi<-t\}\cap V}|h|^2e^{-a\Psi-b\varphi}<+\infty$, where $a\ge0$ and $b\ge0$. Denote that
$$I_+(a\Psi+b\varphi)_o:=\cup_{s>a}I(s\Psi+b\varphi)_o$$
 for any $b\ge0$. In \cite{BGY-boundary}, Bao-Guan-Yuan gave  the strong openness property of $I(a\Psi)_o$.

\begin{Theorem}
	[\cite{BGY-boundary}]\label{p:soc}
	$I(a\Psi)_o=I_+(a\Psi)_o$ holds for any $a\ge0$.
\end{Theorem}

 When $F\equiv1$ and $\psi(o)=-\infty$, Theorem \ref{p:soc} is the strong openness property of multiplier ideal sheaves \cite{GZSOC}, i.e. $\mathcal I(\psi)_o=\mathcal I_+(\psi)_o:=\cup_{s>1}\mathcal{I}(s\psi)_o$. We present the following strong openness property of $I(a\Psi+\varphi)_o$.

\begin{Theorem}\label{thm:soc}
	$I(a\Psi+\varphi)_o=I_+(a\Psi+\varphi)_o$ holds for any $a\ge0$.
\end{Theorem}

Let $p$ be a positive real number. Denote that
\begin{displaymath}
	c_{o,p}^h(\varphi):=\sup\{c\geq0: |h|^pe^{-2c\varphi}\,\text{is}\,L^1\,\,\text{on a neighborhood of}\,\, o\}.
\end{displaymath}
When $p=2$,  $c_{o,2}^h(\psi)$ is the jumping number $c_{o}^h(\psi)$ (see \cite{JM13}). Note that the strong openness property of multiplier ideal sheaves is equivalent to $|h|^2e^{-2c_o^h(\psi)\psi}$ is not integrable near $o$ for any holomorphic function $h$ on a neighborhood of $o$.

Using the strong openness property of multiplier ideal sheaves, Forn{\ae}ss \cite{Fo15} obtained the following strong openness property of multiplier ideal sheaves in $L^p$:

\emph{$|h|^pe^{-2c_{o,p}^h(\psi)\psi}$ is not integrable near $o$ for any holomorphic function $h$ on a neighborhood of $o$.}

In \cite{GY-twisted}, Guan-Yuan gave the following twisted version of strong openness property of multiplier ideal sheaves in $L^p$ (some related results can be referred to \cite{GZ-soc17} and \cite{chen18}).
\begin{Theorem}[\cite{GY-twisted}]
	\label{thm:tsoc}Let $p\in(0,+\infty)$, and let $a(t)$ be a positive measurable function on $(-\infty,+\infty)$. If  one of the following conditions holds:
	
	$(1)$ $a(t)$ is decreasing  near $+\infty$;
	
	$(2)$ $a(t)e^t$ is increasing near $+\infty$,
	
	 then the following three statements are equivalent:
	
	$(A')$ $a(t)$ is not integrable near $+\infty$;
	
	$(B')$ $a(-2c_{o,p}^{h}(\psi)\psi)\exp(-2c_{o,p}^{h}(\psi)\psi+p\log|h|)$ is not integrable near $o$ for any $\psi$ and $h$ satisfying $c_{o,p}^{h}(\psi)<+\infty$;
	
	$(C')$ $a(-2c_{o,p}^{h}(\psi)\psi+p\log|h|)\exp(-2c_{o,p}^{h}(\psi)\psi+p\log|h|)$ is not integrable near $o$ for any $\psi$ and $h$ satisfying $c_{o,p}^{h}(\psi)<+\infty$.	
\end{Theorem}

  Denote that
 $$a_o^f(\Psi;\varphi):=\sup\{a\ge0:f_o\in I(2a\Psi+\varphi)_o\}$$
for any $f_o\in I(\varphi)_o$.
Especially, $a_{o}^f(\Psi;\varphi)$ is the jumping number $c_o^f(\psi)$ (see \cite{JM13}) when $F\equiv1$, $\psi(o)=-\infty$ and $\varphi\equiv0$.

Note that the strong openness property of $I(a\Psi+\varphi)_o$ is equivalent to that $f_o\not\in I(2a_{o}^f(\Psi)\Psi+\varphi)_o$ for any $f_o\in I(\varphi)_o$ satisfying $a_{o}^f(\Psi;\varphi)<+\infty$. We present a twisted version of the strong openness property of $I(a\Psi+\varphi)_o$.

\begin{Theorem}
	\label{p:soc-twist}
Let $a(t)$ be a positive measurable function on $(-\infty,+\infty)$. If  one of the following conditions holds:
	
	$(1)$ $a(t)$ is decreasing  near $+\infty$;
	
	$(2)$ $a(t)e^t$ is increasing near $+\infty$,
	
	 then the following two statements are equivalent:
	
	$(A)$ $a(t)$ is not integrable near $+\infty$;
	
	$(B)$ for any $\Psi$, $\varphi$ and $f_o\in I(\varphi)_o$ satisfying $a_{o}^{f}(\Psi;\varphi)<+\infty$, we have
	$$a(-2a_{o}^{f}(\Psi;\varphi)\Psi)\exp(-2a_{o}^{f}(\Psi;\varphi)\Psi-\varphi+2\log|f|)\not\in L^1(U\cap\{\Psi<-t\}),$$ where $U$ is any neighborhood of $o$ and $t>0$.	
\end{Theorem}

\begin{Remark}
	\label{r:tsoc}Theorem \ref{p:soc-twist} is a generalization of Theorem \ref{thm:tsoc}. We prove this remark in Section \ref{sec:tsoc}.
\end{Remark}

\subsubsection{Effectiveness of the strong openness property of $I(a\Psi+\varphi)_{o}$}

Let $D$ be a pseudoconvex domain in $\mathbb{C}^n$ containing the origin $o$, and let $\psi$ be a plurisubharmonic function on $D$. Let $F\not\equiv0$ be a holomorphic function on $D$, and let $\varphi$ be a plurisubharmonic function on $D$.
Denote that
$$\Psi:=\min\{\psi-2\log|F|,0\}.$$
If $F(z)=0$ for some $z \in M$, we set $\Psi(z)=0$. Let $f$ be a holomorphic function on $\{\Psi<0\}$. Denote that
\begin{displaymath}
	\begin{split}
		\frac{1}{K_{\Psi,f,a}(o)}:=\inf\bigg\{\int_{\{\Psi<0\}}|\tilde f|^2e^{-\varphi-(1-a)\Psi}:(\tilde f-f)_{o}\in I_+(\varphi+&2a_{o}^{f}(\Psi;\varphi)\Psi)_{o}\\
		& \& \,\tilde f\in\mathcal{O}(\{\Psi<0\})\bigg\},
	\end{split}
\end{displaymath}
where $a\in(0,+\infty)$.

We present an effectiveness result of the strong openness property of $I(a\Psi+\varphi)_{z_0}$.
\begin{Theorem}
	\label{thm:effe}
	Let $C_1$ and $C_2$ be two positive constants. If there exists $a>0$, such that
	
	$(1)$ $\int_{\{\Psi<0\}}|f|^2e^{-\varphi-\Psi}\leq C_1$;
	
	$(2)$ $\frac{1}{K_{\Psi,f,a}(o)}\geq C_2$.
	
	Then for any $q>1$ satisfying $$\theta_a(q)>\frac{C_1}{C_2},$$ we have
$f_o\in I(q\Psi+\varphi)_o$, where $\theta_a(q)=\frac{q+a-1}{q-1}$.	\end{Theorem}

When $F\equiv1$ and $\psi(o)=-\infty$, Theorem \ref{thm:effe} degenerates to the effectiveness result of strong openness property of multiplier ideal sheaves in $L^p$ \cite{GY-lp-effe} (some related results can be referred to \cite{GZeff,G16}).

\section{Some preparations}
In this section, we do some preparations.
\subsection{$L^2$ method}
Let $M$ be an $n-$dimensional weakly pseudoconvex K\"ahler manifolds. Let $\psi$ be a plurisubharmonic function on $M$. Let $F$ be a holomorphic function on $M$. We assume that $F$ is not identically zero. Let $\varphi_{\alpha}$ be a Lebesgue measurable function on $M$ such that $\varphi_{\alpha}+\psi$ is a plurisubharmonic function on $M$.

Let $\delta$ be a positive integer. Let $T$ be a real number. Denote
$$\varphi:=\varphi_{\alpha}+(1+\delta)\max\{\psi+T,2\log|F|\}$$
and
$$\Psi:=\min\{\psi-2\log|F|,-T\}.$$
If $F(z)=0$ for some $z \in M$, we set $\Psi(z)=-T$.

Let $c(t)$ be a positive measurable function on $[T,+\infty)$ such that $c(t)e^{-t}$ is decreasing with respect $t$. We have the following lemma.

\begin{Lemma}
\label{L2 method}
Let $B\in(0,+\infty)$ and $t_0>T$ be arbitrarily given. Let $f$ be a holomorphic $(n,0)$ form on $\{\Psi<-t_0\}$ such that
$$\int_{\{\Psi<-t_0\}\cap K}|f|^2<+\infty,$$
for any compact subset $K\subset M$, and
$$\int_{M}\frac{1}{B}\mathbb{I}_{\{-t_0-B<\Psi<-t_0\}}|f|^2e^{-\varphi_{\alpha}-\Psi}<+\infty.$$
 Then there exists a holomorphic $(n,0)$ form $\tilde{F}$ on $M$ such that
\begin{equation*}
  \begin{split}
      & \int_{M}|\tilde{F}-(1-b_{t_0,B}(\Psi))fF^{1+\delta}|^2e^{-\varphi+v_{t_0,B}(\Psi)-\Psi}c(-v_{t_0,B}(\Psi)) \\
      \le & (\frac{1}{\delta}c(T)e^{-T}+\int_{T}^{t_0+B}c(s)e^{-s}ds)
       \int_{M}\frac{1}{B}\mathbb{I}_{\{-t_0-B<\Psi<-t_0\}}|f|^2e^{-\varphi_{\alpha}-\Psi},
  \end{split}
\end{equation*}
where $b_{t_0,B}(t)=\int^{t}_{-\infty}\frac{1}{B} \mathbb{I}_{\{-t_0-B< s < -t_0\}}ds$,
$v_{t_0,B}(t)=\int^{t}_{-t_0}b_{t_0,B}(s)ds-t_0$.
\end{Lemma}

We introduce the following property of gain function $c(t)$.
Let $T_0\in[-\infty,+\infty)$. Denote
$$\Psi_0:=\min\{\psi-2\log|F|,-T_0\}.$$
If $F(z)=0$ for some $z \in M$, we set $\Psi(z)=-T_0$. Let $c(t)\in \mathcal{P}_{T_0,M,\Psi_0}$, i.e.,
\par
$(1)$ $c(t)e^{-t}$ is decreasing with respect to $t$;
\par
$(2)$ There exist $T_1> T_0$ and a closed subset $E$ of $M$ such that $E\subset Z\cap \{\Psi(z)=-\infty\}$ and for any compact subset $K\subset M\backslash E$, $e^{-\varphi_\alpha}c(-\Psi_1)$ has a positive lower bound on $K$, where $\Psi_1:=\min\{\psi-2\log|F|,-T_1\}$.
\begin{Proposition}
\label{property of c} Let $c(t)\in \mathcal{P}_{T_0,M,\Psi_0}$. Let $T$ be a real number such that $T>T_0$.
Then for any compact subset $K\subset M\backslash E$, we have $e^{-\varphi_\alpha}c(-\Psi)$ has a positive lower bound on $K$,
where $\Psi:=\min\{\psi-2\log|F|,-T\}$.
\end{Proposition}
\begin{proof}

We note that $c(t)$ is positive on $[T_1,T]$(or $[T,T_1]$) and has positive lower bound and upper bound on $[T_1,T]$(or $[T,T_1]$). Hence $\frac{c(-\Psi)}{c(-\Psi_1)}$ has a positive lower bound on $M$. Then we know that for any compact subset $K\subset M\backslash E$, $e^{-\varphi_\alpha}c(-\Psi)$ has a positive lower bound on $K$.
\end{proof}

Let $T_0\in[-\infty,+\infty)$. Let $c(t)\in \mathcal{P}_{T_0,M,\Psi_0}$.
It follows from Lemma \ref{L2 method} that we have the following lemma.

\begin{Lemma}
Let $(M,X,Z)$  satisfies condition $(A)$. Let $B \in (0, +\infty)$ and $t_0> T>T_0$ be arbitrarily given.
Let $f$ be a holomorphic $(n,0)$ form on $\{\Psi< -t_0\}$ such that
\begin{equation}
\int_{\{\Psi<-t_0\}} {|f|}^2e^{-\varphi_\alpha}c(-\Psi)<+\infty,
\label{condition of lemma 2.2}
\end{equation}
Then there exists a holomorphic $(n,0)$ form $\tilde{F}$ on $M$  such that
\begin{equation}
\begin{split}
&\int_{M}|\tilde{F}-(1-b_{t_0,B}(\Psi))fF^{1+\delta}|^2e^{-\varphi-\Psi+v_{t_0,B}(\Psi)}c(-v_{t_0,B}(\Psi))\\
\le & \left(\frac{1}{\delta}c(T)e^{-T}+\int_{T}^{t_0+B}c(t)e^{-t}dt\right)\int_M \frac{1}{B} \mathbb{I}_{\{-t_0-B< \Psi < -t_0\}}  {|f|}^2
e^{{-}\varphi_{\alpha}-\Psi},
\end{split}
\end{equation}
where  $b_{t_0,B}(t)=\int^{t}_{-\infty}\frac{1}{B} \mathbb{I}_{\{-t_0-B< s < -t_0\}}ds$ and
$v_{t_0,B}(t)=\int^{t}_{-t_0}b_{t_0,B}(s)ds-t_0$.
\label{L2 method for c(t)}
\end{Lemma}
\begin{proof} It follows from inequality \eqref{condition of lemma 2.2} and $c(t)e^{-t}$ is decreasing with respect to $t$ that
$$\int_M \frac{1}{B} \mathbb{I}_{\{-t_0-B< \Psi < -t_0\}}  {|f|}^2
e^{{-}\varphi_{\alpha}-\Psi}<+\infty.$$

 Note that $c(t)\in \mathcal{P}_{T_0,M,\Psi_0}$ and there exists a closed subset $E\subset Z\cap \{\Psi=-\infty\}$ such that  $e^{-\varphi}c(-\Psi)$ has locally positive lower bound on $M\backslash E$. For any compact subset $K\subset M\backslash E$, we obtain from inequality \eqref{condition of lemma 2.2} that
\begin{equation}\nonumber
\int_{K\cap \{\Psi<-t_0\}}|f|^2<+\infty.
\end{equation}

As $(M,X,Z)$  satisfies condition $(A)$,  $M\backslash (Z\cup X)$ is a weakly pseudoconvex K\"ahler manifold. It follows from Lemma \ref{L2 method} that there exists a holomorphic $(n,0)$ form $\tilde{F}_Z$ on $M\backslash (Z\cup X)$ such that
\begin{equation*}
  \begin{split}
      & \int_{M\backslash (Z\cup X)}|\tilde{F}_Z-(1-b_{t_0,B}(\Psi))fF^{1+\delta}|^2e^{-\varphi+v_{t_0,B}(\Psi)-\Psi}c(-v_{t_0,B}(\Psi)) \\
      \le & (\frac{1}{\delta}c(T)e^{-T}+\int_{T}^{t_0+B}c(s)e^{-s}ds)
       \int_{M}\frac{1}{B}\mathbb{I}_{\{-t_0-B<\Psi<-t_0\}}|f|^2e^{-\varphi_{\alpha}-\Psi}.
  \end{split}
\end{equation*}

For any $z\in \left((Z\cup X)\backslash E\right)$, there exists an open neighborhood $V_z$ of $z$ such that $V_z\Subset M\backslash E$.

Note that $c(t)e^{-t}$ is decreasing with respect to $t$ and $v_{t_0,B}(\Psi)\ge \Psi$, we have

\begin{equation*}
  \begin{split}
   &\int_{M\backslash (Z\cup X)}|\tilde{F}_Z-(1-b_{t_0,B}(\Psi))fF^{1+\delta}|^2e^{-\varphi}c(-\Psi)\\
   \le
      & \int_{M\backslash (Z\cup X)}|\tilde{F}_Z-(1-b_{t_0,B}(\Psi))fF^{1+\delta}|^2e^{-\varphi+v_{t_0,B}(\Psi)-\Psi}c(-v_{t_0,B}(\Psi))<+\infty.
  \end{split}
\end{equation*}

Note that $\varphi=\varphi_{\alpha}+(1+\delta)2\log|F|$ on $\{\Psi<-t_0\}$. We have
$$\int_{\{\Psi<-t_0\}}|fF^{1+\delta}|^2e^{-\varphi}c(-\Psi)=
\int_{\{\Psi<-t_0\}}|f|^2e^{-\varphi_\alpha}c(-\Psi)<+\infty.$$

Note that there exists a positive number $C>0$ such that $c(-\Psi)e^{-\varphi}>C$ on $V_z$. Then we have
\begin{equation*}
  \begin{split}
  &\int_{V_z\backslash (Z\cup X)}|\tilde{F}_Z|^2\\
  \le &
  \frac{1}{C} \int_{V_z\backslash (Z\cup X)}|\tilde{F}_Z|^2c(-\Psi)e^{-\varphi}\\
  \le &  \frac{2}{C}\big(\int_{V_z\backslash (Z\cup X)}|\tilde{F}_Z-(1-b_{t_0,B}(\Psi))fF^{1+\delta}|^2e^{-\varphi}c(-\Psi)\\
  &+\int_{V_z\backslash (Z\cup X)}|(1-b_{t_0,B}(\Psi))fF^{1+\delta}|^2e^{-\varphi}c(-\Psi)
  \big)\\
  \le &
  \frac{2}{C}\big(\int_{M\backslash (Z\cup X)}|\tilde{F}_Z-(1-b_{t_0,B}(\Psi))fF^{1+\delta}|^2e^{-\varphi+v_{t_0,B}(\Psi)-\Psi}c(-v_{t_0,B}(\Psi))\\
  &+\int_{\{\Psi<-t_0\}}|fF^{1+\delta}|^2e^{-\varphi}c(-\Psi)
  \big)<+\infty.
  \end{split}
\end{equation*}
As $Z\cup X$ is locally negligible with respect to $L^2$ holomorphic function, we can find a holomorphic extension $\tilde{F}_E$ of $\tilde{F}_Z$ from $M\backslash (Z\cup X)$ to $M\backslash E$ such that

\begin{equation*}
  \begin{split}
      & \int_{M\backslash E}|\tilde{F}_E-(1-b_{t_0,B}(\Psi))fF^{1+\delta}|^2e^{-\varphi+v_{t_0,B}(\Psi)-\Psi}c(-v_{t_0,B}(\Psi)) \\
      \le & (\frac{1}{\delta}c(T)e^{-T}+\int_{T}^{t_0+B}c(s)e^{-s}ds)
       \int_{M}\frac{1}{B}\mathbb{I}_{\{-t_0-B<\Psi<-t_0\}}|f|^2e^{-\varphi_{\alpha}-\Psi}.
  \end{split}
\end{equation*}

 Note that $E\subset\{\Psi=-\infty\}\subset\{\Psi<-t_0\}$ and $\{\Psi<-t_0\}$ is open, then for any $z\in E$, there exists an open neighborhood $U_z$ of $z$ such that $U_z\Subset\{\Psi<-t_0\}$.

As $v_{t_0,B}(t)\ge -t_0-\frac{B}{2}$, we have $c(-v_{t_0,B}(\Psi))e^{v_{t_0,B}(\Psi)}\ge c(t_0+\frac{B}{2})e^{-t_0-\frac{B}{2}}>0$. Note that $\varphi+\Psi$ is plurisubharmonic on $M$. Thus we have
\begin{equation*}
  \begin{split}
      & \int_{U_z\backslash E}|\tilde{F}_E-(1-b_{t_0,B}(\Psi))fF^{1+\delta}|^2 \\
      \le & \frac{1}{C_1}\int_{U_z\backslash E}|\tilde{F}_E-(1-b_{t_0,B}(\Psi))fF^{1+\delta}|^2e^{-\varphi+v_{t_0,B}(\Psi)-\Psi}c(-v_{t_0,B}(\Psi))
      <+\infty,
  \end{split}
\end{equation*}
where $C_1$ is some positive number.

As $U_z\Subset\{\Psi<-t_0\}$, we have
\begin{equation*}
  \begin{split}
       \int_{U_z\backslash E}|(1-b_{t_0,B}(\Psi))fF^{1+\delta}|^2
      \le
      \left(\sup_{U_Z}|F^{1+\delta}|^2\right)\int_{U_z}|f|^2 <+\infty.
  \end{split}
\end{equation*}
Hence we have
$$\int_{U_z\backslash E}|\tilde{F}_E|^2<+\infty. $$
As $E$ is contained in some analytic subset of $M$, we can find a holomorphic extension $\tilde{F}$ of $\tilde{F}_E$ from $M\backslash E$ to $M$ such that

\begin{equation}
\begin{split}
&\int_{M}|\tilde{F}-(1-b_{t_0,B}(\Psi))fF^{1+\delta}|^2e^{-\varphi-\Psi+v_{t_0,B}(\Psi)}c(-v_{t_0,B}(\Psi))\\
\le & \left(\frac{1}{\delta}c(T)e^{-T}+\int_{T}^{t_0+B}c(t)e^{-t}dt\right)\int_M \frac{1}{B} \mathbb{I}_{\{-t_0-B< \Psi < -t_0\}}  {|f|}^2
e^{{-}\varphi_{\alpha}-\Psi}.
\end{split}
\end{equation}
Lemma \ref{L2 method for c(t)} is proved.
\end{proof}

Let $T_0\in[-\infty,+\infty)$. Let $c(t)\in \mathcal{P}_{T_0,M,\Psi_0}$.
Using Lemma \ref{L2 method for c(t)}, we have the following lemma, which will be used to prove Theorem \ref{main theorem}.
\begin{Lemma}
\label{L2 method in JM concavity}
Let $(M,X,Z)$  satisfies condition $(A)$. Let $B \in (0, +\infty)$ and $t_0>t_1\ge T> T_0$ be arbitrarily given.
Let $f$ be a holomorphic $(n,0)$ form on $\{\Psi< -t_0\}$ such that
\begin{equation}
\int_{\{\Psi<-t_0\}} {|f|}^2e^{-\varphi_\alpha}c(-\Psi)<+\infty,
\label{condition of JM concavity}
\end{equation}
Then there exists a holomorphic $(n,0)$ form $\tilde{F}$ on $\{\Psi<-t_1\}$ such that

 \begin{equation*}
  \begin{split}
      & \int_{\{\Psi<-t_1\}}|\tilde{F}-(1-b_{t_0,B}(\Psi))f|^2e^{-\varphi_\alpha+v_{t_0,B}(\Psi)-\Psi}c(-v_{t_0,B}(\Psi)) \\
      \le & \left(\int_{t_1}^{t_0+B}c(s)e^{-s}ds\right)
       \int_{M}\frac{1}{B}\mathbb{I}_{\{-t_0-B<\Psi<-t_0\}}|f|^2e^{-\varphi_{\alpha}-\Psi},
  \end{split}
\end{equation*}
where $b_{t_0,B}(t)=\int^{t}_{-\infty}\frac{1}{B} \mathbb{I}_{\{-t_0-B< s < -t_0\}}ds$,
$v_{t_0,B}(t)=\int^{t}_{-t_0}b_{t_0,B}(s)ds-t_0$.
\end{Lemma}

\begin{proof}[Proof of Lemma \ref{L2 method in JM concavity}]

Denote that
$$\tilde{\varphi}:=\varphi_\alpha+(1+\delta)\max\{\psi+t_1,2\log|F|\}$$
and
$$\tilde{\Psi}:=\min\{\psi-2\log|F|,-t_1\}.$$

As $t_0>t_1\ge T$, we have $\{\Psi<-t_0\}=\{\tilde{\Psi}<-t_0\}$. It follows from inequality \eqref{condition of JM concavity} and Lemma \ref{L2 method for c(t)} that there exists a holomorphic function $\tilde{F}_{\delta}$ on $M$ such that

\begin{equation*}
  \begin{split}
      & \int_{M}|\tilde{F}_{\delta}-(1-b_{t_0,B}(\tilde\Psi))fF^{1+\delta}|^2e^{-\tilde \varphi+v_{t_0,B}(\tilde \Psi)-\tilde \Psi}c(-v_{t_0,B}(\tilde \Psi)) \\
      \le & \left(\frac{1}{\delta}c(t_1)e^{-t_1}+\int_{t_1}^{t_0+B}c(s)e^{-s}ds\right)
       \int_{M}\frac{1}{B}\mathbb{I}_{\{-t_0-B<\tilde\Psi<-t_0\}}|f|^2e^{-\varphi_{\alpha}-\tilde\Psi}.
  \end{split}
\end{equation*}

Note that on $\{\Psi<-t_1\}$, we have $\Psi=\tilde{\Psi}=\psi-2\log|F|$ and $\tilde{\varphi}=\varphi=\varphi_\alpha+(1+\delta)2\log|F|$.  Hence
\begin{equation}\label{1st formula in L2 method JM concavity}
\begin{split}
    & \int_{\{\Psi<-t_1\}}|\tilde{F}_{\delta}-(1-b_{t_0,B}(\Psi))fF^{1+\delta}|^2
    e^{-\varphi+v_{t_0,B}(\Psi)-\Psi}c(-v_{t_0,B}(\Psi)) \\
    = & \int_{\{\Psi<-t_1\}}|\tilde{F}_{\delta}-(1-b_{t_0,B}(\tilde\Psi))fF^{1+\delta}|^2
     e^{-\tilde\varphi+v_{t_0,B}(\tilde\Psi)-\tilde\Psi}c(-v_{t_0,B}(\tilde\Psi))\\
     \le &
     \int_{M}|\tilde{F}_{\delta}-(1-b_{t_0,B}(\tilde\Psi))fF^{1+\delta}|^2
     e^{-\tilde\varphi+v_{t_0,B}(\tilde\Psi)-\tilde\Psi}c(-v_{t_0,B}(\tilde\Psi))\\
     \le & \left(\frac{1}{\delta}c(t_1)e^{-t_1}+\int_{t_1}^{t_0+B}c(s)e^{-s}ds\right)
       \int_{M}\frac{1}{B}\mathbb{I}_{\{-t_0-B<\tilde\Psi<-t_0\}}|f|^2e^{-\varphi_{\alpha}-\tilde\Psi}\\
       =&
     \left(\frac{1}{\delta}c(t_1)e^{-t_1}+\int_{t_1}^{t_0+B}c(s)e^{-s}ds\right)
       \int_{M}\frac{1}{B}\mathbb{I}_{\{-t_0-B<\Psi<-t_0\}}|f|^2e^{-\varphi_{\alpha}-\Psi}<+\infty.
\end{split}
\end{equation}

 Let $F_{\delta}:=\frac{\tilde{F}_{\delta}}{F^{\delta}}$ be a holomorphic function on $\{\Psi<-t_1\}$. Then it follows from \eqref{1st formula in L2 method JM concavity} that
\begin{equation}\label{2nd formula in L2 method JM concavity}
\begin{split}
    & \int_{\{\Psi<-t_1\}}|F_{\delta}-(1-b_{t_0,B}(\Psi))fF|^2
    e^{-\varphi_\alpha+v_{t_0,B}(\Psi)-\psi}c(-v_{t_0,B}(\Psi)) \\
    \le&
     \bigg(\frac{1}{\delta}c(t_1)e^{-t_1}+\int_{t_1}^{t_0+B}c(s)e^{-s}ds\bigg)
       \int_{M}\frac{1}{B}\mathbb{I}_{\{-t_0-B<\Psi<-t_0\}}|f|^2e^{-\varphi_{\alpha}-\Psi}.
\end{split}
\end{equation}

 Let $K$ be any compact subset of $M\backslash E$. Note that $\inf_{K}e^{-\varphi_{\alpha}+v_{t_0,B}(\Psi)-\psi}c(-v_{t_0,B}(\Psi))\ge \left(c(t_0+\frac{2}{B})e^{-t_0-\frac{2}{B}}\right)\inf_{K}e^{-\varphi_{\alpha}-\psi}>0$. It follows from \eqref{2nd formula in L2 method JM concavity} that we have
 $$\sup_{\delta} \int_{\{\Psi<-t_1\}\cap K}|F_{\delta}-(1-b_{t_0,B}(\Psi))fF|^2<+\infty.$$

We also note that
$$\int_{\{\Psi<-t_1\}\cap K}|(1-b_{t_0,B}(\Psi))fF|^2\le
\left(\sup_{K}|F|^2\right)\int_{\{\Psi<-t_0\}\cap K}|f|^2<+\infty.$$
Then we know that
$$\sup_{\delta} \int_{\{\Psi<-t_1\}\cap K}|F_{\delta}|^2<+\infty,$$
and there exists a subsequence of $\{F_\delta\}$ (also denoted by $F_\delta$) compactly convergent to a holomorphic $(n,0)$ form $\tilde{F}_1$ on $\{\Psi<-t_1\}\backslash E$.
It follows from Fatou's Lemma and inequality \eqref{2nd formula in L2 method JM concavity} that we have

\begin{equation}\label{3rd formula in L2 method JM concavity}
\begin{split}
  & \int_{\{\Psi<-t_1\}\backslash E}|\tilde{F}_1-(1-b_{t_0,B}(\Psi))fF|^2
    e^{-\varphi_{\alpha}+v_{t_0,B}(\Psi)-\psi}c(-v_{t_0,B}(\Psi)) \\
    \le &\liminf_{\delta\to +\infty} \int_{\{\Psi<-t_1\}\backslash E}|F_{\delta}-(1-b_{t_0,B}(\Psi))fF|^2
    e^{-\varphi_{\alpha}+v_{t_0,B}(\Psi)-\psi}c(-v_{t_0,B}(\Psi)) \\
     \le &\liminf_{\delta\to +\infty} \int_{\{\Psi<-t_1\}}|F_{\delta}-(1-b_{t_0,B}(\Psi))fF|^2
    e^{-\varphi_{\alpha}+v_{t_0,B}(\Psi)-\psi}c(-v_{t_0,B}(\Psi)) \\
       \le&\liminf_{\delta\to +\infty}
     \bigg(\frac{1}{\delta}c(t_1)e^{-t_1}+\int_{t_1}^{t_0+B}c(s)e^{-s}ds\bigg)
       \int_{M}\frac{1}{B}\mathbb{I}_{\{-t_0-B<\Psi<-t_0\}}|f|^2e^{-\varphi_{\alpha}-\Psi}\\
       \le &\left(\int_{t_1}^{t_0+B}c(s)e^{-s}ds\right)
       \int_{M}\frac{1}{B}\mathbb{I}_{\{-t_0-B<\Psi<-t_0\}}|f|^2e^{-\varphi_{\alpha}-\Psi}.
\end{split}
\end{equation}
Note that $E\subset\{\Psi=-\infty\}\subset\{\Psi<-t_1\}$ and $\{\Psi<-t_1\}$ is open, then for any $z\in E$, there exists an open neighborhood $U_z$ of $z$ such that $U_z\Subset\{\Psi<-t_1\}$. Note that $\varphi_{\alpha}+\psi$ is plurisubharmonic function on $M$. As $v_{t_0,B}(t)\ge -t_0-\frac{B}{2}$, we have $c(-v_{t_0,B}(\Psi_1))e^{v_{t_0,B}(\Psi_1)}\ge c(t_0+\frac{B}{2})e^{-t_0-\frac{B}{2}}>0$. Thus by \eqref{3rd formula in L2 method JM concavity}, we have
\begin{equation*}
  \begin{split}
      & \int_{U_z\backslash E}|\tilde{F}_1-(1-b_{t_0,B}(\Psi))fF|^2 \\
      \le  & \frac{1}{C_1}\int_{U_z\backslash E}|\tilde{F}_1-(1-b_{t_0,B}(\Psi))fF|^2e^{-\varphi_{\alpha}+v_{t_0,B}(\Psi)-\psi}c(-v_{t_0,B}(\Psi))
      <+\infty,
  \end{split}
\end{equation*}
where $C_1:=c(t_0+\frac{B}{2})e^{-t_0-\frac{B}{2}}\inf_{U_z}e^{-\varphi_{\alpha}-\psi}$ is some positive number.

As $U_z\Subset\{\Psi<-t_1\}$, we have
\begin{equation*}
  \begin{split}
       \int_{U_z\backslash E}|(1-b_{t_0,B}(\Psi))fF|^2
      \le
      \left(\sup_{U_z}|F|^2\right)\int_{U_z}|f|^2 <+\infty.
  \end{split}
\end{equation*}
Hence we have
$$\int_{U_z\backslash E}|\tilde{F}_1|^2<+\infty. $$
As $E$ is contained in some analytic subset of $M$, we can find a holomorphic extension $\tilde{F}_0$ of $\tilde{F}_1$ from $\{\Psi<-t_1\}\backslash E$ to $\{\Psi<-t_1\}$ such that
\begin{equation}\label{4th formula in L2 method JM concavity}
\begin{split}
  & \int_{\{\Psi<-t_1\}}|\tilde{F}_0-(1-b_{t_0,B}(\Psi))fF|^2
    e^{-\varphi_{\alpha}+v_{t_0,B}(\Psi)-\psi}c(-v_{t_0,B}(\Psi)) \\
       \le &\left(\int_{t_1}^{t_0+B}c(s)e^{-s}ds\right)
       \int_{M}\frac{1}{B}\mathbb{I}_{\{-t_0-B<\Psi<-t_0\}}|f|^2e^{-\varphi_{\alpha}-\Psi}.
\end{split}
\end{equation}

Denote $\tilde{F}:=\frac{\tilde{F}_0}{F}$. Note that on $\{\Psi<-t_1\}$, we have $\Psi=\psi-2\log|F|$. It follows from \eqref{4th formula in L2 method JM concavity} that we have
\begin{equation}\nonumber
\begin{split}
  & \int_{\{\Psi<-t_1\}}|\tilde{F}-(1-b_{t_0,B}(\Psi))f|^2
    e^{-\varphi_{\alpha}+v_{t_0,B}(\Psi)-\Psi}c(-v_{t_0,B}(\Psi)) \\
       \le &\left(\int_{t_1}^{t_0+B}c(s)e^{-s}ds\right)
       \int_{M}\frac{1}{B}\mathbb{I}_{\{-t_0-B<\Psi<-t_0\}}|f|^2e^{-\varphi_{\alpha}-\Psi}.
\end{split}
\end{equation}

Lemma \ref{L2 method in JM concavity} is proved.

\end{proof}

\subsection{Properties of $\mathcal{O}_{M,z_0}$-module $J_{z_0}$}
\label{sec:properties of module}
In this section, we present some properties of $\mathcal{O}_{M,z_0}$-module $J_{z_0}$.

Since the case is local, we assume that $F$ is a holomorphic function on a pseudoconvex domain $D\subset \mathbb{C}^n$ containing the origin $o\in \mathbb{C}^n$. Let $\psi$ be a plurisubharmonic function on $D$. Let $\varphi_\alpha$ be a Lebesgue measurable function on $D$ such that $\psi+\varphi_\alpha$ is plurisubharmonic. Let $T_0\in [-\infty,+\infty)$ and $T>T_0$ be a real number. Denote
$$\varphi_1:=2\max\{\psi+T,2\log|F|\}$$
and
$$\Psi:=\min\{\psi-2\log|F|,-T\}.$$
If $F(z)=0$ for some $z \in D$, we set $\Psi(z)=-T$.

Let $c(t)$ be a positive measurable function on $(T_0,+\infty)$ such that\\
(1) $c(t)e^{-t}$ is decreasing with respect $t$;\\
(2) $c(-\Psi)e^{-\varphi_\alpha}$ has a positive lower bound on $D$.\\

Denote that $H_o:=\{f_o\in J(\Psi)_o:\int_{\{\Psi<-t\}\cap V_0}|f|^2e^{-\varphi_\alpha}c(-\Psi)<+\infty \text{ for some }t>T_0 \text{ and } V_0 \text{ is an open neighborhood of o}\}$ and
$\mathcal{H}_o:=\{(F,o)\in \mathcal{O}_{\mathbb{C}^n,o}:\int_{U_0}|F|^2e^{-\varphi_\alpha-\varphi_1}c(-\Psi)<+\infty \text{ for some open neighborhood} \ U_0 \text{ of } o\}$.

As  $c(-\Psi)e^{-\varphi_\alpha}$ has a positive lower bound on $D$ and $c(t)e^{-t}$ is decreasing with respect to $t$, we have $I(\Psi+\varphi_\alpha)_o\subset H_o\subset I_o$.
We also note that $\mathcal{H}_o$ is an ideal of $\mathcal{O}_{\mathbb{C}^n,o}$.

\begin{Lemma}\label{construction of morphism} For any $f_o\in H_o$, there exist a pseudoconvex domain  $D_0\subset D$ containing $o$ and a holomorphic function $\tilde{F}$ on $D_0$ such that $(\tilde{F},o)\in \mathcal{H}_o$ and
$$\int_{\{\Psi<-t_1\}\cap D_0}|\tilde{F}-fF^2|e^{-\varphi_\alpha-\varphi_1-\Psi}<+\infty,$$
for some $t_1>T$.
\end{Lemma}
\begin{proof}It follows from $f_o\in H_o$ that there exists $t_0>T>T_0$ and a pseudoconvex domain  $D_0\subset D$ containing $o$ such that
\begin{equation}\label{construction of morphism formula 1}
\int_{\{\Psi<-t_0\}\cap D_0}|f|^2e^{-\varphi_\alpha}c(-\Psi)<+\infty.
\end{equation}
As $c(t)e^{-t}$ is decreasing with respect to $t$, it follows from inequality \eqref{construction of morphism formula 1} that we have
$\int_{D_0}\mathbb{I}_{\{-t_0-1<\Psi<-t_0\}}|f|^2e^{-\varphi_\alpha-\Psi}<+\infty$.
As  $c(-\Psi)e^{-\varphi_\alpha}$ has a positive lower bound on $D$, we have $\int_{\{\Psi<-t_0\}\cap D_0}|f|^2<+\infty$. Then it follows from Lemma \ref{L2 method} that there exists a holomorphic function $\tilde F$ on $D_0$ such that
\begin{equation*}
  \begin{split}
      & \int_{D_0}|\tilde{F}-(1-b_{t_0}(\Psi))fF^{2}|^2e^{-\varphi_\alpha-\varphi_1+v_{t_0}(\Psi)-\Psi}c(-v_{t_0}(\Psi)) \\
      \le & \left(c(T)e^{-T}+\int_{T}^{t_0+1}c(s)e^{-s}ds\right)
       \int_{D_0}\mathbb{I}_{\{-t_0-1<\Psi<-t_0\}}|f|^2e^{-\varphi_{\alpha}-\Psi},
  \end{split}
\end{equation*}
where $b_{t_0}(t)=\int^{t}_{-\infty} \mathbb{I}_{\{-t_0-1< s < -t_0\}}ds$,
$v_{t_0}(t)=\int^{t}_{-t_0}b_{t_0}(s)ds-t_0$. Denote $C:=c(T)e^{-T}+\int_{T}^{t_0+B}c(s)e^{-s}ds$, we note that $C$ is a positive number.

Note that $v_{t_0}(t)>-t_0-1$. We have $e^{v_{t_0}(\Psi)}c(-v_{t_0}(\Psi))\ge c(t_0+1)e^{-(t_0+1)}>0$. As $b_{t_0}(t)\equiv 0$ on $(-\infty,-t_0-1)$, we have
\begin{equation}\label{construction of morphism formula 2}
\begin{split}
   &\int_{D_0\cap\{\Psi<-t_0-1\}}|\tilde{F}-fF^2|^2e^{-\varphi_{\alpha}-\varphi_1-\Psi} \\
   \le & \frac{1}{c(t_0+1)e^{-(t_0+1)}}
   \int_{D_0}|\tilde{F}-(1-b_{t_0}(\Psi))fF^2|^2e^{-\varphi_{\alpha}-\varphi_1-\Psi+v_{t_0}(\Psi)}c(-v_{t_0}(\Psi))\\
   \le &\frac{C}{c(t_0+1)e^{-(t_0+1)}}
   \int_{D_0}\mathbb{I}_{\{-t_0-1<\Psi<-t_0\}}|f|^2e^{-\varphi_{\alpha}-\Psi}<+\infty.
\end{split}
\end{equation}
Note that on $\{\Psi<-t_0\}$, $|F|^2e^{-\varphi_1}=1$. As $v_{t_0}(\Psi)\ge \Psi$, we have $c(-v_{t_0}(\Psi))e^{v_{t_0}(\Psi)}\ge c(-\Psi)e^{-\Psi}$. Hence we have
\begin{equation}\nonumber
\begin{split}
   &\int_{D_0}|\tilde{F}|^2e^{-\varphi_{\alpha}-\varphi_1}c(-\Psi) \\
   \le & 2\int_{D_0}|\tilde{F}-(1-b_{t_0}(\Psi))fF^2|^2e^{-\varphi_{\alpha}-\varphi_1}c(-\Psi)\\
   +&2\int_{D_0}|(1-b_{t_0}(\Psi))fF^2|^2e^{-\varphi_{\alpha}-\varphi_1}c(-\Psi)\\
   \le&
   2\int_{D_0}|\tilde{F}-(1-b_{t_0}(\Psi))fF^2|^2e^{-\varphi_{\alpha}-\varphi_1-\Psi+v_{t_0}(\Psi)}c(-v_{t_0}(\Psi))\\
   +&2\int_{D_0\cap\{\Psi<-t_0\}}|f|^2e^{-\varphi_{\alpha}}c(-\Psi)\\
   < &+\infty.
\end{split}
\end{equation}
Hence we know that $(\tilde{F},o)\in \mathcal{H}_o$.
\end{proof}

For any $(\tilde{F},o)\in\mathcal{H}_o$ and $(\tilde{F}_1,o)\in\mathcal{H}_o$ such that $\int_{D_1\cap\{\Psi<-t_1\}}|\tilde{F}-fF^2|^2e^{-\varphi_{\alpha}-\varphi_1-\Psi}<+\infty$ and
$\int_{D_1\cap\{\Psi<-t_1\}}|\tilde{F}_1-fF^2|^2e^{-\varphi_{\alpha}-\varphi_1-\Psi}<+\infty$, for some open neighborhood $D_1$ of $o$ and $t_1\ge T$, we have
$$\int_{D_1\cap\{\Psi<-t_1\}}|\tilde{F}_1-\tilde{F}|^2e^{-\varphi_{\alpha}-\varphi_1-\Psi}<+\infty.$$
As $(\tilde{F},o)\in\mathcal{H}_o$ and $(\tilde{F}_1,o)\in\mathcal{H}_o$, there exists a neighborhood $D_2$ of $o$ such that
\begin{equation}\label{construction of morphism formula 3}
\int_{D_2}|\tilde{F}_1-\tilde{F}|^2e^{-\varphi_{\alpha}-\varphi_1}c(-\Psi)<+\infty.
\end{equation}
Note that we have $c(-\Psi)e^{\Psi}\ge c(t_1)e^{-t_1}$ on $\{\Psi\ge-t_1\}$. It follows from inequality \eqref{construction of morphism formula 3} that we have
$$\int_{D_2\cap \{\Psi\ge-t_1\}}|\tilde{F}_1-\tilde{F}|^2e^{-\varphi_{\alpha}-\varphi_1-\Psi}<+\infty.$$
Hence we have $(\tilde{F}-\tilde{F}_1,o)\in \mathcal{I}(\varphi_{\alpha}+\varphi_1+\Psi)_o$.

Thus it follows from Lemma \ref{construction of morphism} that there exists a map $\tilde{P}:H_o\to \mathcal{H}_o/\mathcal{I}(\varphi_{\alpha}+\varphi_1+\Psi)_o$ given by
$$\tilde{P}(f_o)=[(\tilde{F},o)]$$
for any $f_o\in H_o$, where $(\tilde{F},o)$ satisfies $(\tilde{F},o)\in \mathcal{H}_o$ and
$\int_{D_1\cap\{\Psi<-t\}}|\tilde{F}-fF^2|^2e^{-\varphi_{\alpha}-\varphi_1-\Psi}<+\infty,$
for some $t>T$ and some open neighborhood $D_1$ of $o$, and $[(\tilde{F},o)]$ is the equivalence class of $(\tilde{F},o)$ in $\mathcal{H}_o/\mathcal{I}(\varphi_{\alpha}+\varphi_1+\Psi)_o$.

\begin{Proposition}\label{proposition of morphism}
$\tilde{P}$ is an $\mathcal{O}_{\mathbb{C}^n,o}$-module homomorphism and $Ker(\tilde{P})=I(\varphi_\alpha+\Psi)_o$.
\end{Proposition}
\begin{proof}For any $f_o,g_o\in H_o$. Denote that $\tilde{P}(f_o)=[(\tilde{F},o)]$, $\tilde{P}(g_o)=[(\tilde{G},o)]$ and $\tilde{P}(f_o+g_o)=[(\tilde{H},o)]$.

Note that there exists an open neighborhood $D_1$ of $o$ and $t\ge T$ such that $\int_{D_1\cap\{\Psi<-t\}}|\tilde{F}-fF^2|^2e^{-\varphi_{\alpha}-\varphi_1-\Psi}<+\infty$,
$\int_{D_1\cap\{\Psi<-t\}}|\tilde{G}-gF^2|^2e^{-\varphi_{\alpha}-\varphi_1-\Psi}<+\infty$, and
$\int_{D_1\cap\{\Psi<-t\}}|\tilde{H}-(f+g)F^2|^2e^{-\varphi_{\alpha}-\varphi_1-\Psi}<+\infty$. Hence we have
$$\int_{D_1\cap\{\Psi<-t\}}|\tilde{H}-(\tilde{F}+\tilde{G})|^2e^{-\varphi_{\alpha}-\varphi_1-\Psi}<+\infty.$$
As $(\tilde{F},o),(\tilde{G},o)$ and $(\tilde{H},o)$ belong to $ \mathcal{H}_o$, there exists an open neighborhood $\tilde{D}_1\subset D_1$ of $o$ such that
$\int_{\tilde{D}_1}|\tilde{H}-(\tilde{F}+\tilde{G})|^2e^{-\varphi_{\alpha}-\varphi_1}c(-\Psi)<+\infty$.
As $c(t)e^{-t}$ is decreasing with respect to $t$, we have $c(-\Psi)e^{\Psi}\ge c(t)e^{-t}$ on $\{\Psi\ge -t\}$. Hence we have
$$\int_{\tilde{D}_1\cap\{\Psi\ge -t\}}|\tilde{H}-(\tilde{F}+\tilde{G})|^2e^{-\varphi_{\alpha}-\varphi_1-\Psi}
\le\frac{1}{c(t)e^{-t}}\int_{\tilde{D}_1\cap\{\Psi\ge -t\}}|\tilde{H}-(\tilde{F}+\tilde{G})|^2e^{-\varphi_{\alpha}-\varphi_1}c(-\Psi)<+\infty.$$
Thus we have $\int_{\tilde{D}_1}|\tilde{H}-(\tilde{F}+\tilde{G})|^2e^{-\varphi_{\alpha}-\varphi_1-\Psi}<+\infty$, which implies that $\tilde{P}(f_o+g_o)=\tilde{P}(f_o)+\tilde{P}(g_o)$.

For any $(h,o) \in \mathcal{O}_{\mathbb{C}^n,o}$. Denote $\tilde{P}((hf)_o)=[(\tilde{F}_h,o)]$. Note that there exists an open neighborhood $D_2$ of $o$ and $t\ge T$ such that $\int_{D_2\cap\{\Psi<-t\}}|\tilde{F}_h-(hf)F^2|^2e^{-\varphi_{\alpha}-\varphi_1-\Psi}<+\infty$. It follows from $\int_{D_2\cap\{\Psi<-t\}}|\tilde{F}-fF^2|^2e^{-\varphi_{\alpha}-\varphi_1-\Psi}<+\infty$ and $h$ is holomorphic on $\overline{D_2}$ (shrink $D_2$ if necessary) that $\int_{D_2\cap\{\Psi<-t\}}|h\tilde{F}-hfF^2|^2e^{-\varphi_{\alpha}-\varphi_1-\Psi}<+\infty$. Then we have
$$\int_{D_2\cap\{\Psi<-t\}}|\tilde{F}_h-h\tilde{F}|^2e^{-\varphi_{\alpha}-\varphi_1-\Psi}<+\infty.$$
Note that $(h\tilde{F},o) $ and $(\tilde{F}_h,o)$ belong to $ \mathcal{H}_o$, we have
$\int_{D_2}|\tilde{F}_h-h\tilde{F}|^2e^{-\varphi_{\alpha}-\varphi_1}c(-\Psi)<+\infty$.
As $c(t)e^{-t}$ is decreasing with respect to $t$, we have $c(-\Psi)e^{\Psi}\ge c(t)e^{-t}$ on $\{\Psi\ge -t\}$. Hence we have
$$\int_{D_2\cap\{\Psi\ge -t\}}|\tilde{F}_h-h\tilde{F}|^2e^{-\varphi_{\alpha}-\varphi_1-\Psi}
\le\frac{1}{c(t)e^{-t}}\int_{D_2\cap\{\Psi\ge -t\}}|\tilde{F}_h-h\tilde{F}|^2e^{-\varphi_{\alpha}-\varphi_1}c(-\Psi)<+\infty.$$
Thus we have $\int_{D_2}|\tilde{F}_h-h\tilde{F}|^2e^{-\varphi_{\alpha}-\varphi_1-\Psi}<+\infty$, which implies that $\tilde{P}(hf_o)=(h,o)\tilde{P}(f_o)$.

Now we have proved that $\tilde{P}$ is an $\mathcal{O}_{\mathbb{C}^n,o}$-module homomorphism.

Next, we prove $Ker(\tilde{P})=I(\varphi_\alpha+\Psi)_o$.

If $f_o\in I(\varphi_\alpha+\Psi)_o$. Denote $\tilde{P}(f_o)=[(\tilde{F},o)]$. It follows from Lemma \ref{construction of morphism} that $(\tilde{F},o)\in \mathcal{H}_o$ and there exists an open neighborhood $D_3$ of $o$ and a real number $t_1>T$ such that
$$\int_{\{\Psi<-t_1\}\cap D_3}|\tilde{F}-fF^2|e^{-\varphi_\alpha-\varphi_1-\Psi}<+\infty.$$
As $f_o\in I(\varphi_\alpha+\Psi)_o$, shrink $D_3$ and $t_1$ if necessary, we have
\begin{equation}\label{proposition of morphism formula 1}
\begin{split}
&\int_{\{\Psi<-t_1\}\cap D_3}|\tilde{F}|^2e^{-\varphi_\alpha-\varphi_1-\Psi}\\
\le &2\int_{\{\Psi<-t_1\}\cap D_3}|\tilde{F}-fF^2|^2e^{-\varphi_\alpha-\varphi_1-\Psi}
+2\int_{\{\Psi<-t_1\}\cap D_3}|fF^2|^2e^{-\varphi_\alpha-\varphi_1-\Psi}\\
\le &2\int_{\{\Psi<-t_1\}\cap D_3}|\tilde{F}-fF^2|^2e^{-\varphi_\alpha-\varphi_1-\Psi}
+2\int_{\{\Psi<-t_1\}\cap D_3}|f|^2e^{-\varphi_\alpha-\Psi}\\
<&+\infty.
\end{split}
\end{equation}
As $c(t)e^{-t}$ is decreasing with respect to $t$, $c(-\Psi)e^{\Psi}\ge C_0>0$ for some positive number $C_0$ on $\{\Psi\ge-t_1\}$. Then we have
\begin{equation}\label{proposition of morphism formula 1'}
\begin{split}
\int_{\{\Psi\ge-t_1\}\cap D_3}|\tilde{F}|^2e^{-\varphi_\alpha-\varphi_1-\Psi}
\le\frac{1}{C_0}\int_{\{\Psi\ge-t_1\}\cap D_3}|\tilde{F}|^2e^{-\varphi_\alpha-\varphi_1}c(-\Psi)<+\infty.
\end{split}
\end{equation}
Combining inequality \eqref{proposition of morphism formula 1} and inequality  \eqref{proposition of morphism formula 1'}, we know that $\tilde{F}\in \mathcal{I}(\varphi_\alpha+\varphi_1+\Psi)_o$, which means $\tilde{P}(f_o)=0$ in $\mathcal{H}_o/\mathcal{I}(\varphi_{\alpha}+\varphi_1+\Psi)_o$. Hence we know $I(\varphi_\alpha+\Psi)_o\subset Ker(\tilde{P})$.

If $f_o\in Ker(\tilde{P})$, we know $\tilde{F}\in \mathcal{I}(\varphi_\alpha+\varphi_1+\Psi)_o$.
We can assume that $\tilde{F}$ satisfies $\int_{D_4}|\tilde{F}|^2e^{-\varphi_\alpha-\varphi_1-\Psi}<+\infty$ for some open neighborhood $D_4$ of $o$. Then we have
\begin{equation}\label{proposition of morphism formula 2'}
\begin{split}
&\int_{ \{\Psi<-t_1\}\cap D_4}|f|^2e^{-\varphi_\alpha-\Psi}\\
=&\int_{\{\Psi<-t_1\}\cap D_4}|fF^2|^2e^{-\varphi_\alpha-\varphi_1-\Psi}\\
\le & \int_{\{\Psi<-t_1\}\cap D_4}|\tilde{F}|^2e^{-\varphi_\alpha-\varphi_1-\Psi}+\int_{\{\Psi<-t_1\}\cap D_4}|\tilde{F}-fF^2|e^{-\varphi_\alpha-\varphi_1-\Psi}\\
\le &+\infty.
\end{split}
\end{equation}
By definition, we know $f_o\in I(\varphi_\alpha+\Psi)_o$. Hence $ Ker(\tilde{P})\subset I(\varphi_\alpha+\Psi)_o$.

$ Ker(\tilde{P})= I(\varphi_\alpha+\Psi)_o$ is proved.

\end{proof}

Now we can define an $\mathcal{O}_{\mathbb{C}^n,o}$-module homomorphism $P:H_o/I(\varphi_\alpha+\Psi)_o\to \mathcal{H}_o/\mathcal{I}(\varphi_{\alpha}+\varphi_1+\Psi)_o$ as follows,
$$P([f_o])=\tilde{P}(f_o)$$
for any $[f_o]\in H_o/I(\varphi_\alpha+\Psi)_o$, where $f_o\in H_o$ is any representative of $[f_o]$. It follows from Proposition \ref{proposition of morphism} that $P([f_o])$ is independent of the choices of the representatives of $[f_o]$.

Let $(\tilde{F},o)\in \mathcal{H}_o$, i.e. $\int_{U}|\tilde{F}|^2e^{-\varphi_\alpha-\varphi_1}c(-\Psi)<+\infty$ for some neighborhood $U$ of $o$. Note that $|F|^4e^{-\varphi_1}\equiv 1$ on $\{\Psi<-T\}$. Hence we have $\int_{U\cap \{\Psi<-t\}}|\frac{\tilde{F}}{F^2}|^2e^{-\varphi_\alpha}c(-\Psi)<+\infty$ for some $t>T$, i.e. $(\frac{\tilde{F}}{F^2})_o\in H_o$. And if $(\tilde{F},o)\in \mathcal{I}(\varphi_{\alpha}+\varphi_1+\Psi)_o$, it is easy to verify that $(\frac{\tilde{F}}{F^2})_o\in I(\varphi_\alpha+\Psi)_o$. Hence we have an $\mathcal{O}_{\mathbb{C}^n,o}$-module homomorphism $Q:\mathcal{H}_o/\mathcal{I}(\varphi_{\alpha}+\varphi_1+\Psi)_o\to H_o/I(\varphi_\alpha+\Psi)_o$ defined as follows,
$$Q([(\tilde{F},o)])=[(\frac{\tilde{F}}{F^2})_o].$$

The above discussion shows that $Q$ is independent of the choices of the representatives of $[(\tilde{F},o)]$ and hence $Q$ is well defined.

\begin{Proposition}\label{module isomorphism}$P:H_o/I(\varphi_\alpha+\Psi)_o\to \mathcal{H}_o/\mathcal{I}(\varphi_{\alpha}+\varphi_1+\Psi)_o$ is an $\mathcal{O}_{\mathbb{C}^n,o}$-module isomorphism and $P^{-1}=Q$.
\end{Proposition}
\begin{proof} It follows from Proposition \ref{proposition of morphism} that we know $P$ is injective.

Now we prove $P$ is surjective.

For any $[(\tilde{F},o)]$ in $\mathcal{H}_o/\mathcal{I}(\varphi_{\alpha}+\varphi_1+\Psi)_o$. Let $(\tilde{F},o)$ be any representatives of $[(\tilde{F},o)]$ in $\mathcal{H}_o$. Denote that $[(f_1)_o]:=[(\frac{\tilde{F}}{F^2})_o]=Q([(\tilde{F},o)])$. Let $(f_1)_o:=(\frac{\tilde{F}}{F^2})_o\in H_o$ be the representative of $[(f_1)_o]$. Denote $[(\tilde{F}_1,o)]:=\tilde{P}((f_1)_o)=P([(f_1)_o])$. By the construction of $\tilde{P}$, we know that $(\tilde{F}_1,o)\in \mathcal{H}_o$ and
$$\int_{D_1\cap\{\Psi<-t\}}|\tilde{F}_1-f_1F^2|e^{-\varphi_\alpha-\varphi_1-\Psi}<+\infty,$$
where $t>T$ and $D_1$ is some neighborhood of $o$. Note that $(f_1)_o:=(\frac{\tilde{F}}{F^2})_o$. Hence  we have
$$\int_{D_1\cap\{\Psi<-t\}}|\tilde{F}_1-\tilde{F}|e^{-\varphi_\alpha-\varphi_1-\Psi}<+\infty.$$
It follows from $(\tilde{F},o)\in \mathcal{H}_o$ and $(\tilde{F}_1,o)\in \mathcal{H}_o$ that there exists a neighborhood $D_2\subset D_1$ of $o$ such that
$$\int_{D_2}|\tilde{F}-\tilde{F}_1|^2e^{-\varphi_\alpha-\varphi_1}c(-\Psi)<+\infty.$$
Note that on $\{\Psi\ge -t\}$, we have $c(-\Psi)e^{\Psi}\ge c(t)e^{-t}>0$. Hence we have
$$\int_{D_2\cap \{\Psi\ge-t\}}|\tilde{F}-\tilde{F}_1|^2e^{-\varphi_\alpha-\varphi_1-\Psi}<+\infty.$$
Thus we know that $(\tilde{F}_1-\tilde{F},o) \in \mathcal{I}(\varphi_{\alpha}+\varphi_1+\Psi)_o$, i.e. $[(\tilde{F},o)]=[(\tilde{F}_1,o)]$ in $ \mathcal{H}_o/\mathcal{I}(\varphi_{\alpha}+\varphi_1+\Psi)_o$. Hence we have $P\circ Q([(\tilde{F},o)])=[(\tilde{F},o)]$, which implies that $P$ is surjective.

We have proved that $P:H_o/I(\varphi_\alpha+\Psi)_o\to \mathcal{H}_o/\mathcal{I}(\varphi_{\alpha}+\varphi_1+\Psi)_o$ is an $\mathcal{O}_{\mathbb{C}^n,o}$-module isomorphism and $P^{-1}=Q$.
\end{proof}

We recall the following property of closedness of holomorphic functions on a neighborhood of $o$.
\begin{Lemma}[see \cite{G-R}]
\label{closedness}
Let $N$ be a submodule of $\mathcal O_{\mathbb C^n,o}^q$, $1\leq q<+\infty$, let $f_j\in\mathcal O_{\mathbb C^n}(U)^q$ be a sequence of $q-$tuples holomorphic in an open neighborhood $U$ of the origin $o$. Assume that the $f_j$ converge uniformly in $U$ towards  a $q-$tuples $f\in\mathcal O_{\mathbb C^n}(U)^q$, assume furthermore that all germs $(f_{j},o)$ belong to $N$. Then $(f,o)\in N$.	
\end{Lemma}

The following lemma shows the closedness of submodules of $H_o$.

\begin{Lemma}\label{closedness of module}
Let $J_o$ be an $\mathcal{O}_{\mathbb{C}^n,o}$-submodule of $H_o$ such that $I(\varphi_\alpha+\Psi)_o\subset J_o$. Assume that $f_o\in J(\Psi)_o$. Let $U_0$ be a Stein open neighborhood of $o$. Let $\{f_j\}_{j\ge 1}$ be a sequence of holomorphic functions on $U_0\cap \{\Psi<-t_j\}$ for any $j\ge 1$, where $t_j>T>T_0$. Assume that $t_0:=\lim_{j\to +\infty}t_j\in[T,+\infty)$,
\begin{equation}\label{convergence property of module}
\limsup\limits_{j\to+\infty}\int_{U_0\cap\{\Psi<-t_j\}}|f_j|^2e^{-\varphi_\alpha}c(-\Psi)\le C<+\infty,
\end{equation}
and $(f_j-f)_o\in J_o$. Assume that $e^{-\varphi_\alpha}c(-\Psi)$ has a positive lower bound on $U_o$. Then there exists a subsequence of $\{f_j\}_{j\ge 1}$ compactly convergent to a holomorphic function $f_0$ on $\{\Psi<-t_0\}\cap U_0$ which satisfies
$$\int_{U_0\cap\{\Psi<-t_0\}}|f_0|^2e^{-\varphi_\alpha}c(-\Psi)\le C,$$
and $(f_0-f)_o\in J_o$.
\end{Lemma}
\begin{proof}It follows from $\limsup\limits_{j\to+\infty}\int_{U_0\cap\{\Psi<-t_j\}}|f_j|^2e^{-\varphi_\alpha}c(-\Psi)\le C<+\infty$ and $e^{-\varphi_\alpha}c(-\Psi)$ has a positive lower bound on $U_o$ that
$$\limsup\limits_{j\to+\infty}\int_{U_0\cap\{\Psi<-t_j\}}|f_j|^2<+\infty.$$
As $t_0:=\lim_{j\to +\infty}t_j\in[T,+\infty)$, by Montel's theorem, there exists a subsequence of $\{f_j\}_{j\ge 1}$ (also denoted by $\{f_j\}_{j\ge 1}$) compactly convergent to a holomorphic function $f_0$ on $\{\Psi<-t_0\}\cap U_0$. It follows from Fatou's Lemma that
$$\int_{U_0\cap\{\Psi<-t_0\}}|f_0|^2e^{-\varphi_\alpha}c(-\Psi)\le \liminf\limits_{j\to+\infty}\int_{U_0\cap\{\Psi<-t_j\}}|f_j|^2e^{-\varphi_\alpha}c(-\Psi)\le C.$$

Now we prove $(f_0-f)_o\in J_o$. We firstly recall some constructions in Lemma \ref{construction of morphism}.

As $c(t)e^{-t}$ is decreasing with respect to $t$, it follows from inequality \eqref{convergence property of module} that we have
$\sup_{j\ge1}\left(\int_{U_0}\mathbb{I}_{\{-t_j-1<\Psi<-t_j\}}|f_j|^2e^{-\varphi_\alpha-\Psi}\right)<+\infty$.
As  $c(-\Psi)e^{-\varphi_\alpha}$ has a positive lower bound on $U_0$, we have $\sup_{j\ge1}\left(\int_{\{\Psi<-t_j\}\cap U_0}|f_j|^2\right)<+\infty$. Then it follows from Lemma \ref{L2 method} that there exists a holomorphic function $\tilde{F}_j$ on $U_0$ such that
\begin{equation}\label{convergence property of module formula 1}
  \begin{split}
      & \int_{U_0}|\tilde{F}_j-(1-b_{t_j}(\Psi))f_jF^{2}|^2e^{-\varphi_\alpha-\varphi_1+v_{t_j}(\Psi)-\Psi}c(-v_{t_j}(\Psi)) \\
      \le & \left(c(T)e^{-T}+\int_{T}^{t_j+1}c(s)e^{-s}ds\right)
       \int_{U_0}\mathbb{I}_{\{-t_0-1<\Psi<-t_0\}}|f_j|^2e^{-\varphi_{\alpha}-\Psi},
  \end{split}
\end{equation}
where $b_{t_j}(t)=\int^{t}_{-\infty} \mathbb{I}_{\{-t_j-1< s < -t_j\}}ds$,
$v_{t_j}(t)=\int^{t}_{-t_j}b_{t_j}(s)ds-t_j$. Denote $C_j:=c(T)e^{-T}+\int_{T}^{t_j+1}c(s)e^{-s}ds$. As $t_0:=\lim_{j\to +\infty}t_j\in[T,+\infty)$, we can assume that there exists a positive number $C_0<+\infty$ such that $C_j\le C_0$ for all $j\ge 1$.

Note that $v_{t_j}(t)>-t_j-1$. We have $e^{v_{t_j}(\Psi)-\Psi}c(-v_{t_j}(\Psi))\ge c(t_j+1)e^{-(t_j+1)}>0$. As $b_{t_j}(t)\equiv 0$ on $(-\infty,-t_j-1)$, we have
\begin{equation}\label{convergence property of module formula 2}
\begin{split}
   &\int_{U_0\cap\{\Psi<-t_j-1\}}|\tilde{F}_j-f_jF^2|^2e^{-\varphi_{\alpha}-\varphi_1-\Psi} \\
   \le & \frac{1}{c(t_j+1)e^{-(t_j+1)}}
   \int_{U_0}|\tilde{F}_j-(1-b_{t_j}(\Psi))f_jF^2|^2e^{-\varphi_{\alpha}-\varphi_1-\Psi+v_{t_j}(\Psi)}c(-v_{t_j}(\Psi))\\
   \le &\frac{C}{c(t_j+1)e^{-(t_j+1)}}
   \int_{U_0}\mathbb{I}_{\{-t_j-1<\Psi<-t_j\}}|f_j|^2e^{-\varphi_{\alpha}-\Psi}<+\infty.
\end{split}
\end{equation}
Note that $|F^2|^2e^{-\varphi_1}=1$ on $\{\Psi<-t_j\}$. As $v_{t_j}(\Psi)\ge \Psi$, we have $c(-v_{t_j}(\Psi))e^{v_{t_j}(\Psi)}\ge c(-\Psi)e^{-\Psi}$. Hence we have
\begin{equation}\label{convergence property of module formula 3}
\begin{split}
   &\int_{U_0}|\tilde{F}_j|^2e^{-\varphi_{\alpha}-\varphi_1}c(-\Psi) \\
   \le & 2\int_{U_0}|\tilde{F}_j-(1-b_{t_j}(\Psi))f_jF^2|^2e^{-\varphi_{\alpha}-\varphi_1}c(-\Psi)\\
   +&2\int_{U_0}|(1-b_{t_j}(\Psi))f_jF^2|^2e^{-\varphi_{\alpha}-\varphi_1}c(-\Psi)\\
   \le&
   2\int_{U_0}|\tilde{F}_j-(1-b_{t_j}(\Psi))f_jF^2|^2e^{-\varphi_{\alpha}-\varphi_1-\Psi+v_{t_j}(\Psi)}c(-v_{t_j}(\Psi))\\
   +&2\int_{U_0\cap\{\Psi<-t_j\}}|f_j|^2e^{-\varphi_{\alpha}}c(-\Psi)\\
   < &+\infty.
\end{split}
\end{equation}
Hence we know that $(\tilde{F}_j,o)\in \mathcal{H}_o$.

It follows from inequality \eqref{convergence property of module}, $\sup_{j\ge1}\left(\int_{U_0}\mathbb{I}_{\{-t_j-1<\Psi<-t_j\}}|f_j|^2e^{-\varphi_\alpha-\Psi}\right)<+\infty$ and inequality \eqref{convergence property of module formula 3} that we actually have $\sup_j\left(\int_{U_0}|\tilde{F}_j|^2e^{-\varphi_{\alpha}-\varphi_1}c(-\Psi)\right)<+\infty$. Note that $\varphi_1$ is a plurisubharmonic function and $c(-\Psi)e^{-\varphi_\alpha}$ has a positive lower bound on $U_0$. We have (shrink $U_0$ if necessary)
$$\sup_j\left(\int_{U_0}|\tilde{F}_j|^2\right)<+\infty.$$
Hence we know there exists a subsequence of $\{\tilde{F}_j\}_{j\ge 1}$ (also denoted by $\{\tilde{F}_j\}_{j\ge 1}$) compactly convergent to a holomorphic function $\tilde{F}_0$ on $U_0$.
It follows from Fatou's Lemma and inequality \eqref{convergence property of module formula 3} that
\begin{equation}\label{convergence property of module formula 4}
\int_{U_0}|\tilde{F}_0|^2e^{-\varphi_{\alpha}-\varphi_1}c(-\Psi)\le
\liminf_{j\to +\infty}\int_{U_0}|\tilde{F}_j|^2e^{-\varphi_{\alpha}-\varphi_1}c(-\Psi)<+\infty.
\end{equation}
As $f_j$ converges to $f_0$, it follows from Fatou's Lemma and inequality \eqref{convergence property of module formula 1} that
\begin{equation}\nonumber
  \begin{split}
  &\int_{U_0}|\tilde{F}_0-(1-b_{t_0}(\Psi))f_0F^{2}|^2e^{-\varphi_\alpha-\varphi_1+v_{t_0}(\Psi)-\Psi}c(-v_{t_0}(\Psi)) \\
     \le & \liminf_{j\to+\infty} \int_{U_0}|\tilde{F}_j-(1-b_{t_j}(\Psi))f_jF^{2}|^2e^{-\varphi_\alpha-\varphi_1+v_{t_j}(\Psi)-\Psi}c(-v_{t_j}(\Psi)) \\
     < &+\infty,
  \end{split}
\end{equation}
which implies that
\begin{equation}\label{convergence property of module formula 5}
  \begin{split}
\int_{U_0\cap\{\Psi<-t_0-1\}}|\tilde{F}_0-f_0F^2|^2e^{-\varphi_{\alpha}-\varphi_1-\Psi}<+\infty.
  \end{split}
\end{equation}
It follows from inequality \eqref{convergence property of module formula 2}, inequality \eqref{convergence property of module formula 3}, inequality \eqref{convergence property of module formula 4}, inequality \eqref{convergence property of module formula 5} and definition of $P:H_o/I(\varphi_\alpha+\Psi)_o\to \mathcal{H}_o/\mathcal{I}(\varphi_{\alpha}+\varphi_1+\Psi)_o$ that for any $j\ge 0$, we have
$$P([(f_j)_o])=[(\tilde{F}_j,o)].$$

As $(f_j-f)_o\in J_o$ for any $j\ge 1$, we have $(f_j-f_1)_o\in J_o$ for any $j\ge 1$.
It follows from Proposition \ref{module isomorphism} that there exists an ideal $\tilde{J}$ of $\mathcal{O}_{\mathbb{C}^n,o}$ such that $\mathcal{I}(\varphi_{\alpha}+\varphi_1+\Psi)_o\subset \tilde{J}\subset \mathcal{H}_o$ and $\tilde{J}/\mathcal{I}(\varphi_{\alpha}+\varphi_1+\Psi)_o=\text{Im}(P|_{J_o/I(\varphi_\alpha+\Psi)_o})$. It follows from $(f_j-f_1)_o\in J_o$ and $P([(f_j)_o])=[(F_j,o)]$ for any $j\ge 1$ that we have
$$(\tilde{F}_j-\tilde{F}_1)\in \tilde{J},$$
for any $j\ge 1$.

As $\tilde{F}_j$ compactly converges to $\tilde{F}_0$, using Lemma \ref{closedness}, we obtain that $(\tilde{F}_0-\tilde{F}_1,o)\in\tilde{J}$. Note that $P$ is an  $\mathcal{O}_{\mathbb{C}^n,o}$-module isomorphism and $\tilde{J}/\mathcal{I}(\varphi_{\alpha}+\varphi_1+\Psi)_o=\text{Im}(P|_{J_o/I(\varphi_\alpha+\Psi)_o})$. We have $(f_0-f_1)_o\in J_o$, which implies that $(f_0-f)_o\in J_o$.

Lemma \ref{closedness of module} is proved.
\end{proof}

Let $\varphi_{\alpha}$ be a plurisubharmonic function on $D$, and let $c\equiv1$. Note that $H_o=I(\varphi_{\alpha})_o$ and   $\mathcal{H}_o=\mathcal{I}(\varphi_{\alpha}+\varphi_1)_o$.
We know that $I(a\Psi+\varphi_{\alpha})_o\subset I(a'\Psi+\varphi_{\alpha})_o$ for any $0\le a'<a<+\infty$. Denote that $I_+(a\Psi+\varphi_{\alpha})_o:=\cup_{p>a}I(p\Psi+\varphi_{\alpha})_o$ is an $\mathcal{O}_{\mathbb{C}^n,o}$-submodule of $H_o$, where $a\ge0$.
\begin{Lemma}
	\label{l:m5} There exists $a'>a$ such that $I(a'\Psi+\varphi_{\alpha})_o=I_+(a\Psi+\varphi_{\alpha})_o$ for any $a\ge0$.
\end{Lemma}
\begin{proof}
By the definition of $I_+(a\Psi+\varphi_{\alpha})_o$, we know $I(p\Psi+\varphi_{\alpha})_o\subset I_+(a\Psi+\varphi_{\alpha})_o$ for any $p>a$. It suffices to prove that  there exists $a'>a$ such that $I_+(a\Psi+\varphi_{\alpha})_o\subset I(a'\Psi+\varphi_{\alpha})_o$.

 Let $k>a$ be an integer. Denote that $\tilde\varphi_1:=k\varphi_1=2\max\{k\psi+kT,2\log|F^k|\}$ and $\tilde \Psi:=k\Psi=\min\{k\psi-2\log|F^k|,-kT\}$. Proposition \ref{module isomorphism} shows that there exists an $\mathcal{O}_{\mathbb{C}^n,o}$-module isomorphism $P$ from $I(\varphi_{\alpha})_o/I(\varphi_{\alpha}+\tilde \Psi)_o\rightarrow \mathcal{I}(\varphi_{\alpha}+\varphi_1)_o/\mathcal{I}(\varphi_{\alpha}+\tilde\varphi_1+\tilde \Psi)_o$, which implies that there exists an ideal $K_p$ of $\mathcal{O}_{\mathbb{C}^n,o}$ such that
 $$P(I(\varphi_{\alpha}+p\Psi)_o/I(\varphi_{\alpha}+\tilde \Psi)_o)=K_p/\mathcal{I}(\varphi_{\alpha}+\tilde\varphi_1+\tilde \Psi)_o,$$
 where $p\in(0,k)$.
	Denote that
	$$L:=\cup_{a<p<k}K_p$$
	 be an ideal of $\mathcal{O}_{\mathbb{C}^n,o}$. Hence $P|_{I_+(a\Psi+\varphi_{\alpha})_o/I(\varphi_{\alpha}+\tilde \Psi)_o}$ is an  $\mathcal{O}_{\mathbb{C}^n,o}$-module isomorphism from $I_+(a\Psi+\varphi_{\alpha})_o/I(\varphi_{\alpha}+\tilde \Psi)_o$ to $L/\mathcal{I}(\varphi_{\alpha}+\tilde\varphi_1+\tilde \Psi)_o$. As $\mathcal{O}_{\mathbb{C}^n,o}$ is a Noetherian ring (see \cite{hormander}), we get that $L$ is finitely generated. Hence there exists a finite set $\{(f_1)_o,\ldots,(f_m)_o\}\subset I_+(a\Psi+\varphi_{\alpha})_o$, which satisfies that for any $f_o\in I_+(a\Psi+\varphi_{\alpha})_o$ there exists  $(h_j,o)\in\mathcal{O}_{\mathbb{C}^n,o}$ for any $1\le j\le m$ such that
	 $$f_o-\sum_{j=1}^{m}(h_j,o)\cdot (f_j)_o\in I(\varphi_{\alpha}+\tilde \Psi)_o.$$
	  By the definition of $I_+(a\Psi+\varphi_{\alpha})_o$, there exists $a'\in(a,k)$ such that $\{(f_1)_o,\ldots,(f_m)_o\}\subset I(a'\Psi+\varphi_{\alpha})_o$. As $(h_j,o)\cdot (f_j)_o\in I(a'\Psi+\varphi_{\alpha})_o$ for any $1\le j\le m$ and $I(\varphi_{\alpha}+\tilde\Psi)_o=I(k\Psi+\varphi_{\alpha})_o\subset I(a'\Psi+\varphi_{\alpha})_o$, we obtain that $I_+(a\Psi+\varphi_{\alpha})_o\subset I(a'\Psi+\varphi_{\alpha})_o$.
	
	  Thus, Lemma \ref{l:m5} holds.
\end{proof}

\section{Properties of $G(t)$}
Following the notations in Section \ref{sec:Main result}, we present some properties of the function $G(t)$ in this section.

\begin{Lemma}[see \cite{GY-concavity}]
	\label{l:converge}
	Let $M$ be a complex manifold. Let $S$ be an analytic subset of $M$.  	
	Let $\{g_j\}_{j=1,2,...}$ be a sequence of nonnegative Lebesgue measurable functions on $M$, which satisfies that $g_j$ are almost everywhere convergent to $g$ on  $M$ when $j\rightarrow+\infty$,  where $g$ is a nonnegative Lebesgue measurable function on $M$. Assume that for any compact subset $K$ of $M\backslash S$, there exist $s_K\in(0,+\infty)$ and $C_K\in(0,+\infty)$ such that
	$$\int_{K}{g_j}^{-s_K}dV_M\leq C_K$$
	 for any $j$, where $dV_M$ is a continuous volume form on $M$.
	
 Let $\{F_j\}_{j=1,2,...}$ be a sequence of holomorphic $(n,0)$ form on $M$. Assume that $\liminf_{j\rightarrow+\infty}\int_{M}|F_j|^2g_j\leq C$, where $C$ is a positive constant. Then there exists a subsequence $\{F_{j_l}\}_{l=1,2,...}$, which satisfies that $\{F_{j_l}\}$ is uniformly convergent to a holomorphic $(n,0)$ form $F$ on $M$ on any compact subset of $M$ when $l\rightarrow+\infty$, such that
 $$\int_{M}|F|^2g\leq C.$$
\end{Lemma}

Let $c(t)\in P_{T,M,\Psi}$.
The following lemma will be used to discuss the convergence property of holomorphic forms on $\{\Psi<-t\}$.
\begin{Lemma}\label{global convergence property of module}
 Let $f$ be a holomorphic $(n,0)$ form on $\{\Psi<-\hat{t}_0\}\cap V$, where $V\supset Z_0$ is an open subset of $M$ and $\hat{t}_0>T$
is an real number. For any $z_0\in Z_0$, let $J_{z_0}$ be an $\mathcal{O}_{M,z_0}$-submodule of $J(\Psi)_{z_0}$ such that $I\big(\Psi+\varphi_\alpha\big)_{z_0}\subset J_{z_0}$.

Let $\{f_j\}_{j\ge 1}$ be a sequence of holomorphic $(n,0)$ form on $\{\Psi<-t_j\}$. Assume that $t_0:=\lim_{j\to +\infty}t_j\in[T,+\infty)$,
\begin{equation}\label{global convergence property of module 1}
\limsup\limits_{j\to+\infty}\int_{\{\Psi<-t_j\}}|f_j|^2e^{-\varphi_\alpha}c(-\Psi)\le C<+\infty,
\end{equation}
and $(f_j-f)_{z_0}\in \mathcal{O} (K_M)_{z_0}\otimes J_{z_0}$ for any $z_0\in Z_0$. Then there exists a subsequence of $\{f_j\}_{j\in \mathbb{N}^+}$ compactly convergent to a holomorphic $(n,0)$ form $f_0$ on $\{\Psi<-t_0\}$ which satisfies
$$\int_{\{\Psi<-t_0\}}|f_0|^2e^{-\varphi_\alpha}c(-\Psi)\le C,$$
and $(f_0-f)_{z_0}\in \mathcal{O} (K_M)_{z_0}\otimes  J_{z_0}$ for any $z_0\in Z_0$.
\end{Lemma}
\begin{proof}
As $e^{-\varphi_\alpha}c(-\Psi)$ has positive lower bound on any compact subset of $M\backslash Z$, where $Z$ is some analytic subset of $M$, it follows from Lemma \ref{l:converge} that there exists a subsequence of $\{f_j\}_{j\in \mathbb{N}^+}$ (also denoted by $\{f_j\}_{j\in \mathbb{N}^+}$) that compactly convergent to a holomorphic $(n,0)$ form $f_0$ on $\{\Psi<-t_0\}$ which satisfies
$$\int_{\{\Psi<-t_0\}}|f_0|^2e^{-\varphi}c(-\Psi)\le\liminf_{j\to+\infty}\int_{\{\Psi<-t_j\}}|f_j|^2e^{-\varphi}c(-\Psi)\le C.$$

Next we prove $(f_0-f)_{z_0}\in \mathcal{O} (K_M)_{z_0}\otimes  J_{z_0}$ for any $z_0\in Z_0$.

For any $z_0\in Z_0\cap\{\Psi=-\infty\}$, we know that $\{f_j\}_{j\ge 0}$ and $f$ are holomorphic $(n,0)$ forms on some neighborhood $U_{z_0}$ of $z_0$. It is also easy to verify that  $J(\Psi)_{z_0}=I_{z_0}=\mathcal{O}_{M,z_0}$ and $J_{z_0}$ is an $\mathcal{O}_{M,z_0}$-submodule of $\mathcal{O}_{M,z_0}$. As $J_{z_0}\subset \mathcal{O}_{M,z_0}$ is an $\mathcal{O}_{M,z_0}$-submodule, it follows from Lemma \ref{closedness}, $(f_j-f)_{z_0}\in
\mathcal{O} (K_M)_{z_0} \otimes J_{z_0}$ and $\{f_j\}_{j\in \mathbb{N}^+}$ compactly converges to $f_0$ that we know $(f_0-f)_{z_0}\in
\mathcal{O} (K_M)_{z_0} \otimes J_{z_0}$, for any $z_0\in Z_0\cap\{\Psi=-\infty\}$.

Let $z_0\in Z_0\backslash\{\Psi=-\infty\}$. As $\limsup\limits_{j\to+\infty}\int_{\{\Psi<-t_j\}}|f_j|^2e^{-\varphi_\alpha}c(-\Psi)\le C<+\infty$, we know $(f_j-f_1)\in H_{z_0}$. The definition of $H_{z_0}$ can be referred to Section \ref{sec:properties of module}. It follows from $(f_j-f)_{z_0}\in J_{z_0}$ that we know $(f_j-f_1)_{z_0}\in J_{z_0}$. Hence we have $(f_j-f_1)\in H_{z_0}\cap J_{z_0}$. We note that $e^{-\varphi_\alpha}c(-\Psi)$ has a positive lower bound on some open neighborhood of $z_0$. It follows from inequality \eqref{global convergence property of module 1}, $(f_j-f_1)\in H_{z_0}\cap J_{z_0}$, the uniqueness of limit function and Lemma \ref{closedness of module} that we know $(f_0-f_1)_{z_0}\in
\mathcal{O} (K_M)_{z_0} \otimes (H_{z_0}\cap J_{z_0})$. Hence we know that $(f_0-f)_{z_0}\in
\mathcal{O} (K_M)_{z_0} \otimes J_{z_0}$ for any $z_0\in Z_0\backslash\{\Psi=-\infty\}$.

Now we have $(f_0-f)_{z_0}\in \mathcal{O} (K_M)_{z_0}\otimes J_{z_0}$ for any $z_0\in Z_0$.
Lemma \ref{global convergence property of module} is proved.
\end{proof}

\begin{Lemma}
\label{characterization of g(t)=0} Let $t_0>T$.
The following two statements are equivalent,\\
(1) $G(t_0)=0$;\\
(2) $f_{z_0}\in
\mathcal{O} (K_M)_{z_0} \otimes J_{z_0}$, for any  $ z_0\in Z_0$.
\end{Lemma}
\begin{proof}If $f_{z_0}\in
\mathcal{O} (K_M)_{z_0} \otimes J_{z_0}$, for any  $ z_0\in Z_0$, then take $\tilde{f}\equiv 0$ in the definition of $G(t)$ and we get $G(t_0)\equiv 0$.

If $G(t_0)=0$, by definition, there exists a sequence of holomorphic $(n,0)$ forms $\{f_j\}_{j\in\mathbb{Z}^+}$ on $\{\Psi<-t_0\}$ such that
\begin{equation}\label{estimate in G(t)=0}
\lim_{j\to+\infty}\int_{\{\Psi<-t_0\}}|f_j|^2e^{-\varphi}c(-\Psi)=0,
\end{equation}
 and $(f_j-f)_{z_0}\in
\mathcal{O} (K_M)_{z_0} \otimes J_{z_0}$, for any  $ z_0\in Z_0$ and $j\ge 1$. It follows from Lemma \ref{global convergence property of module} that there
exists a subsequence of $\{f_j\}_{j\in \mathbb{N}^+}$ compactly convergent to a holomorphic $(n,0)$ form $f_0$ on $\{\Psi<-t_0\}$ which satisfies
$$\int_{\{\Psi<-t_0\}}|f_0|^2e^{-\varphi_\alpha}c(-\Psi)=0$$
and
$(f_0-f)_{z_0}\in \mathcal{O} (K_M)_{z_0}\otimes J_{z_0}$ for any $z_0\in Z_0$. It follows from $\int_{\{\Psi<-t_0\}}|f_0|^2e^{-\varphi_\alpha}c(-\Psi)=0$ that we know $f_0\equiv 0$. Hence we have $f_{z_0}\in \mathcal{O} (K_M)_{z_0}\otimes J_{z_0}$ for any $z_0\in Z_0$. Statement (2) is proved.
\end{proof}

The following lemma shows the existence and uniqueness of the holomorphic $(n,0)$ form related to $G(t)$.
\begin{Lemma}
\label{existence of F}
Assume that $G(t)<+\infty$ for some $t\in [T,+\infty)$. Then there exists a unique
holomorphic $(n,0)$ form $F_t$ on $\{\Psi<-t\}$ satisfying
$$\ \int_{\{\Psi<-t\}}|F_t|^2e^{-\varphi_\alpha}c(-\Psi)=G(t)$$  and
$\ (F_t-f)\in
\mathcal{O} (K_M)_{z_0} \otimes J_{z_0}$, for any  $ z_0\in Z_0$.
\par
Furthermore, for any holomorphic $(n,0)$ form $\hat{F}$ on $\{\Psi<-t\}$ satisfying
$$\int_{\{\Psi<-t\}}|\hat{F}|^2e^{-\varphi_\alpha}c(-\Psi)<+\infty$$ and $\ (\hat{F}-f)\in
\mathcal{O} (K_M)_{z_0} \otimes J_{z_0}$, for any  $ z_0\in Z_0$. We have the following equality
\begin{equation}
\begin{split}
&\int_{\{\Psi<-t\}}|F_t|^2e^{-\varphi_\alpha}c(-\Psi)+
\int_{\{\Psi<-t\}}|\hat{F}-F_t|^2e^{-\varphi_\alpha}c(-\Psi)\\
=&\int_{\{\Psi<-t\}}|\hat{F}|^2e^{-\varphi_\alpha}c(-\Psi).
\label{orhnormal F}
\end{split}
\end{equation}
\end{Lemma}

\begin{proof} We firstly show the existence of $F_t$. As $G(t)<+\infty$, then there exists a sequence of holomorphic $(n,0)$ forms $\{f_j\}_{j\in \mathbb{N}^+}$ on $\{\Psi<-t\}$ such that $$\lim\limits_{j \to +\infty}\int_{\{\Psi<-t\}}|f_j|^2e^{-\varphi_\alpha}c(-\Psi)=G(t)$$ and $(f_j-f)\in
\mathcal{O} (K_M)_{z_0} \otimes J_{z_0}$, for any  $ z_0\in Z_0$ and any $j\ge 1$.
It follows from Lemma \ref{global convergence property of module} that there
exists a subsequence of $\{f_j\}_{j\in \mathbb{N}^+}$ compactly convergent to a holomorphic $(n,0)$ form $F$ on $\{\Psi<-t\}$ which satisfies
$$\int_{\{\Psi<-t\}}|F|^2e^{-\varphi_\alpha}c(-\Psi)\le G(t)$$
and
$(F-f)_{z_0}\in \mathcal{O} (K_M)_{z_0}\otimes J_{z_0}$ for any $z_0\in Z_0$. By the definition of $G(t)$, we have $\int_{\{\Psi<-t\}}|F|^2e^{-\varphi_\alpha}c(-\Psi)= G(t)$. Then we obtain the existence of $F_t(=F)$.

We prove the uniqueness of $F_t$ by contradiction: if not, there exist
two different holomorphic $(n,0)$ forms $f_1$ and $f_2$ on $\{\Psi<-t\}$
satisfying $\int_{\{\Psi<-t\}}|f_1|^2e^{-\varphi_\alpha}$  $c(-\Psi)=
\int_{\{\Psi<-t\}}|f_2|^2e^{-\varphi_\alpha}c(-\Psi)=G(t)$, $(f_1-f)_{z_0}\in \mathcal{O} (K_M)_{z_0}\otimes J_{z_0}$ for any $z_0\in Z_0$ and $(f_2-f)_{z_0}\in \mathcal{O} (K_M)_{z_0}\otimes J_{z_0}$ for any $z_0\in Z_0$. Note that
\begin{equation}\nonumber
\begin{split}
\int_{\{\Psi<-t\}}|\frac{f_1+f_2}{2}|^2e^{-\varphi_\alpha}c(-\Psi)+
\int_{\{\Psi<-t\}}|\frac{f_1-f_2}{2}|^2e^{-\varphi_\alpha}c(-\Psi)\\
=\frac{1}{2}(\int_{\{\Psi<-t\}}|f_1|^2e^{-\varphi_\alpha}c(-\Psi)+
\int_{\{\Psi<-t\}}|f_1|^2e^{-\varphi_\alpha}c(-\Psi))=G(t),
\end{split}
\end{equation}
then we obtain that
\begin{equation}\nonumber
\begin{split}
\int_{\{\Psi<-t\}}|\frac{f_1+f_2}{2}|^2e^{-\varphi_\alpha}c(-\Psi)
< G(t)
\end{split}
\end{equation}
and $(\frac{f_1+f_2}{2}-f)_{z_0}\in \mathcal{O} (K_M)_{z_0}\otimes J_{z_0}$ for any $z_0\in Z_0$, which contradicts to the definition of $G(t)$.

Now we prove the equality \eqref{orhnormal F}. Let $h$ be any holomorphic $(n,0)$ form on $\{\Psi<-t\}$
such that $\int_{\{\Psi<-t\}}|h|^2e^{-\varphi_\alpha}c(-\Psi)<+\infty$ and $h \in \mathcal{O} (K_M)_{z_0}\otimes J_{z_0}$ for any $z_0\in Z_0$.  It is clear that for any complex
number $\alpha$, $F_t+\alpha h$ satisfying $((F_t+\alpha h)-f) \in \mathcal{O} (K_M)_{z_0}\otimes J_{z_0}$ for any $z_0\in Z_0$ and
$\int_{\{\Psi<-t\}}|F_t|^2e^{-\varphi_\alpha}c(-\Psi) \leq \int_{\{\Psi<-t\}}|F_t+\alpha
h|^2e^{-\varphi_\alpha}c(-\Psi)$. Note that
\begin{equation}\nonumber
\begin{split}
\int_{\{\Psi<-t\}}|F_t+\alpha
h|^2e^{-\varphi_\alpha}c(-\Psi)-\int_{\{\psi<-t\}}|F_t|^2e^{-\varphi_\alpha}c(-\Psi)\geq 0
\end{split}
\end{equation}
(By considering $\alpha \to 0$) implies
\begin{equation}\nonumber
\begin{split}
\mathfrak{R} \int_{\{\Psi<-t\}}F_t\bar{h}e^{-\varphi_\alpha}c(-\Psi)=0,
\end{split}
\end{equation}
then we have
\begin{equation}\nonumber
\begin{split}
\int_{\{\Psi<-t\}}|F_t+h|^2e^{-\varphi_\alpha}c(-\Psi)=
\int_{\{\Psi<-t\}}(|F_t|^2+|h|^2)e^{-\varphi_\alpha}c(-\Psi).
\end{split}
\end{equation}
\par
Letting $h=\hat{F}-F_t$, we obtain equality \eqref{orhnormal F}.
\end{proof}

The following lemma shows the  lower semicontinuity property of $G(t)$.
\begin{Lemma}$G(t)$ is decreasing with respect to $t\in
[T,+\infty)$, such that $\lim \limits_{t \to t_0+0}G(t)=G(t_0)$ for any $t_0\in
[T,+\infty)$, and if $G(t)<+\infty$ for some $t>T$, then $\lim \limits_{t \to +\infty}G(t)=0$. Especially, $G(t)$ is lower semicontinuous on $[T,+\infty)$.
 \label{semicontinuous}
\end{Lemma}
\begin{proof}By the definition of $G(t)$, it is clear that $G(t)$ is decreasing on
$[T,+\infty)$. If $G(t)<+\infty$ for some $t>T$, by the dominated convergence theorem, we know $\lim\limits_{t\to +\infty}G(t)=0$. It suffices
to prove $\lim \limits_{t \to t_0+0}G(t)=G(t_0)$ . We prove it by
contradiction: if not, then $\lim \limits_{t \to t_0+0}G(t)<
G(t_0)$.

By using Lemma \ref{existence of F}, for any $t>t_0$, there exists a unique holomorphic $(n,0)$ form
$F_t$ on $\{\Psi<-t\}$ satisfying
$\int_{\{\Psi<-t\}}|F_t|^2e^{-\varphi_\alpha}c(-\Psi)=G(t)$ and $(F_t-f) \in \mathcal{O} (K_M)_{z_0}\otimes J_{z_0}$ for any $z_0\in Z_0$. Note that $G(t)$ is decreasing with respect to $t$. We have $\int_{\{\Psi<-t\}}|F_t|^2e^{-\varphi}c(-\psi)\leq \lim
\limits_{t \to t_0+0}G(t)$ for any $t>t_0$. If $\lim\limits_{t \to t_0+0}G(t)=+\infty$, the equality $\lim \limits_{t \to t_0+0}G(t)=G(t_0)$ obviously holds, thus it suffices to prove the case $\lim\limits_{t \to t_0+0}G(t)<+\infty$. It follows from $\int_{\{\Psi<-t\}}|F_t|^2e^{-\varphi_\alpha}c(-\Psi)\le \lim\limits_{t \to t_0+0}G(t)<+\infty$  holds for any $t\in (t_0,t_1]$ (where $t_1>t_0$ is a fixed number) and Lemma \ref{global convergence property of module} that there exists a subsequence of $\{F_t\}$ (denoted by $\{F_{t_j}\}$) compactly convergent to a holomorphic $(n,0)$ form $\hat{F}_{t_0}$ on $\{\Psi<-t_0\}$ satisfying
$$\int_{\{\Psi<-t_0\}}|\hat{F}_{t_0}|^2e^{-\varphi_\alpha}c(-\Psi)\le \lim\limits_{t \to t_0+0}G(t)<+\infty$$
and $(\hat{F}_{t_0}-f)_{z_0} \in \mathcal{O} (K_M)_{z_0}\otimes J_{z_0}$ for any $z_0\in Z_0$.

Then we obtain that $G(t_0)\leq
\int_{\{\Psi<-t_0\}}|\hat{F}_{t_0}|^2e^{-\varphi_\alpha}c(-\Psi)
\leq \lim \limits_{t\to t_0+0} G(t)$,
which contradicts $\lim \limits_{t\to t_0+0} G(t) <G(t_0)$. Thus we have $\lim \limits_{t \to t_0+0}G(t)=G(t_0)$.
\end{proof}

We consider the derivatives of $G(t)$ in the following lemma.

\begin{Lemma}
\label{derivatives of G}
Assume that $G(t_1)<+\infty$, where $t_1\in (T,+\infty)$. Then for any $t_0>t_1$, we have
\begin{equation}\nonumber
\begin{split}
\frac{G(t_1)-G(t_0)}{\int^{t_0}_{t_1} c(t)e^{-t}dt}\leq
\liminf\limits_{B \to
0+0}\frac{G(t_0)-G(t_0+B)}{\int_{t_0}^{t_0+B}c(t)e^{-t}dt},
\end{split}
\end{equation}
i.e.
\begin{equation}\nonumber
\frac{G(t_0)-G(t_1)}{\int_{T_1}^{t_0}
c(t)e^{-t}dt-\int_{T_1}^{t_1} c(t)e^{-t}dt} \geq
\limsup \limits_{B \to 0+0}
\frac{G(t_0+B)-G(t_0)}{\int_{T_1}^{t_0+B}
c(t)e^{-t}dt-\int_{T_1}^{t_0} c(t)e^{-t}dt}.
\end{equation}
\end{Lemma}

\begin{proof}
It follows from Lemma \ref{semicontinuous} that $G(t)<+\infty$ for any $t>t_1$. By Lemma \ref{existence of F}, there exists a holomorphic $(n,0)$ form $F_{t_0}$ on $\{\Psi<-t_0\}$, such that $(F_{t_0}-f)\in \mathcal{O} (K_M)_{z_0}\otimes J_{z_0}$ for any $z_0\in Z_0$ and $G(t_0)=\int_{\{\Psi<-t_0\}}|F_{t_0}|^2e^{-\varphi_\alpha}c(-\Psi)$.

It suffices to consider that $\liminf\limits_{B\to 0+0} \frac{G(t_0)-G(t_0+B)}{\int_{t_0}^{t_0+B}c(t)e^{-t}dt}\in [0,+\infty)$ because of the decreasing property of $G(t)$. Then there exists $1\ge B_j\to 0+0$ (as $j\to+\infty$) such that
\begin{equation}
	\label{derivatives of G c(t)1}
\lim\limits_{j\to +\infty} \frac{G(t_0)-G(t_0+B_j)}{\int_{t_0}^{t_0+B_j}c(t)e^{-t}dt}=\liminf\limits_{B\to 0+0} \frac{G(t_0)-G(t_0+B)}{\int_{t_0}^{t_0+B}c(t)e^{-t}dt}
\end{equation}
and $\{\frac{G(t_0)-G(t_0+B_j)}{\int_{t_0}^{t_0+B_j}c(t)e^{-t}dt}\}_{j\in\mathbb{N}^{+}}$ is bounded. As $c(t)e^{-t}$ is decreasing and positive on $(t,+\infty)$, then
\begin{equation}\label{derivatives of G c(t)2}
\begin{split}
\lim\limits_{j\to +\infty} \frac{G(t_0)-G(t_0+B_j)}{\int_{t_0}^{t_0+B_j}c(t)e^{-t}dt}
=&\big(\lim\limits_{j\to +\infty} \frac{G(t_0)-G(t_0+B_j)}{B_j})(\frac{1}{\lim\limits_{t\to t_0+0}c(t)e^{-t}}\big)\\
=&\big(\lim\limits_{j\to +\infty} \frac{G(t_0)-G(t_0+B_j)}{B_j})(\frac{e^{t_0}}{\lim\limits_{t\to t_0+0}c(t)}\big).
\end{split}
\end{equation}
Hence $\{\frac{G(t_0)-G(t_0+B_j)}{B_j}\}_{j\in\mathbb{N}^+}$ is uniformly bounded with respect to $j$.

As $t \leq v_{t_0,j}(t)$, the decreasing property of $c(t)e^{-t}$ shows that
\begin{equation}\nonumber
e^{-\Psi+v_{t_0,B_j}(\Psi)}c(-v_{t_0,B_j}(\Psi))\geq c(-\Psi).
\end{equation}
\par
It follows from Lemma \ref{L2 method in JM concavity} that, for any $B_j$, there exists holomorphic
$(n,0)$ form $\tilde{F}_j$ on $\{\Psi<-t_1\}$ such that
\begin{flalign}
&\int_{\{\Psi<-t_1\}}|\tilde{F}_j-(1-b_{t_0,B_j}(\Psi))F_{t_0}|^2e^{-\varphi_\alpha}c(-\Psi)\nonumber\\
\leq &
\int_{\{\Psi<-t_1\}}|\tilde{F}_j-(1-b_{t_0,B_j}(\Psi))F_{t_0}|^2e^{-\varphi_\alpha}e^{-\Psi+v_{t_0,B_j}(\Psi)}c(-v_{t_0,B_j}(\Psi))\nonumber\\
\leq &
\int^{t_0+B_j}_{t_1}c(t)e^{-t}dt\int_{\{\Psi<-t_1\}}\frac{1}{B_j}
\mathbb{I}_{\{-t_0-B_j<\Psi<-t_0\}}|F_{t_0}|^2e^{-\varphi_\alpha-\Psi}\nonumber\\
\leq &
\frac{e^{t_0+B_j}\int^{t_0+B_j}_{t_1}c(t)e^{-t}dt}{\inf
\limits_{t\in(t_0,t_0+B_j)}c(t)}\int_{\{\Psi<-t_1\}}\frac{1}{B_j}
\mathbb{I}_{\{-t_0-B_j<\Psi<-t_0\}}|F_{t_0}|^2e^{-\varphi_\alpha}c(-\Psi)\nonumber\\
= &
\frac{e^{t_0+B_j}\int^{t_0+B_j}_{t_1}c(t)e^{-t}dt}{\inf
\limits_{t\in(t_0,t_0+B_j)}c(t)}\times
\bigg(\int_{\{\Psi<-t_1\}}\frac{1}{B_j}\mathbb{I}_{\{\Psi<-t_0\}}|F_{t_0}|^2e^{-\varphi_\alpha}c(-\Psi)\nonumber\\
&-\int_{\{\Psi<-t_1\}}\frac{1}{B_j}\mathbb{I}_{\{\Psi<-t_0-B_j\}}|F_{t_0}|^2e^{-\varphi_\alpha}c(-\Psi)\bigg)\nonumber\\
\leq &
\frac{e^{t_0+B_j}\int^{t_0+B_j}_{t_1}c(t)e^{-t}dt}{\inf
\limits_{t\in(t_0,t_0+B_j)}c(t)} \times
\frac{G(t_0)-G(t_0+B_j)}{B_j}<+\infty.
\label{derivative of G 1}
\end{flalign}

Note that $b_{t_0,B_j}(t)=0$ for $t\le-t_0-B_j$, $b_{t_0,B_j}(t)=1$ for $t\ge t_0$, $v_{t_0,B_j}(t)>-t_0-B_j$ and $c(t)e^{-t}$ is decreasing with respect to $t$. It follows from inequality \eqref{derivative of G 1} that $(F_j-F_{t_0})_{z_0}\in \mathcal{O} (K_M)_{z_0}\otimes I(\Psi+\varphi_\alpha)_{z_0} \subset \mathcal{O} (K_M)_{z_0}\otimes J_{z_0}$ for any $z_0\in Z_0$.

Note that
\begin{equation}\label{derivative of G 2}
\begin{split}
&\int_{\{\Psi<-t_1\}}|\tilde{F}_j|^2e^{-\varphi_\alpha}c(-\Psi)\\
\le&2\int_{\{\Psi<-t_1\}}|\tilde{F}_j-(1-b_{t_0,B_j}(\Psi))F_{t_0}|^2e^{-\varphi_\alpha}c(-\Psi)
+2\int_{\{\Psi<-t_1\}}|(1-b_{t_0,B_j}(\Psi))F_{t_0}|^2e^{-\varphi_\alpha}c(-\Psi)\\
\le&2
\frac{e^{t_0+B_j}\int^{t_0+B_j}_{t_1}c(t)e^{-t}dt}{\inf
\limits_{t\in(t_0,t_0+B_j)}c(t)} \times
\frac{G(t_0)-G(t_0+B_j)}{B_j}
+2\int_{\{\Psi<-t_0\}}|F_{t_0}|^2e^{-\varphi_\alpha}c(-\Psi).
\end{split}
\end{equation}

We also note that $B_j\le 1$, $\frac{G(t_0)-G(t_0+B_j)}{B_j}$ is uniformly bounded with respect to $j$ and $G(t_0)=\int_{\{\Psi<-t_0\}}|F_{t_0}|^2e^{-\varphi_\alpha}c(-\Psi)$. It follows from inequality \eqref{derivative of G 2} that we know $\int_{\{\Psi<-t_1\}}|\tilde{F}_j|^2e^{-\varphi_\alpha}c(-\Psi)$ is uniformly bounded with respect to $j$.

It follows from Lemma \ref{global convergence property of module} that there exists a subsequence of $\{\tilde{F}_j\}_{j\in \mathbb{N}^+}$ compactly convergent to a holomorphic $(n,0)$ form $\tilde{F}_{t_1}$ on $\{\Psi<-t_1\}$ which satisfies
$$\int_{\{\Psi<-t_1\}}|\tilde{F}_{t_1}|^2e^{-\varphi_\alpha}c(-\Psi)\le \liminf_{j\to+\infty} \int_{\{\Psi<-t_1\}}|\tilde{F}_j|^2e^{-\varphi_\alpha}c(-\Psi)<+\infty,$$
and $(\tilde{F}_{t_1}-F_{t_0})_{z_0}\in \mathcal{O} (K_M)_{z_0}\otimes  J_{z_0}$ for any $z_0\in Z_0$.

Note that
$\lim_{j\to+\infty}b_{t_0,B_j}(t)=\mathbb{I}_{\{t\ge -t_0\}}$ and
\begin{equation}\nonumber
v_{t_0}(t):=\lim_{j\to+\infty}v_{t_0,B_j}(t)=\left\{
\begin{aligned}
&-t_0  &\text{ if } & x<-t_0, \\
&\ t  &\text{ if }  & x\ge t_0 .
\end{aligned}
\right.
\end{equation}

It follows from inequality \eqref{derivative of G 1} and Fatou's lemma that

\begin{flalign}
\label{derivative of G 3}
&\int_{\{\Psi<-t_0\}}|\tilde{F}_{t_1}-F_{t_0}|^2e^{-\varphi_\alpha}c(-\Psi)
+\int_{\{-t_0\le\Psi<-t_1\}}|\tilde{F}_{t_1}|^2e^{-\varphi_\alpha}c(-\Psi)\nonumber\\
\leq &
\int_{\{\Psi<-t_1\}}|\tilde{F}_{t_1}-\mathbb{I}_{\{\Psi< -t_0\}}F_{t_0}|^2e^{-\varphi_\alpha}e^{-\Psi+v_{t_0}(\Psi)}c(-v_{t_0}(\Psi))\nonumber\\
\le&\liminf_{j\to+\infty}\int_{\{\Psi<-t_1\}}|\tilde{F}_j-(1-b_{t_0,B_j}(\Psi))F_{t_0}|^2e^{-\varphi_\alpha}c(-\Psi)\nonumber\\
\leq &\liminf_{j\to+\infty}
\bigg(\frac{e^{t_0+B_j}\int^{t_0+B_j}_{t_1}c(t)e^{-t}dt}{\inf
\limits_{t\in(t_0,t_0+B_j)}c(t)} \times
\frac{G(t_0)-G(t_0+B_j)}{B_j}\bigg).
\end{flalign}

It follows from Lemma \ref{existence of F}, equality \eqref{derivatives of G c(t)1}, equality \eqref{derivatives of G c(t)2} and inequality \eqref{derivative of G 3} that we have

\begin{equation}
\label{derivative of G 4}
\begin{split}
&\int_{\{\Psi<-t_1\}}|\tilde{F}_{t_1}|^2e^{-\varphi_\alpha}c(-\Psi)
-\int_{\{\Psi<-t_0\}}|F_{t_0}|^2e^{-\varphi_\alpha}c(-\Psi)\\
\le&\int_{\{\Psi<-t_0\}}|\tilde{F}_{t_1}-F_{t_0}|^2e^{-\varphi_\alpha}c(-\Psi)
+\int_{\{-t_0\le\Psi<-t_1\}}|\tilde{F}_{t_1}|^2e^{-\varphi_\alpha}c(-\Psi)\\
\leq &
\int_{\{\Psi<-t_1\}}|\tilde{F}_{t_1}-\mathbb{I}_{\{\Psi< -t_0\}}F_{t_0}|^2e^{-\varphi_\alpha}e^{-\Psi+v_{t_0}(\Psi)}c(-v_{t_0}(\Psi))\\
\le&\liminf_{j\to+\infty}\int_{\{\Psi<-t_1\}}|\tilde{F}_j-(1-b_{t_0,B_j}(\Psi))F_{t_0}|^2e^{-\varphi_\alpha}c(-\Psi)\\
\leq &\liminf_{j\to+\infty}
\big(\frac{e^{t_0+B_j}\int^{t_0+B_j}_{t_1}c(t)e^{-t}dt}{\inf
\limits_{t\in(t_0,t_0+B_j)}c(t)} \times
\frac{G(t_0)-G(t_0+B_j)}{B_j}\big)\\
\le &\bigg(\int^{t_0}_{t_1}c(t)e^{-t}dt\bigg)\liminf\limits_{B\to 0+0} \frac{G(t_0)-G(t_0+B)}{\int_{t_0}^{t_0+B}c(t)e^{-t}dt}.
\end{split}
\end{equation}
Note that $(\tilde{F}_{t_1}-F_{t_0})_{z_0}\in \mathcal{O} (K_M)_{z_0}\otimes  J_{z_0}$ for any $z_0\in Z_0$. It follows from the definition of $G(t)$ and inequality \eqref{derivative of G 4} that we have

\begin{equation}
\label{derivative of G 5}
\begin{split}
&G(t_1)-G(t_0)\\
\le&\int_{\{\Psi<-t_1\}}|\tilde{F}_{t_1}|^2e^{-\varphi_\alpha}c(-\Psi)
-\int_{\{\Psi<-t_0\}}|F_{t_0}|^2e^{-\varphi_\alpha}c(-\Psi)\\
\le&\int_{\{\Psi<-t_1\}}|\tilde{F}_{t_1}-\mathbb{I}_{\{\Psi< -t_0\}}F_{t_0}|^2e^{-\varphi_\alpha}c(-\Psi)\\
\leq &
\int_{\{\Psi<-t_1\}}|\tilde{F}_{t_1}-\mathbb{I}_{\{\Psi< -t_0\}}F_{t_0}|^2e^{-\varphi_\alpha}e^{-\Psi+v_{t_0}(\Psi)}c(-v_{t_0}(\Psi))\\\
\le &\big(\int^{t_0}_{t_1}c(t)e^{-t}dt\big)\liminf\limits_{B\to 0+0} \frac{G(t_0)-G(t_0+B)}{\int_{t_0}^{t_0+B}c(t)e^{-t}dt}.
\end{split}
\end{equation}

Lemma \ref{derivatives of G} is proved.

\end{proof}

The following property of concave
functions will be used in the proof of Theorem \ref{main theorem}.
\begin{Lemma}[see \cite{G16}]
Let $H(r)$ be a lower semicontinuous function on $(0,R]$. Then $H(r)$ is concave
if and only if
\begin{equation}\nonumber
\begin{split}
\frac{H(r_1)-H(r_2)}{r_1-r_2} \leq
\liminf\limits_{r_3 \to r_2-0}
\frac{H(r_3)-H(r_2)}{r_3-r_2}
\end{split}
\end{equation}
holds for any $0<r_2<r_1 \leq R$.
\label{characterization of concave function}
\end{Lemma}

\section{Proofs of Theorem \ref{main theorem}, Remark \ref{infty2}, Corollary \ref{necessary condition for linear of G} and Remark \ref{rem:linear}}
We firstly prove Theorem \ref{main theorem}.

\begin{proof} We firstly show that if $G(t_0)<+\infty$ for some $t_0> T$, then $G(t_1)<+\infty$ for any $T< t_1<t_0$. As $G(t_0)<+\infty$, it follows from Lemma \ref{existence of F} that there exists a unique
holomorphic $(n,0)$ form $F_{t_0}$ on $\{\Psi<-t\}$ satisfying
$$\ \int_{\{\Psi<-t_0\}}|F_{t_0}|^2e^{-\varphi_\alpha}c(-\Psi)=G(t_0)<+\infty$$  and
$\ (F_{t_0}-f)_{z_0}\in
\mathcal{O} (K_M)_{z_0} \otimes J_{z_0}$, for any  $ z_0\in Z_0$.

It follows from Lemma \ref{L2 method in JM concavity} that there exists a holomorphic $(n,0)$ form $\tilde{F}_1$ on $\{\Psi<-t_1\}$ such that

 \begin{equation}\label{main theorem 1}
  \begin{split}
      & \int_{\{\Psi<-t_1\}}|\tilde{F}_1-(1-b_{t_0,B}(\Psi))F_{t_0}|^2e^{-\varphi_\alpha+v_{t_0,B}(\Psi)-\Psi}c(-v_{t_0,B}(\Psi)) \\
      \le & (\int_{t_1}^{t_0+B}c(s)e^{-s}ds)
       \int_{M}\frac{1}{B}\mathbb{I}_{\{-t_0-B<\Psi<-t_0\}}|F_{t_0}|^2e^{-\varphi_{\alpha}-\Psi}<+\infty.
  \end{split}
\end{equation}
Note that $b_{t_0,B}(t)=0$ on $\{\Psi<-t_0-B\}$ and $v_{t_0,B}(\Psi)> -t_0-B$. We have $e^{v_{t_0,B}(\Psi)}c(-v_{t_0,B}(\Psi))$ has a positive lower bound. It follows from inequality \eqref{main theorem 1} that
we have  $\ (F_{1}-F_{t_0})_{z_0}
\in \mathcal{O} (K_M)_{z_0}\otimes I(\Psi+\varphi_\alpha)_{z_0} \subset \mathcal{O} (K_M)_{z_0}\otimes J_{z_0}$ for any $z_0\in Z_0$, which implies that
$(F_{1}-f)_{z_0}\in
\mathcal{O} (K_M)_{z_0} \otimes J_{z_0}$, for any  $ z_0\in Z_0$. As $v_{t_0,B}(\Psi)\ge\Psi$ and $c(t)e^{-t}$ is decreasing with respect to $t$, it follows from inequality \eqref{main theorem 1} that we have
 \begin{equation}\label{main theorem 2}
  \begin{split}
      & \int_{\{\Psi<-t_1\}}|\tilde{F}_1-(1-b_{t_0,B}(\Psi))F_{t_0}|^2e^{-\varphi_\alpha}c(-\Psi)
      \\
      \le&\int_{\{\Psi<-t_1\}}|\tilde{F}_1-(1-b_{t_0,B}(\Psi))F_{t_0}|^2e^{-\varphi_\alpha+v_{t_0,B}(\Psi)-\Psi}c(-v_{t_0,B}(\Psi)) \\
      \le & (\int_{t_1}^{t_0+B}c(s)e^{-s}ds)
       \int_{M}\frac{1}{B}\mathbb{I}_{\{-t_0-B<\Psi<-t_0\}}|F_{t_0}|^2e^{-\varphi_{\alpha}-\Psi}<+\infty.
  \end{split}
\end{equation}
Then we have
 \begin{equation}\label{main theorem 3}
  \begin{split}
  &\int_{\{\Psi<-t_1\}}|\tilde{F}_1|^2e^{-\varphi_\alpha}c(-\Psi)\\
     \le & 2\int_{\{\Psi<-t_1\}}|\tilde{F}_1-(1-b_{t_0,B}(\Psi))F_{t_0}|^2e^{-\varphi_\alpha}c(-\Psi)
      +2\int_{\{\Psi<-t_1\}}|(1-b_{t_0,B}(\Psi))F_{t_0}|^2e^{-\varphi_\alpha}c(-\Psi)\\
      \le & 2(\int_{t_1}^{t_0+B}c(s)e^{-s}ds)
       \int_{M}\frac{1}{B}\mathbb{I}_{\{-t_0-B<\Psi<-t_0\}}|F_{t_0}|^2e^{-\varphi_{\alpha}-\Psi}
       +2\int_{\{\Psi<-t_0\}}|F_{t_0}|^2e^{-\varphi_\alpha}c(-\Psi)\\
       <&+\infty.
  \end{split}
\end{equation}
Hence we have $G(t_1)\le \int_{\{\Psi<-t_1\}}|\tilde{F}_1|^2e^{-\varphi_\alpha}c(-\Psi)<+\infty$.

Now, it follows from Lemma \ref{semicontinuous}, Lemma \ref{derivatives of G} and Lemma \ref{characterization of concave function} that we know $G(h^{-1}(r))$ is concave with respect to $r$. It follows from Lemma \ref{semicontinuous} that $\lim\limits_{t\to T+0}G(t)=G(T)$ and $\lim\limits_{t\to+\infty}G(t)=0$.

Theorem \ref{main theorem} is proved.

\end{proof}

Now we prove Remark \ref{infty2}.
\begin{proof}Note that if there exists a positive decreasing concave function $g(t)$ on $(a,b)\subset\mathbb{R}$ and $g(t)$ is not a constant function, then $b<+\infty$.

Assume that $G(t_0)<+\infty$ for some $t_0\geq T$. As $f_{z_0}\notin
\mathcal{O} (K_M)_{z_0} \otimes J_{z_0}$ for some  $ z_0\in Z_0$, Lemma \ref{characterization of g(t)=0} shows that $G(t_0)\in(0,+\infty)$. Following from Theorem \ref{main theorem} we know $G({h}^{-1}(r))$ is concave with respect to $r\in(\int_{T_1}^{T}c(t)e^{-t}dt,\int_{T_1}^{+\infty}c(t)e^{-t}dt)$ and $G({h}^{-1}(r))$ is not a constant function, therefore we obtain $\int_{T_1}^{+\infty}c(t)e^{-t}dt<+\infty$, which contradicts to $\int_{T_1}^{+\infty}c(t)e^{-t}dt=+\infty$. Thus we have $G(t)\equiv+\infty$.

When $G(t_2)\in(0,+\infty)$ for some $t_2\in[T,+\infty)$, Lemma \ref{characterization of g(t)=0} shows that $f_{z_0}\notin
\mathcal{O} (K_M)_{z_0} \otimes J_{z_0}$, for any  $ z_0\in Z_0$. Combining the above discussion, we know $\int_{T_1}^{+\infty}c(t)e^{-t}dt<+\infty$. Using Theorem \ref{main theorem}, we obtain that $G(\hat{h}^{-1}(r))$ is concave with respect to  $r\in (0,\int_{T}^{+\infty}c(t)e^{-t}dt)$, where $\hat{h}(t)=\int_{t}^{+\infty}c(l)e^{-l}dl$.

Thus, Remark \ref{infty2} holds.
\end{proof}

Now we prove Corollary \ref{necessary condition for linear of G}

\begin{proof} As $G(h^{-1}(r))$ is linear with respect to $r\in[0,\int_T^{+\infty}c(s)e^{-s}ds)$, we have $G(t)=\frac{G(T_1)}{\int_{T_1}^{+\infty}c(s)e^{-s}ds}\int_{t}^{+\infty}c(s)e^{-s}ds$ for any $t\in[T,+\infty)$ and $T_1 \in (T,+\infty)$.

We follow the notation and the construction in Lemma \ref{derivatives of G}. Let $t_0>t_1> T$ be given. It follows from $G(h^{-1}(r))$ is linear with respect to $r\in[0,\int_T^{+\infty}c(s)e^{-s}ds)$ that we know that all inequalities in \eqref{derivative of G 5} should be equalities, i.e., we have

\begin{equation}
\label{necessary condition for linear of G 1}
\begin{split}
&G(t_1)-G(t_0)\\
=&\int_{\{\Psi<-t_1\}}|\tilde{F}_{t_1}|^2e^{-\varphi_\alpha}c(-\Psi)
-\int_{\{\Psi<-t_0\}}|F_{t_0}|^2e^{-\varphi_\alpha}c(-\Psi)\\
=&\int_{\{\Psi<-t_1\}}|\tilde{F}_{t_1}-\mathbb{I}_{\{\Psi< -t_0\}}F_{t_0}|^2e^{-\varphi_\alpha}c(-\Psi)\\
= &
\int_{\{\Psi<-t_1\}}|\tilde{F}_{t_1}-\mathbb{I}_{\{\Psi< -t_0\}}F_{t_0}|^2e^{-\varphi_\alpha}e^{-\Psi+v_{t_0}(\Psi)}c(-v_{t_0}(\Psi))\\
= &\big(\int^{t_0}_{t_1}c(t)e^{-t}dt\big)\liminf\limits_{B\to 0+0} \frac{G(t_0)-G(t_0+B)}{\int_{t_0}^{t_0+B}c(t)e^{-t}dt}.
\end{split}
\end{equation}
Note that $G(t_0)=\int_{\{\Psi<-t_0\}}|F_{t_0}|^2e^{-\varphi_\alpha}c(-\Psi)$. Equality \eqref{necessary condition for linear of G 1} shows that $G(t_1)=\int_{\{\Psi<-t_1\}}|\tilde{F}_{t_1}|^2e^{-\varphi_\alpha}c(-\Psi)$.

Note that on $\{\Psi\ge -t_0\}$, we have $e^{-\Psi+v_{t_0}(\Psi)}c(-v_{t_0}(\Psi))=c(-\Psi)$. It follows from
\begin{equation}\nonumber
\begin{split}
&\int_{\{\Psi<-t_1\}}|\tilde{F}_{t_1}-\mathbb{I}_{\{\Psi< -t_0\}}F_{t_0}|^2e^{-\varphi_\alpha}c(-\Psi)\\
= &
\int_{\{\Psi<-t_1\}}|\tilde{F}_{t_1}-\mathbb{I}_{\{\Psi< -t_0\}}F_{t_0}|^2e^{-\varphi_\alpha}e^{-\Psi+v_{t_0}(\Psi)}c(-v_{t_0}(\Psi))\\
\end{split}
\end{equation}
that we have (note that $v_{t_0}(\Psi)=-t_0$ on  $\{\Psi< -t_0\}$)
\begin{equation}
\begin{split}
\label{necessary condition for linear of G 2}
&\int_{\{\Psi<-t_0\}}|\tilde{F}_{t_1}-F_{t_0}|^2e^{-\varphi_\alpha}c(-\Psi)\\
= &
\int_{\{\Psi<-t_0\}}|\tilde{F}_{t_1}-F_{t_0}|^2e^{-\varphi_\alpha}e^{-\Psi-t_0}c(t_0).\\
\end{split}
\end{equation}
As $\int_{T}^{+\infty}c(t)e^{-t}dt<+\infty$ and $c(t)e^{-t}$ is decreasing with respect to $t$, we know that there exists $t_2>t_0$ such that $c(t)e^{-t}<c(t_0)e^{-t_0}-\epsilon$ for any $t\ge t_2$, where $\epsilon>0$ is a constant. Then equality \eqref{necessary condition for linear of G 2} implies that
\begin{equation}
\begin{split}
\label{necessary condition for linear of G 3}
&\epsilon\int_{\{\Psi<-t_2\}}|\tilde{F}_{t_1}-F_{t_0}|^2e^{-\varphi_\alpha-\Psi}\\
\le &
\int_{\{\Psi<-t_2\}}|\tilde{F}_{t_1}-F_{t_0}|^2e^{-\varphi_\alpha}(e^{-\Psi-t_0}c(t_0)-c(-\Psi))\\
\le &
\int_{\{\Psi<-t_0\}}|\tilde{F}_{t_1}-F_{t_0}|^2e^{-\varphi_\alpha}(e^{-\Psi-t_0}c(t_0)-c(-\Psi))\\
= &0.
\end{split}
\end{equation}
Note that $e^{-\varphi_\alpha-\Psi}\ge e^{-(\varphi_\alpha+\psi)}|F|^2$, $\varphi_\alpha+\psi$ is a plurisubharmonic function and the integrand in \eqref{necessary condition for linear of G 3} is nonnegative, we must have $\tilde{F}_{t_1}|_{\{\Psi<-t_0\}}=F_{t_0}$.

It follows from Lemma \ref{existence of F} that for any $t>T$, there exists a unique holomorphic $(n,0)$ form $F_t$ on $\{\Psi<-t\}$ satisfying
$$\ \int_{\{\Psi<-t\}}|F_t|^2e^{-\varphi_\alpha}c(-\Psi)=G(t)$$  and
$\ (F_t-f)\in
\mathcal{O} (K_M)_{z_0} \otimes J_{z_0}$, for any  $ z_0\in Z_0$.  By the above discussion, we know $F_{t}=F_{t'}$ on $\{\Psi<-\max{\{t,t'\}}\}$ for any $t\in(T,+\infty)$ and $t'\in(T,+\infty)$. Hence combining $\lim_{t\rightarrow T+0}G(t)=G(T)$, we obtain that there  exists a unique holomorphic $(n,0)$ form $\tilde{F}$ on $\{\Psi<-T\}$ satisfying $(\tilde{F}-f)_{z_0}\in
\mathcal{O} (K_M)_{z_0} \otimes J_{z_0}$ for any  $z_0\in Z_0$ and $G(t)=\int_{\{\Psi<-t\}}|\tilde{F}|^2e^{-\varphi_{\alpha}}c(-\Psi)$ for any $t\ge T$.

Secondly, we prove equality \eqref{other a also linear}.

As $a(t)$ is a nonnegative measurable function on $(T,+\infty)$, then there exists a sequence of functions $\{\sum\limits_{j=1}^{n_i}a_{ij}\mathbb{I}_{E_{ij}}\}_{i\in\mathbb{N}^+}$ $(n_i<+\infty$ for any $i\in\mathbb{N}^+)$ satisfying that $\sum\limits_{j=1}^{n_i}a_{ij}\mathbb{I}_{E_{ij}}$ is increasing with respect to $i$ and $\lim\limits_{i\to +\infty}\sum\limits_{j=1}^{n_i}a_{ij}\mathbb{I}_{E_{ij}}=a(t)$ for any $t\in(T,+\infty)$, where $E_{ij}$ is a Lebesgue measurable subset of $(T,+\infty)$ and $a_{ij}\ge 0$ is a constant for any $i,j$.  It follows from Levi's Theorem that it suffices to prove the case that $a(t)=\mathbb{I}_{E}(t)$, where $E\subset\subset (T,+\infty)$ is a Lebesgue measurable set.

Note that $G(t)=\int_{\{\Psi<-t\}}|\tilde{F}|^2e^{-\varphi_\alpha}c(-\Psi)=\frac{G(T_1)}{\int_{T_1}^{+\infty}c(s)e^{-s}ds}
\int_{t}^{+\infty}c(s)e^{-s}ds$ where $T_1 \in (T,+\infty)$, then
  \begin{equation}\label{linear 3.4}
\int_{\{-t_1\le\Psi<-t_2\}}|\tilde{F}|^2e^{-\varphi_\alpha}c(-\Psi)=\frac{G(T_1)}{\int_{T_1}^{+\infty}c(s)e^{-s}ds}
\int_{t_2}^{t_1}c(s)e^{-s}ds
  \end{equation}
  holds for any $T\le t_2<t_1<+\infty$. It follows from the dominated convergence theorem and equality \eqref{linear 3.4} that
    \begin{equation}\label{linear 3.5}
\int_{\{z\in M:-\Psi(z)\in N\}}|\tilde{F}|^2e^{-\varphi_\alpha}=0
  \end{equation}
  holds for any $N\subset\subset (T,+\infty)$ such that $\mu(N)=0$, where $\mu$ is the Lebesgue measure on $\mathbb{R}$.

  As $c(t)e^{-t}$ is decreasing on $(T,+\infty)$, there are at most countable points denoted by $\{s_j\}_{j\in \mathbb{N}^+}$ such that $c(t)$ is not continuous at $s_j$. Then there is a decreasing sequence of open sets $\{U_k\}$,  such that
$\{s_j\}_{j\in \mathbb{N}^+}\subset U_k\subset (T,+\infty)$ for any $k$, and $\lim\limits_{k \to +\infty}\mu(U_k)=0$. Choosing any closed interval $[t'_2,t'_1]\subset (T,+\infty)$, then we have
\begin{equation}\label{linear 3.6}
\begin{split}
&\int_{\{-t'_1\le\Psi<-t'_2\}}|\tilde{F}|^2e^{-\varphi_\alpha}\\
=&\int_{\{z\in M:-\Psi(z)\in(t'_2,t'_1]\backslash U_k\}}|\tilde{F}|^2e^{-\varphi_\alpha}+
\int_{\{z\in M:-\Psi(z)\in[t'_2,t'_1]\cap U_k\}}|\tilde{F}|^2e^{-\varphi_\alpha}\\
=&\lim_{n\to+\infty}\sum_{i=0}^{n-1}\int_{\{z\in M:-\Psi(z)\in I_{i,n}\backslash U_k\}}|\tilde{F}|^2e^{-\varphi_\alpha}+
\int_{\{z\in M:-\Psi(z)\in[t'_2,t'_1]\cap U_k\}}|\tilde{F}|^2e^{-\varphi_\alpha},
\end{split}
\end{equation}
where $I_{i,n}=(t'_1-(i+1)\alpha_n,t'_1-i\alpha_n]$ and $\alpha_n=\frac{t'_1-t'_2}{n}$. Note that
\begin{equation}\label{linear 3.7}
\begin{split}
&\lim_{n\to+\infty}\sum_{i=0}^{n-1}\int_{\{z\in M:-\Psi(z)\in I_{i,n}\backslash U_k\}}|\tilde{F}|^2e^{-\varphi_\alpha}\\
\le&\limsup_{n\to+\infty}\sum_{i=0}^{n-1}\frac{1}{\inf_{I_{i,n}\backslash U_k}c(t)}\int_{\{z\in M:-\Psi(z)\in I_{i,n}\backslash U_k\}}|\tilde F|^2e^{-\varphi_\alpha}c(-\Psi).
\end{split}
\end{equation}

It follows from equality \eqref{linear 3.4} that inequality \eqref{linear 3.7} becomes
\begin{equation}\label{linear 3.8}
\begin{split}
&\lim_{n\to+\infty}\sum_{i=0}^{n-1}\int_{\{z\in M:-\Psi(z)\in I_{i,n}\backslash U_k\}}|\tilde F|^2e^{-\varphi_\alpha}\\
\le&\frac{G(T_1)}{\int_{T_1}^{+\infty}c(s)e^{-s}ds}
\limsup_{n\to+\infty}\sum_{i=0}^{n-1}\frac{1}{\inf_{I_{i,n}\backslash U_k}c(t)}\int_{I_{i,n}\backslash U_k}c(s)e^{-s}ds.
\end{split}
\end{equation}
It is clear that $c(t)$ is uniformly continuous and has positive lower bound and upper bound on $[t'_2,t'_1]\backslash U_k$. Then we have
\begin{equation}\label{linear 3.9}
\begin{split}
&\limsup_{n\to+\infty}\sum_{i=0}^{n-1}\frac{1}{\inf_{I_{i,n}\backslash U_k}c(t)}\int_{I_{i,n}\backslash U_k}c(s)e^{-s}ds \\
\le&\limsup_{n\to+\infty}\sum_{i=0}^{n-1}\frac{\sup_{I_{i,n}\backslash U_k}c(t)}{\inf_{I_{i,n}\backslash U_k}c(t)}\int_{I_{i,n}\backslash U_k}e^{-s}ds\\
=&\int_{(t'_2,t'_1]\backslash U_k}e^{-s}ds.
\end{split}
\end{equation}

Combining inequality \eqref{linear 3.6}, \eqref{linear 3.8} and \eqref{linear 3.9}, we have
\begin{equation}\label{linear 3.10}
\begin{split}
&\int_{\{-t'_1\le\Psi<-t'_2\}}|\tilde{F}|^2e^{-\varphi_\alpha}\\
=&\int_{\{z\in M:-\Psi(z)\in(t'_2,t'_1]\backslash U_k\}}|\tilde{F}|^2e^{-\varphi_\alpha}+
\int_{\{z\in M:-\Psi(z)\in[t'_2,t'_1]\cap U_k\}}|\tilde{F}|^2e^{-\varphi_\alpha}\\
\le&\frac{G(T_1)}{\int_{T_1}^{+\infty}c(s)e^{-s}ds}\int_{(t'_2,t'_1]\backslash U_k}e^{-s}ds+
\int_{\{z\in M:-\Psi(z)\in[t'_2,t'_1]\cap U_k\}}|\tilde{F}|^2e^{-\varphi_\alpha}.
\end{split}
\end{equation}
Let $k\to +\infty$, following from equality \eqref{linear 3.5} and inequality \eqref{linear 3.10}, then we obtain that
\begin{equation}\label{linear 3.11}
\begin{split}
\int_{\{-t'_1\le\psi<-t'_2\}}|\tilde{F}|^2e^{-\varphi_\alpha}
\le\frac{G(T_1)}{\int_{T_1}^{+\infty}c(s)e^{-s}ds}\int_{t'_2}^{t'_1}e^{-s}ds.
\end{split}
\end{equation}
Following from a similar discussion we can obtain that
\begin{equation}\nonumber
\begin{split}
\int_{\{-t'_1\le\psi<-t'_2\}}|\tilde{F}|^2e^{-\varphi_\alpha}
\ge\frac{G(T_1)}{\int_{T_1}^{+\infty}c(s)e^{-s}ds}\int_{t'_2}^{t'_1}e^{-s}ds.
\end{split}
\end{equation}
Then combining inequality \eqref{linear 3.11}, we know
\begin{equation}\label{linear 3.12}
\begin{split}
\int_{\{-t'_1\le\Psi<-t'_2\}}|\tilde{F}|^2e^{-\varphi_\alpha}
=\frac{G(T_1)}{\int_{T_1}^{+\infty}c(s)e^{-s}ds}\int_{t'_2}^{t'_1}e^{-s}ds.
\end{split}
\end{equation}
Then it is clear that for any open set $U\subset (T,+\infty)$ and compact set $V\subset (T,+\infty)$,
$$
\int_{\{z\in M;-\Psi(z)\in U\}}|\tilde{F}|^2e^{-\varphi_\alpha}
=\frac{G(T_1)}{\int_{T_1}^{+\infty}c(s)e^{-s}ds}\int_{U}e^{-s}ds,
$$
and
$$
\int_{\{z\in M;-\Psi(z)\in V\}}|\tilde{F}|^2e^{-\varphi_\alpha}
=\frac{G(T_1)}{\int_{T_1}^{+\infty}c(s)e^{-s}ds}\int_{V}e^{-s}ds.
$$
As $E\subset\subset (T,+\infty)$, then $E\cap(t_2,t_1]$ is a Lebesgue measurable subset of $(T+\frac{1}{n},n)$ for some large $n$, where $T\le t_2<t_1\le+\infty$. Then there exists a sequence of compact sets $\{V_j\}$ and a sequence of open subsets $\{V'_j\}$ satisfying $V_1\subset \ldots \subset V_j\subset V_{j+1}\subset\ldots \subset E\cap(t_2,t_1]\subset \ldots \subset V'_{j+1}\subset V'_j\subset \ldots\subset V'_1\subset\subset (T,+\infty)$ and $\lim\limits_{j\to +\infty}\mu(V'_j-V_j)=0$, where $\mu$ is the Lebesgue measure on $\mathbb{R}$. Then we have
\begin{equation}\nonumber
\begin{split}
\int_{\{-t'_1\le\Psi<-t'_2\}}|\tilde{F}|^2e^{-\varphi_\alpha}\mathbb{I}_E(-\Psi)
=&\int_{z\in M:-\Psi(z)\in E\cap (t_2,t_1]}|\tilde{F}|^2e^{-\varphi_\alpha}\\
\le&\liminf_{j\to+\infty}\int_{\{z\in M:-\Psi(z)\in V'_j\}}|\tilde{F}|^2e^{-\varphi_\alpha}\\
\le&\liminf_{j\to+\infty}\frac{G(T_1)}{\int_{T_1}^{+\infty}c(s)e^{-s}ds}\int_{V'_j}e^{-s}ds\\
\le&\frac{G(T_1)}{\int_{T_1}^{+\infty}c(s)e^{-s}ds}\int_{E\cap(t_2,t_1]}e^{-s}ds\\
=&\frac{G(T_1)}{\int_{T_1}^{+\infty}c(s)e^{-s}ds}\int_{t_2}^{t_1}e^{-s}\mathbb{I}_E(s)ds,
\end{split}
\end{equation}
and
\begin{equation}\nonumber
\begin{split}
\int_{\{-t'_1\le\Psi<-t'_2\}}|\tilde{F}|^2e^{-\varphi_\alpha}\mathbb{I}_E(-\Psi)
\ge&\liminf_{j\to+\infty}\int_{\{z\in M:-\Psi(z)\in V_j\}}|\tilde{F}|^2e^{-\varphi_\alpha}\\
\ge&\liminf_{j\to+\infty}\frac{G(T_1)}{\int_{T_1}^{+\infty}c(s)e^{-s}ds}\int_{V_j}e^{-s}ds\\
=&\frac{G(T_1)}{\int_{T_1}^{+\infty}c(s)e^{-s}ds}\int_{t_2}^{t_1}e^{-s}\mathbb{I}_E(s)ds,
\end{split}
\end{equation}
which implies that
$$\int_{\{-t'_1\le\Psi<-t'_2\}}|\tilde{F}|^2e^{-\varphi_\alpha}\mathbb{I}_E(-\Psi)=
\frac{G(T_1)}{\int_{T_1}^{+\infty}c(s)e^{-s}ds}\int_{t_2}^{t_1}e^{-s}\mathbb{I}_E(s)ds.$$
Hence we know that equality \eqref{other a also linear} holds.

Corollary \ref{necessary condition for linear of G} is proved.
\end{proof}

Now we prove Remark \ref{rem:linear}.

\begin{proof}[Proof of Remark \ref{rem:linear}]
By the definition of $G(t;\tilde{c})$, we have $G(t_0;\tilde{c})\le\int_{\{\Psi<-t_0\}}|\tilde{F}|^2e^{-\varphi_\alpha}\tilde{c}(-\Psi)$, where $\tilde{F}$ is the holomorphic $(n,0)$ form on $\{\Psi<-T\}$ such that $G(t)=\int_{\{\Psi<-t\}}|\tilde{F}|^2e^{-\varphi_{\alpha}}c(-\Psi)$ for any $t\ge T$.  Hence we only consider the case $G(t_0;\tilde{c})<+\infty$.

By the definition of $G(t;\tilde{c})$, we can choose a holomorphic $(n,0)$ form $F_{t_0,\tilde{c}}$ on $\{\Psi<-t_0\}$ satisfying $\ (F_{t_0,\tilde{c}}-f)_{z_0}\in
\mathcal{O} (K_M)_{z_0} \otimes J_{z_0}$, for any  $ z_0\in Z_0$ and $\int_{ \{ \Psi<-t_0\}}|F_{t_0,\tilde{c}}|^2e^{-\varphi_\alpha}\tilde{c}(-\Psi)<+\infty$. As $\mathcal{H}^2(\tilde{c},t_0)\subset \mathcal{H}^2(c,t_0)$, we have $\int_{ \{ \Psi<-t_0\}}|F_{t_0,\tilde{c}}|^2e^{-\varphi_\alpha}c(-\Psi)<+\infty$. Using
Lemma \ref{existence of F}, we obtain that
\begin{equation}\nonumber
\begin{split}
\int_{ \{ \Psi<-t\}}|F_{t_0,\tilde{c}}|^2e^{-\varphi_\alpha}c(-\Psi)
=&\int_{ \{ \Psi<-t\}}|\tilde{F}|^2e^{-\varphi_\alpha}c(-\Psi)\\
+&\int_{ \{ \Psi<-t\}}|F_{t_0,\tilde{c}}-\tilde{F}|^2e^{-\varphi_\alpha}c(-\Psi)
\end{split}
\end{equation}
for any $t\ge t_0,$ then
\begin{equation}\label{linear 3.13}
\begin{split}
\int_{ \{-t_3\le \Psi<-t_4\}}|F_{t_0,\tilde{c}}|^2e^{-\varphi_\alpha}c(-\Psi)
=&\int_{ \{-t_3\le \Psi<-t_4\}}|\tilde{F}|^2e^{-\varphi_\alpha}c(-\Psi)\\
+&\int_{\{-t_3\le \Psi<-t_4\}}|F_{t_0,\tilde{c}}-\tilde{F}|^2e^{-\varphi_\alpha}c(-\Psi)
\end{split}
\end{equation}
holds for any $t_3>t_4\ge t_0$. It follows from the dominant convergence theorem, equality \eqref{linear 3.13}, \eqref{linear 3.5} and $c(t)>0$ for any $t>T$, that
\begin{equation}\label{linear 3.14}
\begin{split}
\int_{ \{z\in M:-\Psi(z)=t\}}|F_{t_0,\tilde{c}}|^2e^{-\varphi_\alpha}
=\int_{\{z\in M:-\Psi(z)=t\}}|F_{t_0,\tilde{c}}-\tilde{F}|^2e^{-\varphi_\alpha}
\end{split}
\end{equation}
holds for any $t>t_0$.

Choosing any closed interval $[t'_4,t'_3]\subset (t_0,+\infty)\subset (T,+\infty)$. Note that $c(t)$ is uniformly continuous and have positive lower bound and upper bound on $[t'_4,t'_3]\backslash U_k$, where $\{U_k\}$ is the decreasing sequence of open subsets of $(T,+\infty)$, such that $c$ is continuous on $(T,+\infty)\backslash U_k$ and $\lim\limits_{k \to +\infty}\mu(U_k)=0$. Take $N=\cap_{k=1}^{+\infty}U_k.$ Note that
\begin{equation}\label{linear 3.15}
\begin{split}
&\int_{ \{-t'_3\le\Psi<-t'_4\}}|F_{t_0,\tilde{c}}|^2e^{-\varphi_\alpha}\\
=&\lim_{n\to+\infty}\sum_{i=0}^{n-1}\int_{\{z\in M:-\Psi(z)\in S_{i,n}\backslash U_k\}}|F_{t_0,\tilde{c}}|^2e^{-\varphi_\alpha}
+\int_{\{z\in M:-\Psi(z)\in(t'_4,t'_3]\cap U_k\}}|F_{t_0,\tilde{c}}|^2e^{-\varphi_\alpha}\\
\le&\limsup_{n\to+\infty}\sum_{i=0}^{n-1}\frac{1}{\inf_{S_{i,n}}c(t)}\int_{\{z\in M:-\Psi(z)\in S_{i,n}\backslash U_k\}}|F_{t_0,\tilde{c}}|^2e^{-\varphi_\alpha}c(-\Psi)\\
&+\int_{\{z\in M:-\Psi(z)\in(t'_4,t'_3]\cap U_k\}}|F_{t_0,\tilde{c}}|^2e^{-\varphi_\alpha},
\end{split}
\end{equation}
where $S_{i,n}=(t'_4-(i+1)\alpha_n,t'_3-i\alpha_n]$ and $\alpha_n=\frac{t'_3-t'_4}{n}$.
It follows from equality \eqref{linear 3.13},\eqref{linear 3.14}, \eqref{linear 3.5} and the dominated theorem that
\begin{equation}\label{linear 3.16}
\begin{split}
&\int_{\{z\in M:-\Psi(z)\in S_{i,n}\backslash U_k\}}|F_{t_0,\tilde{c}}|^2e^{-\varphi_\alpha}c(-\Psi)\\
=&\int_{\{z\in M:-\Psi(z)\in S_{i,n}\backslash U_k\}}|\tilde F|^2e^{-\varphi_\alpha}c(-\Psi)
+\int_{\{z\in M:-\Psi(z)\in S_{i,n}\backslash U_k\}}|F_{t_0,\tilde{c}}-\tilde F|^2e^{-\varphi_\alpha}c(-\Psi).
\end{split}
\end{equation}
As $c(t)$ is uniformly continuous and have positive lower bound and upper bound on $[t'_3,t'_4]\backslash U_k$, combining equality \eqref{linear 3.16}, we have
\begin{equation}\label{linear 3.17}
\begin{split}
&\limsup_{n\to+\infty}\sum_{i=0}^{n-1}\frac{1}{\inf_{S_{i,n}\backslash U_k}c(t)}\int_{\{z\in M:-\Psi(z)\in S_{i,n}\backslash U_k\}}|F_{t_0,\tilde{c}}|^2e^{-\varphi_\alpha}c(-\Psi)\\
=&\limsup_{n\to+\infty}\sum_{i=0}^{n-1}\frac{1}{\inf_{S_{i,n}\backslash U_k}c(t)}(\int_{\{z\in M:-\Psi(z)\in S_{i,n}\backslash U_k\}}|\tilde F|^2e^{-\varphi_\alpha}c(-\Psi)\\
&+\int_{\{z\in M:-\Psi(z)\in S_{i,n}\backslash U_k\}}|F_{t_0,\tilde{c}}-\tilde F|^2e^{-\varphi_\alpha}c(-\Psi))\\
\le & \limsup_{n\to+\infty}\sum_{i=0}^{n-1}\frac{\sup_{S_{i,n}\backslash U_k}c(t)}{\inf_{S_{i,n}\backslash U_k}c(t)}(\int_{\{z\in M:-\Psi(z)\in S_{i,n}\backslash U_k\}}|\tilde F|^2e^{-\varphi_\alpha}\\
&+\int_{\{z\in M:-\Psi(z)\in S_{i,n}\backslash U_k\}}|F_{t_0,\tilde{c}}-\tilde F|^2e^{-\varphi_\alpha})\\
=&\int_{\{z\in M:-\Psi(z)\in (t'_4,t'_3]\backslash U_k\}}|\tilde F|^2e^{-\varphi_\alpha}
+\int_{\{z\in M:-\Psi(z)\in (t'_4,t'_3]\backslash U_k\}}|F_{t_0,\tilde{c}}-\tilde F|^2e^{-\varphi_\alpha}.
\end{split}
\end{equation}
If follows from inequality \eqref{linear 3.15} and \eqref{linear 3.17} that
\begin{equation}\label{linear 3.18}
\begin{split}
&\int_{ \{-t'_3\le\Psi<-t'_4\}}|F_{t_0,\tilde{c}}|^2e^{-\varphi_\alpha}\\
\le & \int_{\{z\in M:-\Psi(z)\in (t'_4,t'_3]\backslash U_k\}}|\tilde F|^2e^{-\varphi_\alpha}
+\int_{\{z\in M:-\Psi(z)\in (t'_4,t'_3]\backslash U_k\}}|F_{t_0,\tilde{c}}-\tilde F|^2e^{-\varphi_\alpha}\\
&+\int_{ \{z\in M: -\Psi(z)\in(t'_4,t'_3]\cap U_k\}}|F_{t_0,\tilde{c}}|^2e^{-\varphi_\alpha}.
\end{split}
\end{equation}
It follows from $F_{t_0,\tilde{c}}\in\mathcal{H}^2(c,t_0)$ that $\int_{ \{-t'_3\le\Psi<-t'_4\}}|F_{t_0,\tilde{c}}|^2e^{-\varphi_\alpha}<+\infty$. Let $k\to+\infty,$ by equality \eqref{linear 3.5}, inequality \eqref{linear 3.18} and the dominated theorem, we have
\begin{equation}\label{linear 3.19}
\begin{split}
&\int_{ \{-t'_3\le\Psi<-t'_4\}}|F_{t_0,\tilde{c}}|^2e^{-\varphi_\alpha}\\
\le & \int_{\{z\in M:-\Psi(z)\in (t'_4,t'_3]\}}|\tilde F|^2e^{-\varphi_\alpha}
+\int_{\{z\in M:-\Psi(z)\in (t'_4,t'_3]\backslash N\}}|F_{t_0,\tilde{c}}-\tilde F|^2e^{-\varphi_\alpha}\\
&+\int_{ \{z\in M: -\Psi(z)\in(t'_4,t'_3]\cap N\}}|F_{t_0,\tilde{c}}|^2e^{-\varphi_\alpha}.
\end{split}
\end{equation}
By similar discussion, we also have that
\begin{equation}\nonumber
\begin{split}
&\int_{ \{-t'_3\le\Psi<-t'_4\}}|F_{t_0,\tilde{c}}|^2e^{-\varphi_\alpha}\\
\ge & \int_{\{z\in M:-\Psi(z)\in (t'_4,t'_3]\}}|\tilde F|^2e^{-\varphi_\alpha}
+\int_{\{z\in M:-\Psi(z)\in (t'_4,t'_3]\backslash N\}}|F_{t_0,\tilde{c}}-\tilde F|^2e^{-\varphi_\alpha}\\
&+\int_{ \{z\in M: -\Psi(z)\in(t'_4,t'_3]\cap N\}}|F_{t_0,\tilde{c}}|^2e^{-\varphi_\alpha}.
\end{split}
\end{equation}
then combining inequality \eqref{linear 3.19}, we have
\begin{equation}\label{linear 3.20}
\begin{split}
&\int_{ \{-t'_3\le\Psi<-t'_4\}}|F_{t_0,\tilde{c}}|^2e^{-\varphi_\alpha}\\
= & \int_{\{z\in M:-\Psi(z)\in (t'_4,t'_3]\}}|\tilde F|^2e^{-\varphi_\alpha}
+\int_{\{z\in M:-\Psi(z)\in (t'_4,t'_3]\backslash N\}}|F_{t_0,\tilde{c}}-\tilde F|^2e^{-\varphi_\alpha}\\
&+\int_{ \{z\in M: -\Psi(z)\in(t'_4,t'_3]\cap N\}}|F_{t_0,\tilde{c}}|^2e^{-\varphi_\alpha}.
\end{split}
\end{equation}
Using equality \eqref{linear 3.5}, \eqref{linear 3.14} and Levi's Theorem, we have
\begin{equation}\label{linear 3.21}
\begin{split}
&\int_{ \{z\in M:-\Psi(z)\in U\}}|F_{t_0,\tilde{c}}|^2e^{-\varphi_\alpha}\\
= & \int_{\{z\in M:-\Psi(z)\in U\}}|\tilde F|^2e^{-\varphi}
+\int_{\{z\in M:-\Psi(z)\in U\backslash N\}}|F_{t_0,\tilde{c}}-\tilde F|^2e^{-\varphi_\alpha}\\
&+\int_{ \{z\in M: -\Psi(z)\in U\cap N\}}|F_{t_0,\tilde{c}}|^2e^{-\varphi_\alpha}
\end{split}
\end{equation}
holds for any open set $U\subset\subset (t_0,+\infty)$, and
\begin{equation}\label{linear 3.22}
\begin{split}
&\int_{ \{z\in M:-\Psi(z)\in V\}}|F_{t_0,\tilde{c}}|^2e^{-\varphi_\alpha}\\
= & \int_{\{z\in M:-\Psi(z)\in V\}}|\tilde F|^2e^{-\varphi}
+\int_{\{z\in M:-\Psi(z)\in V\backslash N\}}|F_{t_0,\tilde{c}}-\tilde F|^2e^{-\varphi_\alpha}\\
&+\int_{ \{z\in M: -\Psi(z)\in V\cap N\}}|F_{t_0,\tilde{c}}|^2e^{-\varphi_\alpha}
\end{split}
\end{equation}
holds for any compact set $V\subset (t_0,+\infty)$. For any measurable set $E\subset\subset (t_0,+\infty)$, there exists a sequence of compact set $\{V_l\}$, such that $V_l\subset V_{l+1}\subset E$ for any $l$ and $\lim\limits_{l\to +\infty}\mu(V_l)=\mu(E)$, hence by equality \eqref{linear 3.22}, we have
\begin{equation}\label{linear 3.23}
\begin{split}
\int_{ \{\Psi<-t_0\}}|F_{t_0,\tilde{c}}|^2e^{-\varphi_\alpha}\mathbb{I}_E(-\psi)
\ge&\lim_{l \to +\infty} \int_{ \{\psi<-t_0\}}|F_{t_0,\tilde{c}}|^2e^{-\varphi_\alpha}\mathbb{I}_{V_j}(-\psi)\\
\ge&\lim_{l \to +\infty} \int_{ \{\psi<-t_0\}}|\tilde F|^2e^{-\varphi_\alpha}\mathbb{I}_{V_j}(-\psi)\\
=& \int_{ \{\psi<-t_0\}}|\tilde F|^2e^{-\varphi}\mathbb{I}_{V_j}(-\psi).
\end{split}
\end{equation}
It is clear that for any $t>t_0$, there exists a sequence of functions $\{\sum_{j=1}^{n_i}\mathbb{I}_{E_{i,j}}\}_{i=1}^{+\infty}$ defined on $(t,+\infty)$, satisfying $E_{i,j}\subset\subset (t,+\infty)$, $\sum_{j=1}^{n_{i+1}}\mathbb{I}_{E_{i+1,j}}(s)\ge \sum_{j=1}^{n_{i}}\mathbb{I}_{E_{i,j}}(s)$ and $\lim\limits_{i\to+\infty}\sum_{j=1}^{n_i}\mathbb{I}_{E_{i,j}}(s)=\tilde{c}(s)$ for any $s>t$. Combining Levi's Theorem and inequality \eqref{linear 3.23}, we have
\begin{equation}\label{linear 3.24}
\begin{split}
\int_{ \{\Psi<-t_0\}}|F_{t_0,\tilde{c}}|^2e^{-\varphi_\alpha}\tilde{c}(-\Psi)
\ge\int_{ \{\Psi<-t_0\}}|\tilde F|^2e^{-\varphi_\alpha}\tilde{c}(-\Psi).
\end{split}
\end{equation}
By the definition of $G(t_0,\tilde{c})$, we have $G(t_0,\tilde{c})=\int_{ \{\Psi<-t_0\}}|\tilde F|^2e^{-\varphi_\alpha}\tilde{c}(-\Psi).$ Equality \eqref{other c also linear} is proved.

\end{proof}

\section{Proof of Corollary \ref{generalization for concavity II}}

In this section, we prove Corollary \ref{generalization for concavity II}.
\begin{proof}Let $Z_0\subset \{\psi=-\infty\}$. Let $F\equiv 1$ on $M$. Denote $\varphi_\alpha:=\varphi$. As $\psi<-T$ on $M$, we have
$\Psi:=\min\{\psi-2\log|F|,-T\}=\psi$. which implies that $H_{z_0}=\mathcal{H}_{z_0}$ for any $z_0\in Z_0$. Denote $J_{z_0}:=\mathcal{F}_{z_0}$ for any $z_0\in Z_0$.

Let $c(t)\in P_{T,M}$. Now we prove that $c(t)$ belongs to $P_{T,M,\Psi}$. Let $T_1>T $ be a real number. Denote $\Psi_1:=\min\{\psi,-T_1\}$. Let $K$ be any compact subset of $M\backslash E$. If $K\cap \{\psi<-T_1\}\neq\emptyset$, as $\psi=\Psi_1$ on $\{\psi<-T_1\}$, it follows from $e^{-\varphi}c(-\psi)$ has a positive lower bound on $K$ that $e^{-\varphi}c(-\Psi_1)$ has a positive lower bound on $K\cap \{\psi<-T_1\}$. If $K\cap \{\psi\ge-T_1\}\neq \emptyset$, as $\varphi+\psi$ is a plurisubharmonic function on $M$, then the function
$e^{-\varphi}c(-\Psi_1)=e^{-\varphi-\psi}c(-\Psi_1)e^{\psi}\ge e^{-(\varphi+\psi)}c(T_1)e^{-T_1}$ has a positive lower bound on  $ K\cap \{\psi\ge-T_1\}$. Hence $e^{-\varphi}c(-\Psi_1)$ has a positive lower bound on any compact set $K\subset M\backslash E$. We know $c(t) \in P_{T,M,\Psi}$.

By the definition \eqref{def of g(t) for boundary pt} of $G(t;c,\Psi,\varphi_\alpha,J,f)$ and the definition \eqref{def of g(t) in concavity II} of $G(t;c,\psi,\varphi,\mathcal{F},f)$, we know that $G(t;c,\Psi,\varphi_\alpha,J,f)=G(t;c,\psi,\varphi,\mathcal{F},f)$ for any $t\in [T,+\infty)$.
As there exists a $t_0\ge T$ such that $G(t_0;c,\psi,\varphi,f,\mathcal{F})<+\infty$, we know that $G(t_0;c,\Psi,\varphi_\alpha,J,f)<+\infty$.

Then it follows from Theorem \ref{main theorem} that Corollary \ref{generalization for concavity II} holds.
\end{proof}

\section{Proof of Proposition \ref{p:1}}

In this section, we prove Proposition \ref{p:1} by using Theorem \ref{main theorem}.

\begin{proof}
Let $h(x)=\left\{ \begin{array}{cc}
	e^{-\frac{1}{1-(x-1)^2}} & \textrm{if $|x-1|<1$}\\
	0 & \textrm{if $|x-1|\geq1$}
    \end{array} \right. $ be a real function defined on $\mathbb R$, and let $g_n(x)=\frac{n}{(n+1)d}\int_0^{nx}h(s)ds$, where $d=\int_{\mathbb R}h(s)ds$.
    Note that $h(x)\in C_0^\infty(\mathbb R)$ and $h(x)\geq0$ for any $x\in \mathbb R$. Then  we get that $g_n(x)$ is increasing with respect to $x$, $g_n(x)\leq g_{n+1}(x)$ for any $n\in \mathbb{N}$ and $x\in \mathbb R$, and $\lim_{n\rightarrow+\infty}g_n(x)=\mathbb I_{\{s\in \mathbb R: s>0\}}(x)$ for any $x\in \mathbb R$.
Setting $c_t^n(x)=1-g_n(x-t)$, where $t$ is the given positive number in Proposition \ref{p:1}, it follows from the properties of $\{g_n(x)\}_{n\in \mathbb{N}}$ that $c_t^n(x)$ is decreasing with respect to $x$, $c_t^n(x)\geq c_t^{n+1}(x)$ for any $n\in \mathbb{N}$ and $x\in \mathbb R$, and $\lim_{n\rightarrow+\infty}c_t^n(x)=\mathbb I_{\{s\in \mathbb R: s\leq t\}}(x)$ for any $x\in \mathbb R$. Let $\varphi_{\alpha}\equiv0$. Note that $c_t^n(x)\in[\frac{1}{n+1},1]$ on $(0,+\infty)$, then $c_t^n(x)\in P_{0,M,\Psi}$ for any $n\in \mathbb{N}$.

Denote \begin{equation*}
\begin{split}
\inf\bigg\{\int_{\{ \Psi<-t\}}|\tilde{f}|^2c_t^n(-\Psi): &\tilde{f}\in
H^0(\{\Psi<-t\},\mathcal{O} (K_M)  ) \\
&\&\, (\tilde{f}-f)_{z_0}\in
\mathcal{O} (K_M)_{z_0} \otimes I(\Psi)_{z_0}\,\text{for any }  z_0\in Z_0\bigg\}
\end{split}
\end{equation*}
 by $G_{t,n}(s)$.
    It follows from Theorem \ref{main theorem} that
    \begin{equation}
    	\label{eq:proof1}
    	\int_{\{\Psi<-l\}}|f|^2
c_t^n(-\Psi)\geq G_{t,n}(l)\geq \frac{\int_l^{+\infty}c_t^n(s)e^{-s}ds}{\int_0^{+\infty}c_t^n(s)e^{-s}ds}G_{t,n}(0)
	    \end{equation}
	    for any $l>0$.
	    Following from $\int_{\{\Psi<-l\}}|f|^2<+\infty$ for any $l>0$, the properties of $\{c_t^n\}_{n\in \mathbb{N}}$ and the dominated convergence theorem, we obtain that
	    \begin{equation}
	    	\label{eq:proof2}
	    	\lim_{n\rightarrow+\infty}\int_{\{\Psi<-l\}}|f|^2
c_t^n(-\Psi)=\int_{\{-t\leq\Psi<-l\}}|f|^2.	    \end{equation}
As $c_t^n(x)\geq\mathbb I_{\{s\in\mathbb R:s\leq t\}}(x)$ for any $x>0$ and $n\in \mathbb{N}$, then it follows from the definitions of $G_{t,n}(0)$ and $C_{f,\Psi,t}(Z_0)$ that
\begin{equation}
	\label{eq:proof3}
	G_{t,n}(0)\geq C_{f,\Psi,t}(Z_0).
\end{equation}
Combining inequality \eqref{eq:proof1}, equality \eqref{eq:proof2}, and inequality \eqref{eq:proof3}, we obtain that
\begin{displaymath}
	\begin{split}
		\int_{\{-t\leq\Psi<-l\}}|f|^2 &=	\lim_{n\rightarrow+\infty}\int_{\{\Psi<-l\}}|f|^2
c_t^n(-\Psi)\\
&\geq \lim_{n\rightarrow+\infty} \frac{\int_l^{+\infty}c_t^n(s)e^{-s}ds}{\int_0^{+\infty}c_t^n(s)e^{-s}ds}C_{f,\Psi,t}(Z_0)\\
&=\frac{e^{-l}-e^{-t}}{1-e^{-t}}C_{f,\Psi,t}(Z_0)
	\end{split}	
\end{displaymath}
for any $l\in (0,t)$. Following from the difinition of $C_{f,\Psi,t}(Z_0)$, we have $\int_{\{-t\leq\Psi<0\}}|f|^2\ge C_{f,\Psi,t}(Z_0)$. Thus, we have
    \begin{equation}
    \label{eq:proof4}
    	\int_{\{-t\leq\Psi<-l\}}|f|^2\geq\frac{e^{-l}-e^{-t}}{1-e^{-t}}C_{f,\Psi,t}(Z_0)
    	 \end{equation}
    for any $l\in [0,t)$.

Following from Fubini's Theorem and inequality \eqref{eq:proof4}, we obtain that
\begin{displaymath}
	\begin{split}
		\int_{M_t}|f|^2
e^{-\Psi}&=\int_{M_t}\left(|f|^2
\int_0^{e^{-\Psi}}dr\right)\\
&=\int_0^{+\infty}\left(\int_{M_t\cap\{r<e^{-\Psi}\}}|f|^2 \right)dr\\
&=\int_{-\infty}^t\left(\int_{\{-t\leq\Psi<\min\{-l,0\}\}}|f|^2\right)e^ldl\\
&=\int_{-\infty}^0\left(\int_{\{-t\leq\Psi<\min\{-l,0\}\}}|f|^2\right)e^ldl+\int_{0}^t\left(\int_{\{-t\leq\Psi<-l\}}|f|^2\right)e^ldl\\
&\geq C_{f,\Psi,t}(Z_0)\left(\int_{-\infty}^0e^ldl+\int_0^t\frac{1-e^{l-t}}{1-e^{-t}}dl\right)\\
&=\frac{t}{1-e^{-t}}C_{f,\Psi,t}(Z_0).	\end{split}
\end{displaymath}
Then Proposition \ref{p:1} has thus been proved.	
\end{proof}

\section{Proofs of Theorem \ref{thm:soc}, Theorem \ref{p:soc-twist} and Remark \ref{r:tsoc}}\label{sec:tsoc}

The following estimate will be used in the proofs of Theorem \ref{thm:soc} and Theorem \ref{p:soc-twist}.

\begin{Proposition}\label{p:DK}Let $\varphi$ be a plurisubharmonic function $D$, and let $f$ be a holomorphic function on $\{\Psi<-t_0\}$ such that $f_o\in I(\varphi)_o$.
	Assume that $a_o^f(\Psi;\varphi)<+\infty$, then we have $a_o^f(\Psi;\varphi)>0$ and
\begin{displaymath}
		\frac{1}{r^2}\int_{\{a_o^f(\Psi;\varphi)\Psi<\log r\}}|f|^2e^{-\varphi}\ge G(0;c\equiv1,\Psi,\varphi,I_+(2a_o^f(\Psi;\varphi)\Psi+\varphi)_o,f)>0
\end{displaymath}
holds for any $r\in(0,e^{-a_o^f(\Psi;\varphi)t_0}]$, where the definition of $G(0;c\equiv1,\Psi,\varphi,I_+(2a_o^f(\Psi;\varphi)\Psi+\varphi)_o,f)$ can be found in Section \ref{sec:Main result}.
\end{Proposition}
\begin{proof}
	Lemma \ref{l:m5} tells us that there exists $p_0>2a_o^f(\Psi;\varphi)$ such that $I(p_0\Psi+\varphi)_o=I_+(2a_o^f(\Psi;\varphi)\Psi+\varphi)_o$. Following from the definition of $a_o^{f}(\Psi;\varphi)$ and Lemma \ref{characterization of g(t)=0}, we obtain that
\begin{equation}
	\label{eq:0222d}G(0;c\equiv1,\Psi,\varphi,I_+(2a_o^f(\Psi;\varphi)\Psi+\varphi)_o,f)>0.
\end{equation}
Without loss of generality, assume that there exists $t>t_0$ such that $\int_{\{\Psi<-t\}}|f|^2e^{-\varphi}<+\infty$.
Denote that $t_1:=\inf\{t\ge t_0:\int_{\{\Psi<-t\}}|f|^2e^{-\varphi}<+\infty\}.$
 Denote
\begin{displaymath}
 	\inf\left\{\int_{\{p\Psi<-t\}}|\tilde f|^2e^{-\varphi}:\tilde f\in\mathcal{O}(\{p\Psi<-t\})\,\&\,(\tilde f-f)_o\in I(p\Psi+\varphi)_o\right\}
 \end{displaymath}
by $G_{p}(t)$, where $t\in[0,+\infty)$ and $p>2a_o^f(\Psi;\varphi)$. Then we know that $G_{p}(0)\ge G(0;c\equiv1,\Psi,\varphi,I_+(2a_o^f(\Psi;\varphi)\Psi+\varphi)_o,f)$ for any $p>2a_o^f(\Psi;\varphi)$.
Note that
$$p\Psi=\min\{p\psi+(2\lceil p\rceil-2p)\log|F|-2\log|F^{\lceil p\rceil}|,0\},$$
where $\lceil p\rceil=\min\{n\in \mathbb{Z}:n\geq p\}$,
 and
 $$G_{p}(pt)\le\int_{\{\Psi<-t\}}|f|^2e^{-\varphi}<+\infty$$
 for any $t>t_1$.
  Theorem \ref{main theorem} shows that $G_{p}(-\log r)$ is concave with respect to $r\in(0,1]$ and $\lim_{t\rightarrow+\infty}G_{p}(t)=0$, which implies that
 \begin{equation}
 	\label{eq:0222b}
 	\begin{split}
 	\frac{1}{r_1^2}\int_{\{p\Psi<2\log r_1\}}|f|^2e^{-\varphi}&\ge \frac{1}{r_1^2}G_{p}(-2\log r_1)\\
 	&\ge  G_{p}(0)\\
 	&\ge G(0;c\equiv1,\Psi,\varphi,I_+(2a_o^f(\Psi;\varphi)\Psi+\varphi)_o,f),
 	\end{split}
 \end{equation}
where  $0<r_1\le e^{-\frac{pt_0}{2}}$.

We prove $a_o^f(\Psi;\varphi)>0$ by contradiction: if $a_o^f(\Psi;\varphi)=0$, as $\int_{\{\Psi<-t_1-1\}}|f|^2e^{-\varphi}<+\infty$, it follows from the dominated convergence theorem and inequality \eqref{eq:0222b} that
\begin{equation}
	\label{eq:0222c}\begin{split}
	\frac{1}{r_1^2}\int_{\{\Psi=-\infty\}}|f|^2e^{-\varphi}&=\lim_{p\rightarrow0+0}\frac{1}{r_1^2}\int_{\{p\Psi<2\log r_1\}}|f|^2e^{-\varphi}\\
	&\ge  G(0;c\equiv1,\Psi,\varphi,I_+(2a_o^f(\Psi;\varphi)\Psi+\varphi)_o,f).\end{split}
\end{equation}
Note that $\mu(\{\Psi=-\infty\})=\mu(\{\psi=-\infty\})=0$, where $\mu$ is the Lebesgue measure on $\mathbb{C}^n$. Inequality \eqref{eq:0222c} implies that $G(0;c\equiv1,\Psi,\varphi,I_+(2a_o^f(\Psi;\varphi)\Psi+\varphi)_o,f)=0$, which contradicts inequality \eqref{eq:0222d}. Thus, we get that $a_o^f(\Psi;\varphi)>0$.

For any $r_2\in (0,e^{-a_o^{f}(\Psi;\varphi)t_1})$, note that $\frac{2\log r_2}{p}<-t_1$ for any $p\in(2a_0^f(\Psi;\varphi),-\frac{2\log r_2}{t_1})$, which implies that
$\int_{\{p\Psi<2\log r_2\}}|f|^2e^{-\varphi}<+\infty$ for any $p\in(2a_0^f(\Psi;\varphi),-\frac{2\log r_2}{t_1}).$
 Then it follows from the dominated convergence theorem and inequality \eqref{eq:0222b} that
\begin{equation}
	\label{eq:0222e}\begin{split}
	\frac{1}{r_2^2}\int_{\{2a_0^f(\Psi;\varphi)\Psi\le2\log r_2\}}|f|^2e^{-\varphi}&=\lim_{p\rightarrow2a_0^f(\Psi;\varphi)+0}\frac{1}{r_2^2}\int_{\{p\Psi<2\log r_2\}}|f|^2e^{-\varphi}\\
	&\ge G(0;c\equiv1,\Psi,\varphi,I_+(2a_o^f(\Psi;\varphi)\Psi+\varphi)_o,f).	
	\end{split}
\end{equation}
For any $r\in(0,e^{-a_o^f(\Psi;\varphi)t_0}]$, if $r>e^{-a_o^{f}(\Psi;\varphi)t_1}$, we have $\int_{\{a_0^f(\Psi;\varphi)\Psi<\log r\}}|f|^2e^{-\varphi}=+\infty>G(0;c\equiv1,\Psi,\varphi,I_+(2a_o^f(\Psi;\varphi)\Psi+\varphi)_o,f)$, and if $r\in(0,e^{-a_o^f(\Psi;\varphi)t_1}]$, it follows from  $\{a_o^f(\Psi;\varphi)\Psi<\log r\}=\cup_{0<r_2<r}\{a_o^f(\Psi;\varphi)\Psi<\log r_2\}$ and inequality \eqref{eq:0222e}  that
\begin{displaymath}
	\begin{split}
\int_{\{a_0^f(\Psi;\varphi)\Psi<\log r\}}|f|^2e^{-\varphi}&=
\sup_{r_2\in(0,r)}\int_{\{2a_0^f(\Psi;\varphi)\Psi\le2\log r_2\}}|f|^2e^{-\varphi}
\\&\ge \sup_{r_2\in(0,r)}r_2^2G(0;c\equiv1,\Psi,\varphi,I_+(2a_o^f(\Psi;\varphi)\Psi+\varphi)_o,f)\\
&=r^2G(0;c\equiv1,\Psi,\varphi,I_+(2a_o^f(\Psi;\varphi)\Psi+\varphi)_o,f).
	\end{split}
\end{displaymath}

Thus, Proposition \ref{p:DK} holds.
\end{proof}

\begin{proof}
	[Proof of Theorem \ref{thm:soc}]
	It is clear that $I_+(a\Psi+\varphi)_o\subset I(a\Psi+\varphi)_o$, hence it suffices to prove that   $I(a\Psi+\varphi)\subset I_+(a\Psi+\varphi)_o$.

If there exists $f_o\in I(a\Psi+\varphi)_o$ such that $f_o\not\in I_+(a\Psi+\varphi)_o$, then $a_o^f(\Psi;\varphi)_o=\frac{a}{2}<+\infty$.  Proposition \ref{p:DK} shows that $a>0$. Without loss of generality, assume that $f$ is a holomorphic function on $\{\Psi<-t_0\}\cap D$, where $t_0>0$. For any neighborhood $U\subset D$ of $o$, it follows from Proposition \ref{p:DK} that  there exists $C_U>0$ such that
\begin{equation}
	\label{eq:0222f}\frac{1}{r^2}\int_{\{a\Psi<2\log r\}\cap U}|f|^2e^{-\varphi}\ge C_U
\end{equation}
for any $r\in(0,e^{-\frac{at_0}{2}}]$. For any $t>at_0$, it follows from Fubini's Theorem and inequality \eqref{eq:0222f} that
\begin{displaymath}
	\begin{split}
		\int_{\{a\Psi<-t\}\cap U}|f|^2e^{-a\Psi-\varphi}
		=&\int_{\{a\Psi<-t\}\cap U}\left(|f|^2e^{-\varphi}\int_0^{e^{-a\Psi}}dl\right)\\
		=&\int_0^{+\infty}\left(\int_{\{l<e^{-a\Psi}\}\cap\{a\Psi<-t\}\cap U}|f|^2e^{-\varphi}\right)dl\\
		\ge&\int_{e^t}^{+\infty}\left(\int_{\{a\Psi<-\log l\}\cap U}|f|^2e^{-\varphi}\right)dl\\
		\ge&C_U\int_{e^t}^{+\infty}\frac{1}{l}dl\\
		=&+\infty,
	\end{split}
\end{displaymath}
which contradicts  $f_o\in I(a\Psi+\varphi)_o$. Thus, there is no $f_o\in I(a\Psi+\varphi)_o$ such that $f_o\not\in I_+(a\Psi+\varphi)_o$, which implies that $I(a\Psi+\varphi)= I_+(a\Psi+\varphi)_o$ for any $a\ge0$.	
\end{proof}

We recall two lemmas, which will be used in the proof of Theorem \ref{p:soc-twist}.

\begin{Lemma}[see \cite{GY-twisted}]\label{l:2}
Let $a(t)$ be a positive measurable function on $(-\infty,+\infty)$,
such that $a(t)e^{t}$ is increasing near $+\infty$,
and $a(t)$ is not integrable near $+\infty$.
Then there exists a positive measurable function $\tilde a(t)$ on $(-\infty,+\infty)$ satisfying the following statements:
		
		$(1)$ there exists $T<+\infty$ such that $\tilde{a}(t)\leq a(t)$ for any $t>T$;
		
		$(2)$  $\tilde a(t)e^{t}$ is strictly increasing and continuous near $+\infty$;
		
		$(3)$ $\tilde a(t)$ is not integrable near $+\infty$.
\end{Lemma}
\begin{Lemma}
	[see \cite{GZ-soc17}]\label{l:m} For any two measurable spaces $(X_i,\mu_i)$ and two measurable functions $g_i$ on $X_i$ respectively ($i\in\{1,2\}$), if $\mu_1(\{g_1\geq r^{-1}\})\geq\mu_2(\{g_2\geq r^{-1}\})$ for any $r\in(0,r_0]$, then $\int_{\{g_1\geq r_0^{-1}\}}g_1d\mu_1\geq\int_{\{g_2\geq r_0^{-1}\}}g_2d\mu_2$.
\end{Lemma}

\begin{proof}[Proof of Theorem \ref{p:soc-twist}]
	We prove Theorem \ref{p:soc-twist} in two cases, that $a(t)$ satisfies condition $(1)$ or condition $(2)$.

\

\emph{Case $(1)$. $a(t)$ is decreasing near $+\infty$.}

\

Firstly, we prove $(B)\Rightarrow(A)$.
Consider $F\equiv f\equiv1$, $\varphi\equiv0$ and $\psi=\log|z_1|$ on the unit polydisc $\Delta^n\subset\mathbb{C}^n$. Note that $a_{o}^1(\log|z_1|;0)=1$ and
\begin{displaymath}
	\begin{split}
		\int_{\Delta_{r_0}^n}a(-2\log|z_1|)\frac{1}{|z_1|^2}=&(\pi r_0^2)^{n-1}\int_{\Delta_{r_0}}a(-2\log|z_1|)\frac{1}{|z_1|^2}\\
		=&(\pi r_0^2)^{n-1}2\pi\int_0^{r_0}a(-2\log r)r^{-1}dr\\
		=&(\pi r_0^2)^{n-1}\pi\int_{-2\log{r_0}}^{+\infty}a(t)dt,
	\end{split}
\end{displaymath}
hence we obtain $(B)\Rightarrow (A)$.

Then, we prove $(A)\Rightarrow(B)$.
Theorem \ref{thm:soc} shows that  $f_o\not\in I(2a_o^f(\Psi;\varphi)\Psi+\varphi)_o$ and $a_o^f(\Psi;\varphi)>0$.
Now we assume that there exist $t_0>0$ and a pseudoconvex domain $D_0\subset D$ containing $o$ such that  $\int_{\{\Psi<-t_0\}\cap D_0}|f|^2e^{-2a_o^f(\Psi;\varphi)\Psi-\varphi}a(-2a_o^f(\Psi;\varphi)\Psi)<+\infty$  to get a contradiction. As $f_o\in I(\varphi)_o$, there exist $t_1>t_0$ and
a pseudoconvex domain $D_1\subset D_0$ containing $o$ such that
$\int_{D_1\cap\{\Psi<-t_1\}}|f|^2e^{-\varphi}<+\infty$.
Set $c(t)=a(t)e^{t}+1$, then we have
\begin{equation}
	\label{eq:0305a}
	\int_{D_1\cap\{\Psi<-t_1\}}|f|^2e^{-\varphi}c(-2a_o^f(\Psi;\varphi)\Psi)<+\infty.
\end{equation}
Without loss of generality, assume that $a(t)$ is decreasing on $(2a_o^f(\Psi;\varphi)t_1,+\infty)$. Note that
 $c(t)e^{-t}=a(t)+e^{-t}$ is decreasing on $(2a_o^f(\Psi;\varphi)t_1,+\infty)$ and $\liminf_{t\rightarrow+\infty}c(t)>0$. As $a(t)$ is not integrable near $+\infty$, so is $c(t)e^{-t}$. Note that there exist a plurisubharmonic function $\psi_1$ and a holomorphic function $F_1$ on $D_1$ such that
 $$\psi_1-2\log|F_1|= 2a_o^f(\Psi;\varphi)(\psi-2\log|F|)$$
 on $D_1$. Denote that $\Psi_1:=\min\{\psi_1-2\log|F_1|,-2a_o^f(\Psi;\varphi)t_1\}$ on $D_1$.
  Using Remark \ref{infty2} (replacing $M$, $\Psi$ and $T$ by $D_1$, $\Psi_1$ and $2a_o^f(\Psi;\varphi)t_1$ respectively), as $f_o\not\in I(2a_o^f(\Psi;\varphi)\Psi+\varphi)_o=I(\Psi_1+\varphi)_o$, then we have $G(2a_o^f(\Psi;\varphi)t_1;c,\Psi_1,\varphi,I(\Psi_1+\varphi)_o,f)=+\infty$,
which contradicts to inequality \eqref{eq:0305a}. Thus, we obtain $(A)\Rightarrow (B)$.

\

\emph{Case $(2)$. $a(t)e^{t}$ is increasing near $+\infty$.}

\

In this case, the proof of $(B)\Rightarrow(A)$ is the same as the case $(1)$, therefore it suffices to prove $(A)\Rightarrow(B)$.

Assume that statement $(A)$ holds.
 It follows from Lemma \ref{l:2} that there exists a positive function $\tilde a(t)$ on $(-\infty,+\infty)$ satisfying that: $\tilde a(t)\leq a(t)$ near $+\infty$; $\tilde a(t)e^{t}$ is strictly increasing and continuous near $+\infty$; $\tilde a(t)$ is not integrable near $+\infty$. Thus, it suffices to prove that for any $\Psi$, $\varphi$ and $f_o\in I(\varphi)_o$ satisfying $a_{o}^{f}(\Psi;\varphi)<+\infty$, $\tilde a(-2a_{o}^{f}(\Psi;\varphi)\Psi)\exp(-2a_{o}^{f}(\Psi;\varphi)\Psi-\varphi+2\log|f|)\not\in L^1(U\cap\{\Psi<-t\})$, where $U$ is any neighborhood of $o$ and $t>0$.

Take any  $t_0\gg0$ and any small pseudoconvex domain $D_0\subset D$ containing the origin $o$ such that $f\in\mathcal{O}(D_0\cap\{\Psi<-t_0\})$.
 Let $\mu_1(X)=\int_{X}|f|^2e^{-\varphi}$, where $X$ is a Lebesgue measurable subset of $D_0\cap\{\Psi<-t_0\}$,  and let $\mu_2$ be the Lebeague measure on $(0,1]$. Denote that $Y_r=\{-2a_{o}^f(\Psi;\varphi)\Psi\geq -\log r\}$. Proposition \ref{p:DK} shows that there exists a positive constant $C$ such that $\mu_1(Y_r)\geq Cr$ holds for any $r\in(0,e^{-2a_o^f(\Psi;\varphi)t_0}]$.

Let $g_1=\tilde a(-2a_o^{f}(\Psi;\varphi)\Psi)\exp(-2a_o^{f}(\Psi;\varphi)\Psi)$ and $g_2(x)=\tilde a(-\log x+\log C)Cx^{-1}$. As $\tilde a(t)e^{-t}$ is strictly increasing near $+\infty$, then $g_1\geq \tilde a(-\log r)r^{-1}$ on $Y_r$ implies that
\begin{equation}
	\label{eq:210820h}
	\mu_1(\{g_1\geq \tilde a(-\log r)r^{-1}\})\geq\mu_1(Y_r)\geq Cr
\end{equation}
holds for any  $r>0$ small enough.
As $\tilde a(t)e^{t}$ is strictly increasing near $+\infty$, then there exists $r_0\in(0,e^{-2a_o^f(\Psi)t_0})$ such that
\begin{equation}
	\label{eq:210820i}
	\mu_2(\{x\in(0,r_0]:g_2(x)\geq \tilde a(-\log r)r^{-1}\})=\mu_2(\{0<x\leq Cr\})=C r
\end{equation}
for any $r\in(0,r_0]$.
As  $\tilde a(-\log r)r^{-1}$ converges to $+\infty$ (when $r\rightarrow0+0$) and $\tilde a(t)$ is continuous near $+\infty$, we obtain that
$$\mu_1(\{g_1\geq r^{-1}\})\geq\mu_2(\{x\in(0,r_0]:g_2(x)\geq r^{-1}\})$$
holds for any $r>0$ small enough. Following from Lemma \ref{l:m} and $\tilde a(t)$ is not integrable near $+\infty$, we obtain $\tilde a(-2a_{o}^{f}(\Psi;\varphi)\Psi)\exp(-2a_{o}^{f}(\Psi;\varphi)\Psi-\varphi+2\log|f|)\not\in L^1(D_0\cap\{\Psi<-t_0\})$.

Thus, Theorem \ref{p:soc-twist} holds.
\end{proof}

The following lemma will be used in the proof of Remark \ref{r:tsoc}.
\begin{Lemma}
	\label{l:rtsoc} Let $h$ be a holomorphic function on a pseudoconvex domain $D\subset\mathbb{C}^n$ containing the origin $o$, and let $\psi$ be a plurisubharmonic function on $D$ such that $c_{o,p}^h(\psi)<+\infty$, where $p>0$. Then we have $a_o^1(\Psi;0)=\frac{1}{2}$, where  $\Psi=\min\{2c_{o,p}^h(\psi)\psi-p\log|h|,0\}$.
	\end{Lemma}
	\begin{proof}
		As $\psi$ is a plurisubharmonic function, there exist $\delta>0$ and a neighborhood $U$ of $o$ such that $\int_{U}e^{-\delta\psi}<+\infty$.
		For any $r\in(0,1)$, there exist $r'\in(0,r)$ and a neighborhood $U'\Subset U$ of $o$ such that $2\frac{r-r'}{1-r}c_{o,p}^h(\psi)<\delta$ and $\int_{U'}|h|^{p}e^{-2\frac{r'}{r}c_{o,p}^h(\psi)\psi}<+\infty$.  Following from H\"older Inequality, we get that
		\begin{displaymath}
\begin{split}
				\int_{U'}|h|^{pr}e^{-2rc_{o,p}^h(\psi)\psi}&=\int_{U'}|h|^{pr}e^{-2r'c_{o,p}^h(\psi)\psi}e^{-2(r-r')c_{o,p}^h(\psi)\psi}\\
				&\le\left(\int_{U'}|h|^{p}e^{-2\frac{r'}{r}c_{o,p}^h(\psi)\psi} \right)^{r}\times\left(\int_{U'}e^{-2\frac{r-r'}{1-r}c_{o,p}^h(\psi)\psi} \right)^{1-r}\\
				&<+\infty,
\end{split}		\end{displaymath}
which implies that $a_o^1(\Psi;0)\ge\frac{1}{2}$. We prove $a_o^1(\Psi;0)=\frac{1}{2}$ by contradiction: if not, there exist $s>\frac{1}{2}$, $t>0$ and a neighborhood $V\Subset U$ of $o$ such that
\begin{equation}
	\label{eq:0307a}\int_{\{\Psi<-t\}\cap V}|h|^{2sp}e^{-4sc_{o,p}^h(\psi)\psi}=\int_{\{\Psi<-t\}\cap V}e^{-2s\Psi}<+\infty.
\end{equation}
There exists $s'>1$ such that $\frac{4s(s'-1)}{2s-1}c_{o,p}^h(\psi)<\delta$. Following from inequality \eqref{eq:0307a} and H\"older Inequality, we get that
\begin{displaymath}
	\begin{split}
		&\int_{\{\Psi<-t\}\cap V}|h|^pe^{-2s'c_{o,p}^h(\psi)\psi}\\=&\int_{\{\Psi<-t\}\cap V}|h|^pe^{-2c_{o,p}^h(\psi)\psi}e^{-2(s'-1)c_{o,p}^h(\psi)\psi}\\
		\le&\left(\int_{\{\Psi<-t\}\cap V}|h|^{2sp}e^{-4sc_{o,p}^h(\psi)\psi} \right)^{\frac{1}{2s}}\times\left(\int_{\{\Psi<-t\}\cap V}e^{-\frac{4s(s'-1)}{2s-1}c_{o,p}^h(\psi)\psi} \right)^{1-\frac{1}{2s}}\\
		<&+\infty,
	\end{split}
\end{displaymath}
which contradicts to the definition of $c_{o,p}^h(\psi)$. Thus, we obtain $a_o^1(\Psi;0)=\frac{1}{2}.$
	\end{proof}

\begin{proof}
	[Proof of Remark \ref{r:tsoc}]We give a proof of Theorem \ref{thm:tsoc} by using Theorem \ref{p:soc-twist}. The proofs of $(B')\Rightarrow(A')$ and $(C')\Rightarrow(A')$ are as the same as the proof of $(B)\Rightarrow (A)$ in Theorem \ref{p:soc-twist}. It suffices to prove $(A')\Rightarrow (B')$ and $(A')\Rightarrow (C')$.
	
	Assume that statement $(A')$ holds, then we have the statement $(B)$ in Theorem \ref{p:soc-twist} holds.
	
	For any $h$ and $\psi$ satisfying $c_{o,p}^h(\psi)<+\infty$, let $F\equiv1$, $f=h^{\lceil\frac{p}{2}\rceil}$ and $\varphi=(2\lceil\frac{p}{2}\rceil-p)\log|h|$, where $\lceil m\rceil:=\min\{n\in \mathbb{Z}:n\geq m\}$. Note that $|h|^p=|f|^2e^{-\varphi}$, $\psi(o)=-\infty$ and $a_o^f(\Psi;\varphi)=c_{o,p}^h(\psi)$. The holding of statement $(B)$ in Theorem \ref{p:soc-twist} implies that
	$a(-2c_{o,p}^h(\psi)\psi)\exp(-2c_{o,p}^h(\psi)\psi+p\log|h|)$ is not integrable near $o$.
	
	For any $h$ and $\psi$ satisfying $c_{o,p}^h(\psi)<+\infty$, let $\Psi=\min\{2c_{o,p}^h(\psi)\psi-p\log|h|,0\}$, $\varphi\equiv0$ and $f\equiv1$. Lemma \ref{l:rtsoc} shows that $a_o^f(\Psi;\varphi)=\frac{1}{2}$. The holding of statement $(B)$ in Theorem \ref{p:soc-twist} implies that
	$a(-\Psi)\exp(-\Psi)\not\in L^1(U\cap\{\Psi<0\})$ for any neighborhood $U$ of $o$, which implies that $a(-2c_{o,p}^h(\psi)\psi+p\log|h|)\exp(-2c_{o,p}^h(\psi)\psi+p\log|h|)$ is not integrable near $o$.
\end{proof}

\section{Proof of Theorem \ref{thm:effe}}
In this section, we prove Theorem \ref{thm:effe} by using Theorem \ref{main theorem}.
\begin{proof}
	[Proof of Theorem \ref{thm:effe}]
	Following from Fubini's Theorem, we have
	\begin{equation}
		\label{eq:0307b}\begin{split}
			&\int_{\{\Psi<0\}}|f|^2e^{-\varphi-\Psi}\\
			=&\int_{\{\Psi<0\}}\left(|f|^2e^{-\varphi-\Psi+a\Psi}\int_0^{e^{-a\Psi}}ds\right)\\
			=&\int_0^{+\infty}\left(\int_{\{\Psi<0\}\cap\{s<e^{-a\Psi}\}}|f|^2e^{-\varphi-\Psi+a\Psi} \right)ds\\
			=&\int_{-\infty}^{+\infty}\left(\int_{\{\Psi<-\frac{l}{a}\}\cap\{\Psi<0\}}|f|^2e^{-\varphi-\Psi+a\Psi} \right)e^{l}dl\\
			=&\int_{\{\Psi<0\}}|f|^2e^{-\varphi-\Psi+a\Psi} +\int_{0}^{+\infty}\left(\int_{\{\Psi<-\frac{l}{a}\}}|f|^2e^{-\varphi-\Psi+a\Psi} \right)e^{l}dl.
		\end{split}
	\end{equation}
	Denote \begin{displaymath}
		\begin{split}
			\inf\bigg\{\int_{\{q'\Psi<-t\}}|\tilde f|^2e^{-\varphi}c(-q'\Psi):(\tilde f-f)_o\in I(q'\Psi+\varphi)_o\,\&\,\tilde f\in\mathcal{O}(\{q'\Psi<-t\}) \bigg\}
		\end{split}
	\end{displaymath}
	by $G_{c,q'}(t)$, where $c$ is a Lebesgue measurable function on $(0,+\infty)$ and $q'> 2a_o^f(\Psi;\varphi)\ge1$.
		
	Next we prove  inequality
	\begin{equation}
		\label{eq:0307c}\int_{\{\Psi<-\frac{l}{a}\}}|f|^2e^{-\varphi-\Psi+a\Psi}\ge e^{-\frac{a-1+q'}{a}l}\frac{1}{K_{\Psi,f,a}(o)}
	\end{equation}
	holds for $l\ge0$, $a>0$ and $q'> 2a_o^f(\Psi;\varphi)$.
	
	We prove inequality \eqref{eq:0307c} for the case $a\in(0,1]$. Let $c_1(t)=e^{\frac{1-a}{q'}t}$ on $(0,+\infty)$. Note that $c_1(t)e^{-t}$ is decreasing on $(0,+\infty)$ and $e^{-\varphi}c_1(-q'\Psi)=e^{-\varphi-(1-a)\Psi}$ has a positive lower bound on any compact subset of $D$. Theorem \ref{main theorem} shows that $G_{c_1,q'}(h^{-1}(r))$ is concave with respect to $r$, where $h(t)=\int_t^{+\infty}c_1(s)e^{-s}ds$. Note that $G_{c_1,q'}(0)\ge\frac{1}{K_{\Psi,f,a}(o)}$ for any $q'> 2a_o^f(\Psi;\varphi)$. Hence we have
	\begin{equation*}
	\begin{split}
			\int_{\{\Psi<-\frac{l}{a}\}}|f|^2e^{-\varphi-(1-a)\Psi}\ge& G_{c_1,q'}\left(\frac{q'l}{a}\right)\\
			\ge&\frac{\int_{\frac{q'l}{a}}^{+\infty}c_1(s)e^{-s}ds}{\int_0^{+\infty}c_1(s)e^{-s}ds}G_{c_1,q'}(0)\\
			\ge&e^{-\frac{a-1+q'}{a}l}\frac{1}{K_{\Psi,f,a}(o)}.		\end{split}
	\end{equation*}
	
	We prove inequality \eqref{eq:0307c} for the case $a>1$. Let $\tilde c_{m}(t)$ be a continuous function on $(0,+\infty)$ such that $\tilde c_m(t)=e^{\frac{1-a}{q'}t}$ on $(0,m)$ and $\tilde c_m(t)=e^{\frac{1-a}{q'}m}$ on $(m,+\infty)$ for any positive integer $m$.   Note that $\tilde c_m(t)e^{-t}$ is decreasing on $(0,+\infty)$ and $e^{-\varphi}\tilde c_m(-q'\Psi)$ has a positive lower bound on any compact subset of $D$. Theorem \ref{main theorem} shows that $G_{\tilde c_m,q'}(h_m^{-1}(r))$ is concave with respect to $r$, where $h_m(t)=\int_t^{+\infty}\tilde c_m(s)e^{-s}ds$. Note that $G_{\tilde c_m,q'}(0)\ge\frac{1}{K_{\Psi,f,a}(o)}$ for any $q'> 2a_o^f(\Psi;\varphi)$. Hence we have
	\begin{equation}\label{eq:0307d}
	\begin{split}
			\int_{\{\Psi<-\frac{l}{a}\}}|f|^2e^{-\varphi}\tilde c_m(-q'\Psi)\ge& G_{\tilde c_m,q'}\left(\frac{q'l}{a}\right)\\
			\ge&\frac{\int_{\frac{q'l}{a}}^{+\infty}\tilde c_m(s)e^{-s}ds}{\int_0^{+\infty}\tilde c_m(s)e^{-s}ds}G_{\tilde c_m,q'}(0)\\
			\ge&\frac{\int_{\frac{q'l}{a}}^{+\infty}\tilde c_m(s)e^{-s}ds}{\int_0^{+\infty}\tilde c_m(s)e^{-s}ds}\frac{1}{K_{\Psi,f,a}(o)}.		\end{split}
	\end{equation}
As $\int_{\{\Psi<0\}}|f|^2e^{-\varphi-\Psi}\le C_1<+\infty$, it follows from $\tilde c_m(-q'\Psi)\le e^{-\Psi}$, the dominated convergence theorem and inequality \eqref{eq:0307d} that
\begin{displaymath}
	\begin{split}
		\int_{\{\Psi<-\frac{l}{a}\}}|f|^2e^{-\varphi-(1-a)\Psi}
		=&\lim_{m\rightarrow+\infty}\int_{\{\Psi<-\frac{l}{a}\}}|f|^2e^{-\varphi}\tilde c_m(-q'\Psi)\\
		\ge&\lim_{m\rightarrow+\infty}\frac{\int_{\frac{q'l}{a}}^{+\infty}\tilde c_m(s)e^{-s}ds}{\int_0^{+\infty}\tilde c_m(s)e^{-s}ds}\frac{1}{K_{\Psi,f,a}(o)}\\
		=&e^{-\frac{a-1+q'}{a}l}\frac{1}{K_{\Psi,f,a}(o)}.
	\end{split}
\end{displaymath}
	
	Combining equality \eqref{eq:0307b}, inequality \eqref{eq:0307c} and the definition of $K_{\Psi,f,a}(o)$, we obtain that
	\begin{equation}
		\label{eq:0307e}\begin{split}
			&\int_{\{\Psi<0\}}|f|^2e^{-\varphi-\Psi}\\
			=&\int_{\{\Psi<0\}}|f|^2e^{-\varphi-\Psi+a\Psi} +\int_{0}^{+\infty}\left(\int_{\{\Psi<-\frac{l}{a}\}}|f|^2e^{-\varphi-\Psi+a\Psi} \right)e^{l}dl\\
			\ge&\left(1+\int_0^{+\infty}e^{-\frac{-1+q'}{a}l}\right)\frac{1}{K_{\Psi,f,a}(o)}\\
			=&\frac{a+q'-1}{q'-1}\cdot\frac{1}{K_{\Psi,f,a}(o)}
		\end{split}
	\end{equation}
	for any $q'>2a_o^f(\Psi;\varphi)$.
	Let $q'\rightarrow2a_o^f(\Psi;\varphi)$, we get that inequality \eqref{eq:0307e} also holds when $q'=2a_o^f(\Psi;\varphi)$. Thus, if $q>1$ satisfies
	$$\frac{q+a-1}{q-1}>\frac{C_1}{C_2}\geq K_{\Psi,f,a}(o)\int_{\{\Psi<0\}}|f|^2e^{-\varphi-\Psi},$$
	we have $p<2a_o^f(\Psi;\varphi)$, i.e. $f_o\in I(p\Psi+\varphi)_o$.	
\end{proof}

\section{Appendix: Proof of Lemma \ref{L2 method}}
In this section, we prove Lemma \ref{L2 method}.
\subsection{Some results used in the proof of Lemma \ref{L2 method}}
\begin{Lemma}[see \cite{Demailly00}]
\label{BKN Identity}
Let Q be a Hermitian vector bundle on a K\"ahler manifold M of dimension $n$ with a
K\"ahler metric $\omega$. Assume that $\eta , g >0$ are smooth functions on M. Then
for every form $v\in D(M,\wedge^{n,q}T^*M \otimes Q)$ with compact support we have
\begin{equation}
\begin{split}
&\int_M (\eta+g^{-1})|D^{''*}v|^2_QdV_M+\int_M \eta|D^{''}v|^2_QdV_M \\
\ge  &\int_M \langle[\eta \sqrt{-1}\Theta_Q-\sqrt{-1}\partial \bar{\partial}
\eta-\sqrt{-1}g
\partial\eta \wedge\bar{\partial}\eta, \Lambda_{\omega}]v,v\rangle_QdV_M.
\end{split}
\end{equation}
\end{Lemma}

\begin{Lemma}[Lemma 4.2 in \cite{guan-zhou13ap}]
Let M and Q be as in the above lemma and $\theta$ be a continuous (1,0) form on M.
Then we have
\begin{equation}
[\sqrt{-1}\theta \wedge
\bar{\theta},\Lambda_\omega]\alpha=\bar{\theta}\wedge(\alpha\llcorner(\bar{\theta})^\sharp),
\end{equation}
for any (n,1) form $\alpha$ with value in Q. Moreover, for any positive (1,1) form
$\beta$, we have $[\beta,\Lambda_\omega]$ is semipositive.
\label{sempositive lemma}
\end{Lemma}

\begin{Lemma}
[Remark 3.2 in \cite{Demailly00}]
\label{d-bar equation with error term}
Let ($M,\omega$) be a complete K\"ahler manifold equipped with a (non-necessarily
complete) K\"ahler metric $\omega$, and let $Q$ be a Hermitian vector bundle over $M$.
Assume that $\eta$ and $g$ are smooth bounded positive functions on $M$ and let
$B:=[\eta \sqrt{-1}\Theta_Q-\sqrt{-1}\partial \bar{\partial} \eta-\sqrt{-1}g
\partial\eta \wedge\bar{\partial}\eta, \Lambda_{\omega}]$. Assume that $\delta \ge
0$ is a number such that $B+\delta I$ is semi-positive definite everywhere on
$\wedge^{n,q}T^*M \otimes Q$ for some $q \ge 1$. Then given a form $v \in
L^2(M,\wedge^{n,q}T^*M \otimes Q)$ such that $D^{''}v=0$ and $\int_M \langle
(B+\delta I)^{-1}v,v\rangle_Q dV_M < +\infty$, there exists an approximate solution
$u \in L^2(M,\wedge^{n,q-1}T^*M \otimes Q)$ and a correcting term $h\in
L^2(M,\wedge^{n,q}T^*M \otimes Q)$ such that $D^{''}u+\sqrt{\delta}h=v$ and
\begin{equation}
\int_M(\eta+g^{-1})^{-1}|u|^2_QdV_M+\int_M|h|^2_QdV_M \le \int_M \langle (B+\delta
I)^{-1}v,v\rangle_Q dV_M.
\end{equation}
\end{Lemma}

\begin{Lemma}
[Theorem 6.1 in \cite{DemaillyReg}, see also Theorem 2.2 in \cite{ZZ2019}]
\label{regularization on cpx mfld}
Let ($M,\omega$) be a complex manifold equipped with a Hermitian metric
$\omega$, and $\Omega \subset \subset M $ be an open set. Assume that
$T=\widetilde{T}+\frac{\sqrt{-1}}{\pi}\partial\bar{\partial}\varphi$ is a closed
(1,1)-current on $M$, where $\widetilde{T}$ is a smooth real (1,1)-form and
$\varphi$ is a quasi-plurisubharmonic function. Let $\gamma$ be a continuous real
(1,1)-form such that $T \ge \gamma$. Suppose that the Chern curvature tensor of
$TM$ satisfies
\begin{equation}
\begin{split}
(\sqrt{-1}&\Theta_{TM}+\varpi \otimes Id_{TM})(\kappa_1 \otimes \kappa_2,\kappa_1
\otimes \kappa_2)\ge 0 \\
&\forall \kappa_1,\kappa_2 \in TM \quad with \quad \langle \kappa_1,\kappa_2
\rangle=0
\end{split}
\end{equation}
for some continuous nonnegative (1,1)-form $\varpi$ on M. Then there is a family
of closed (1,1)-current
$T_{\zeta,\rho}=\widetilde{T}+\frac{\sqrt{-1}}{\pi}\partial\bar{\partial}
\varphi_{\zeta,\rho}$ on M ($\zeta \in (0,+\infty)$ and $\rho \in (0,\rho_1)$ for
some positive number $\rho_1$) independent of $\gamma$, such that
\par
$(i)\ \varphi_{\zeta,\rho}$ is quasi-plurisubharmonic on a neighborhood of
$\bar{\Omega}$, smooth on $M\backslash E_{\zeta}(T)$,
\\
increasing with respect to
$\zeta$ and $\rho$ on $\Omega $ and converges to $\varphi$ on $\Omega$ as $\rho
\to 0$.
\par
$(ii)\ T_{\zeta,\rho}\ge\gamma-\zeta\varpi-\delta_{\rho}\omega$ on $\Omega$.
\par
where $E_{\zeta}(T):=\{x\in M:v(T,x)\ge \zeta\}$ ($\zeta>0$) is the $\zeta$-upper level set of
Lelong numbers and $\{\delta_{\rho}\}$ is an increasing family of positive numbers
such that $\lim\limits_{\rho \to 0}\delta_{\rho}=0$.
\end{Lemma}
\begin{Remark}[see Remark 2.1 in \cite{ZZ2019}]
Lemma \ref{regularization on cpx mfld} is stated in \cite{DemaillyReg} in the case $M$ is a compact complex manifold. The similar proof as in \cite{DemaillyReg} shows that Lemma \ref{regularization on cpx mfld} on noncompact complex manifold still holds where the uniform estimate $(i)$ and $(ii)$ are obtained only on a relatively compact subset $\Omega$.
\end{Remark}

\begin{Lemma}
[Theorem 1.5 in \cite{Demailly82}]
\label{completeness}
Let M be a K\"ahler manifold, and Z be an analytic subset of M. Assume that
$\Omega$ is a relatively compact open subset of M possessing a complete K\"ahler
metric. Then $\Omega\backslash Z $ carries a complete K\"ahler metric.

\end{Lemma}

\begin{Lemma}
[Lemma 6.9 in \cite{Demailly82}]
\label{extension of equality}
Let $\Omega$ be an open subset of $\mathbb{C}^n$ and Z be a complex analytic subset of
$\Omega$. Assume that $v$ is a (p,q-1)-form with $L^2_{loc}$ coefficients and h is
a (p,q)-form with $L^1_{loc}$ coefficients such that $\bar{\partial}v=h$ on
$\Omega\backslash Z$ (in the sense of distribution theory). Then
$\bar{\partial}v=h$ on $\Omega$.
\end{Lemma}

Let $M$ be a complex manifold. Let $\omega$ be a continuous Hermitian metric on $M$. Let $dV_M$ be a continuous volume form on $M$. We denote by $L^2_{p,q}(M,\omega,dV_M)$ the spaces of $L^2$ integrable $(p,q)$ forms over $M$ with respect to $\omega$ and $dV_M$. It is known that $L^2_{p,q}(M,\omega,dV_M)$ is a Hilbert space.
\begin{Lemma}
\label{weakly convergence}
Let $\{u_n\}_{n=1}^{+\infty}$ be a sequence of $(p,q)$ forms in $L^2_{p,q}(M,\omega,dV_M)$ which is weakly convergent to $u$. Let $\{v_n\}_{n=1}^{+\infty}$ be a sequence of Lebesgue measurable real functions on $M$ which converges pointwise to $v$. We assume that there exists a constant $C>0$ such that $|v_n|\le C$ for any $n$. Then $\{v_nu_n\}_{n=1}^{+\infty}$ weakly converges to $vu$ in $L^2_{p,q}(M,\omega,dV_M)$.
\end{Lemma}
\begin{proof}Let $g\in L^2_{p,q}(M,\omega,dV_M)$. Consider
\begin{equation*}
  \begin{split}
     I & =|\langle v_nu_n,g\rangle-\langle vu,g\rangle| \\
       &=|\int_{M}(v_nu_n,g)_{\omega}dV_M-\int_{M}(vu,g)_{\omega}dV_M|\\
       &\le|\int_{M}(v_nu_n-vu_n,g)_{\omega}dV_M|+|\int_{M}(vu_n-vu,g)_{\omega}dV_M|\\
       &=|\int_{M}(u_n,v_ng-vg)_{\omega}dV_M|+|\int_{M}(u_n-u,vg)_{\omega}dV_M|\\
       &\le||u_n||\cdot||v_ng-vg||+|\int_{M}(u_n-u,vg)_{\omega}dV_M|.
  \end{split}
\end{equation*}
Denote $I_1:=||u_n||\cdot||v_ng-vg||$ and $I_2:=|\int_{M}(u_n-u,vg)_{\omega}dV_M|$. It follows from $\{u_n\}_{n=1}^{+\infty}$ weakly converges to $u$ that $||u_n||$ is uniformly bounded with respect to $n$. Note that $|v_n|$ is uniformly bounded with respect to $n$. We know $|v|<C$ and then $vg\in L^2_{p,q}(M,\omega,dV_M)$. Hence we have $I_2 \to 0$ as $n\to+\infty$. It follows from Lebesgue dominated convergence theorem that we have $\lim_{n\to+\infty}I_1 = 0$.

Hence $\lim_{n\to+\infty}I = 0$ and we know $\{v_nu_n\}_{n=1}^{+\infty}$ weakly converges to $vu$ in $L^2_{p,q}(M,\omega,dV_M)$.
\end{proof}

The following notations can be referred to \cite{Boucksom note}.

Let $X$ be a complex manifold. An upper semi-continuous function $u:X\to [-\infty,+\infty)$ is quasi-plurisubharmonic if it is locally of the form $u=\varphi+f$ where $\varphi$ is plurisubharmonic and $f$ is smooth.
Let $\theta$ be a closed, real $(1,1)$ form on $X$. By Poincar\'{e} lemma, $\theta$ is locally of the form $\theta=dd^c f$ for a smooth real-valued function $f$ which is called a local potential of $\theta$. We call a quasi-plurisubharmonic function $u$ is $\theta$-plurisubharmonic if $\theta+dd^c u\ge 0$ in the sense of currents.
\begin{Lemma}[see \cite{Demaillybook}, see also \cite{Boucksom note}]
\label{regular of max} For arbitrary $\eta=(\eta_1,\ldots,\eta_p)\in (0,+\infty)^p$, the function
$$M_{\eta}(t_1,\ldots,t_p)=\int_{\mathbb{R}^p}\max\{t_1+h_1,\ldots,t_p+h_p\}\prod\limits_{1\le j\le p}\theta(\frac{h_j}{\eta_j})dh_1\ldots dh_p$$
possesses the following properties:\par
(1) $M_{\eta}(t_1,\ldots,t_p)$ is non decreasing in all variables, smooth and convex on $\mathbb{R}^p$;\par
(2) $\max\{t_1,\ldots,t_p\}\le M_{\eta}(t_1,\ldots,t_p)\le \max\{t_1+\eta_1,\ldots,t_p+\eta_p\}$;\par
(3) $M_{\eta}(t_1,\ldots,t_p)=M_{\eta_1,\ldots,\hat{\eta}_j,\ldots,\eta_p}(t_1,\ldots,\hat{t}_j,\ldots,t_p)$ if $t_j+\eta_j\le \max\limits_{k\neq j}\{t_k-\eta_k\}$;\par
(4) $M_{\eta}(t_1+a,\ldots,t_p+a)=M_{\eta}(t_1,\ldots,t_p)+a$ for any $a\in\mathbb{R}$;\par
(5) if $u_1,\ldots,u_p$ are plurisubharmonic functions, then $u=M_{\eta}(u_1,\ldots,u_p)$ is plurisubharmonic;\par
(6) if $u_1,\ldots,u_p$ are $\theta$-plurisubharmonic functions, then $u=M_{\eta}(u_1,\ldots,u_p)$ is $\theta$-plurisubharmonic function.
\end{Lemma}
\begin{proof}The proof of (1)-(5) can be referred to \cite{Demaillybook} and the proof of (6) can be referred to \cite{Boucksom note}. For the convenience of the readers, we recall the proof of (6).

 Let $f$ be a local potential of $\theta$. We know $f+u_i$ is plurisubharmonic function. It follows from (4) and (5) that $M_{\eta}(u_1+f,\ldots,u_p+f)=M_{\eta}(u_1,\ldots,u_p)+f$ is plurisubharmonic. Hence $u=M_{\eta}(u_1,\ldots,u_p)$ is $\theta$-plurisubharmonic function.
\end{proof}

\subsection{Proof of Lemma \ref{L2 method}}
Now we begin to prove Lemma \ref{L2 method}.

Note that $M\backslash \{F=0\}$ is a weakly pseudoconvex K{\"a}hler manifold. The following remark shows that we can assume that $F$ has no zero points on $M$.
\begin{Remark}Assume that there exists a holomorphic $(n,0)$ form $\hat{F}$ on $M\backslash \{F=0\}$ such that
\begin{equation*}
  \begin{split}
      & \int_{M\backslash \{F=0\}}|\hat{F}-(1-b_{t_0,B}(\Psi))fF^{1+\delta}|^2e^{-\varphi+v_{t_0,B}(\Psi)-\Psi}c(-v_{t_0,B}(\Psi)) \\
      \le & \left(\frac{1}{\delta}c(T)e^{-T}+\int_{T}^{t_0+B}c(s)e^{-s}ds\right)
       \int_{M\backslash \{F=0\}}\frac{1}{B}\mathbb{I}_{\{-t_0-B<\Psi<-t_0\}}|f|^2e^{-\varphi_{\alpha}-\Psi}.
  \end{split}
\end{equation*}

As $v_{t_0,B}(\Psi)\ge\Psi$ and $c(t)e^{-t}$ is decreasing with respect to $t$, we have
$$e^{-\varphi+v_{t_0,B}(\Psi)-\Psi}c(-v_{t_0,B}(\Psi))\ge e^{-\varphi}c(-\Psi)=e^{-\varphi_\alpha}c(-\Psi)e^{-(1+\delta)\max\{\psi+T,2\log|F|\}}.$$
Let $K$ be any compact subset of $M$. Note that $\varphi+\Psi$ is plurisubharmonic function on $M$ ,$v_{t_0,B}(t)\ge -t_0-B$ and $c(t)e^{-t}$ is decreasing with respect to $t$.
 Then we have

\begin{equation*}
  \begin{split}
     &\int_{(M\backslash \{F=0\})\cap K}|\hat{F}|^2 \\
       \le&2\int_{(M\backslash \{F=0\})\cap K}|(1-b_{t_0,B}(\Psi))fF^{1+\delta}|^2
       +
       2\int_{(M\backslash \{F=0\})\cap K}|\hat{F}-(1-b_{t_0,B}(\Psi))fF^{1+\delta}|^2\\
       \le&2\left(\sup_K|F^{1+\delta}|^2\right)\int_{\{\Psi<-t_0\}\cap K}|f|^2\\
       +&
       \frac{2}{M_K}\int_{(M\backslash \{F=0\})\cap K}|\hat{F}-(1-b_{t_0,B}(\Psi))fF^{1+\delta}|^2e^{-\varphi+v_{t_0,B}(\Psi)-\Psi}c(-v_{t_0,B}(\Psi))\\
       <&+\infty,
  \end{split}
\end{equation*}
where $M_K$ is a positive number.
As $K$ is arbitrarily chosen, we know that there exists a holomorphic $(n,0)$ form $\tilde{F}$ on $M$ such that $\tilde{F}=\hat{F}$ on $M\backslash \{F=0\}$. And we have
\begin{equation*}
  \begin{split}
      & \int_{M}|\tilde{F}-(1-b_{t_0,B}(\Psi))fF^{1+\delta}|^2e^{-\varphi+v_{t_0,B}(\Psi)-\Psi}c(-v_{t_0,B}(\Psi)) \\
      \le & \left(\frac{1}{\delta}c(T)e^{-T}+\int_{T}^{t_0+B}c(s)e^{-s}ds\right)
       \int_{M}\frac{1}{B}\mathbb{I}_{\{-t_0-B<\Psi<-t_0\}}|f|^2e^{-\varphi_{\alpha}-\Psi}.
  \end{split}
\end{equation*}
\end{Remark}
The following remark shows that we can assume that $c(t)$ is a smooth function.
\begin{Remark} We firstly introduce the regularization process of $c(t)$.

Let $f(x)=2\mathbb{I}_{(-\frac{1}{2},\frac{1}{2})}\ast\rho(x)$ be a smooth function on $\mathbb{R}$, where $\rho$ is the kernel of convolution satisfying $\text{supp}(\rho)\subset (-\frac{1}{3},\frac{1}{3})$ and $\rho>0$.

Let $g_i(x)=\left\{ \begin{array}{rcl}
if(ix) & \mbox{if}
&x\le 0 \\ if(i^2 x) & \mbox{if} & x>0
\end{array}\right.$, then $\{g_i\}_{i\in \mathbb{N}^+}$ is a family of smooth functions on $\mathbb{R}$ satisfying:

(1) $\text{supp}(g)\subset [-\frac{1}{i},\frac{1}{i}]$, $g_i(x)\ge 0$ for any $x\in\mathbb{R}$,

(2) $\int_{-\frac{1}{i}}^0 g_i(x)dx=1$, $\int^{\frac{1}{i}}_0 g_i(x)dx\le\frac{1}{i}$ for any $i \in \mathbb{N}^+$.

Let $\tilde{h}(t)$ be an extension of the function $c(t)e^{-t}$ from $[T,+\infty)$ to $\mathbb{R}$ such that

(1) $\tilde{h}(t)=h(t):=c(t)e^{-t}$ on $[T,+\infty)$;

(2) $\tilde{h}(t)$ is decreasing with respect to $t$;

(3) $\lim_{t\to T-0}\tilde{h}(t)=c(T)e^{-T}$.

Denote $c_{i}(t):=e^t\int_{\mathbb{R}}\tilde{h}(t+y)g_{i}(y)dy$. By the construction of convolution, we know $c_{i}(t)\in C^{\infty}(-\infty,+\infty)$.  For any $t\ge T$, we have
$$c_i(t)-c(t)\ge e^t\left(\int_{-\frac{1}{i}}^0(\tilde{h}(t+y)-\tilde{h}(t))g_i(y)dy\right)\ge 0.$$

As $\tilde{h}(t)$ is decreasing with respect to $t$, we know that $c_i(t)e^{-t}$ is also decreasing with respect to $t$. Hence $c_i(t)e^{-t}$ is locally $L^1$ integrable on $\mathbb{R}$.

As $\tilde{h}(t)$ is decreasing with respect to $t$, then set $\tilde{h}^{-}(t)=\lim\limits_{s\to t-0}\tilde{h}(s)\ge h(t)$ for any $t\in\mathbb{R}$. Note that $c^{-}(t):=\lim\limits_{s\to t-0}\tilde{h}(s)e^t\ge c(t)$ for any $t\ge T$.

Now we prove $\lim\limits_{i\to +\infty}c_i(t)e^{-t}=\tilde{h}^{-}(t)$. In fact, we have
 \begin{equation}
\begin{split}
|c_i(t)e^{-t}-\tilde{h}^{-}(t)|&\le \int_{-\frac{1}{i}}^0|\tilde{h}(t+y)-h^{-}(t)|g_i(y)dy\\
&+\int_{0}^{\frac{1}{i}}\tilde{h}(t+y)g_i(y)dy.
\label{clt}
\end{split}
\end{equation}
For any $\epsilon>0 $, there exists $\delta>0$ such that $|h(t-\delta)-h^{-}(t)|<\epsilon$. Then $\exists N>0$, such that for any $n>N$, $t\ge t+y>t-\delta$ for all $y \in [-\frac{1}{i},0)$ and $\frac{1}{i}<\epsilon$. It follows from \eqref{clt} that
$$|c_i(t)e^{-t}-\tilde{h}^{-}(t)|\le \epsilon +\epsilon \tilde{h}(t),$$
hence $\lim\limits_{i\to +\infty}c_i(t)e^{-t}=\tilde{h}^{-}(t)$ for any $t\in\mathbb{R}$. Especially, we have $\lim\limits_{i\to +\infty}c_i(T)e^{-T}=\tilde{h}^{-}(T)=c(T)e^{-T}$.

Assume that for each $i$, we have a holomorphic $(n,0)$ form $\tilde{F}_i$ on $M$ such that
\begin{equation}\label{estimate for F_n}
  \begin{split}
      & \int_{M}|\tilde{F}_i-(1-b_{t_0,B}(\Psi))fF^{1+\delta}|^2e^{-\varphi+v_{t_0,B}(\Psi)-\Psi}c_i(-v_{t_0,B}(\Psi)) \\
      \le & \left(\frac{1}{\delta}c_i(T)e^{-T}+\int_{T}^{t_0+B}c_i(s)e^{-s}ds\right)
       \int_{M}\frac{1}{B}\mathbb{I}_{\{-t_0-B<\Psi<-t_0\}}|f|^2e^{-\varphi_{\alpha}-\Psi}.
  \end{split}
\end{equation}

By construction of $c_i(t)$, we have
\begin{equation}
\begin{split}
&\int_T^{t_0+B}c_i(t_1)e^{-t_1}dt_1\\
=&\int_T^{t_0+B}\int_{\mathbb{R}}\tilde{h}(t_1+y)g_i(y)dydt_1\\
=&\int_{\mathbb{R}}g_i(y)\left(\int_{T}^{t_0+B}\tilde{h}(t_1+y)dt_1\right)dy\\
=&\int_{\mathbb{R}}g_i(y)\left(\int_{T+y}^{t_0+B+y}\tilde{h}(s)ds\right)dy\\
=&\int_{\mathbb{R}}g_i(y)\left(\int_{T}^{t_0+B}\tilde{h}(s)ds
+\int_{t_0+B}^{t_0+B+y}\tilde{h}(s)ds-\int_{T}^{T+y}\tilde{h}(s)ds\right)dy,
\label{integral of cn}
\end{split}
\end{equation}
then it follows from the construction of $g_i(t)$, $\tilde{h}(t)$ is decreasing with respect to $t$, inequality \eqref{integral of cn} and $\tilde{h}(t)=c(t)e^{-t}$ on $[T,+\infty)$ that we have
\begin{equation}
\begin{split}
\lim_{i\to+\infty}\int_T^{t_0+B}c_i(t_1)e^{-t_1}dt_1=\int_{T}^{t_0+B}c(t_1)e^{-t_1}dt_1.
\label{limitof integral of cn}
\end{split}
\end{equation}
For any compact subset $K$ of $M$, we have $\inf\limits_i\inf\limits_K e^{v_{t_0,B}(\Psi)-\varphi-\Psi}c_i(-v_{t_0,B}(\Psi))\ge\inf\limits_K e^{v_{t_0,B}(\Psi)-\varphi-\Psi}c(-v_{t_0,B}(\Psi))$, then
\begin{equation}\nonumber
\begin{split}
\sup\limits_i\int_K|\tilde{F}_i-(1-b_{t_0,B}(\Psi))fF^{1+\delta}|^2<+\infty.
\end{split}
\end{equation}
Note that
\begin{equation}\nonumber
\begin{split}
\int_K|(1-b_{t_0,B}(\Psi))fF^{1+\delta}|^2\le(\sup_K|F^{1+\delta}|^2)\int_{K\cap \{\psi<-t_0\}} |f|^2<+\infty,
\end{split}
\end{equation}
then $\sup\limits_i\int_K|\tilde{F}_i|^2<+\infty$, which implies that there exists a subsequence of $\{\tilde{F}_i\}$ (also denoted by $\{\tilde{F}_i\}$), which is compactly convergent to a holomorphic $(n,0)$ form $\tilde{F}$ on $M$. Then it follows from inequality \eqref{estimate for F_n} and Fatou's Lemma that
\begin{equation}\nonumber
  \begin{split}
  &\int_{M}|\tilde{F}-(1-b_{t_0,B}(\Psi))fF^{1+\delta}|^2e^{-\varphi+v_{t_0,B}(\Psi)-\Psi}c(-v_{t_0,B}(\Psi)) \\
 \le &\int_{M}|\tilde{F}-(1-b_{t_0,B}(\Psi))fF^{1+\delta}|^2e^{-\varphi+v_{t_0,B}(\Psi)-\Psi}c^-(-v_{t_0,B}(\Psi)) \\
      \le&\liminf_{i\to+\infty}\int_{M}|\tilde{F}_i-(1-b_{t_0,B}(\Psi))fF^{1+\delta}|^2e^{-\varphi+v_{t_0,B}(\Psi)-\Psi}c_i(-v_{t_0,B}(\Psi)) \\
      \le & \liminf_{i\to+\infty}\left(\frac{1}{\delta}c_i(T)e^{-T}+\int_{T}^{t_0+B}c_i(s)e^{-s}ds\right)
       \int_{M}\frac{1}{B}\mathbb{I}_{\{-t_0-B<\Psi<-t_0\}}|f|^2e^{-\varphi_{\alpha}-\Psi}\\
       = & \left(\frac{1}{\delta}c(T)e^{-T}+\int_{T}^{t_0+B}c(s)e^{-s}ds\right)
       \int_{M}\frac{1}{B}\mathbb{I}_{\{-t_0-B<\Psi<-t_0\}}|f|^2e^{-\varphi_{\alpha}-\Psi}.
  \end{split}
\end{equation}

\end{Remark}

In the following discussion, we assume that $F$ has no zero points on $M$ and $c(t)$ is smooth.

As $M$ is weakly pseudoconvex, there exists a smooth plurisubharmonic
exhaustion function $P$ on $M$. Let $M_j:=\{P<j\}$ $(k=1,2,...,) $. We choose $P$ such that
$M_1\ne \emptyset$.\par
Then $M_j$ satisfies $M_1 \Subset  M_2\Subset  ...\Subset
M_j\Subset  M_{j+1}\Subset  ...$ and $\cup_{j=1}^n M_j=M$. Each $M_j$ is weakly
pseudoconvex K\"ahler manifold with exhaustion plurisubharmonic function
$P_j=1/(j-P)$.
\par
\emph{We will fix $j$ during our discussion until step 8.}

\

\emph{Step 1: Regularization of $\Psi$ and $\varphi_\alpha+\psi$. }

We note that there must exists a continuous nonnegative $(1,1)$-form $\varpi$ on $M_{j+1}$ satisfying
\begin{equation}\nonumber
(\sqrt{-1}\Theta_{TM}+\varpi \otimes Id_{TM})(\kappa_1 \otimes \kappa_2,\kappa_1
\otimes \kappa_2)\ge 0,
\end{equation}
for $\forall \kappa_1,\kappa_2 \in TM$ on $M_{j+1}$.

\par
Let $M=M_{j+1}$, $\Omega=M_{j}$, $T=\frac{\sqrt{-1}}{\pi}\partial\bar{\partial}\psi$
, $\gamma =0$ in Lemma \ref{regularization on cpx mfld}, then there exists a family of functions $\psi_{\zeta,\rho}$ ($\zeta\in(0,+\infty)$ and $\rho\in(0,\rho_1)$ for some positive $\rho_1$)
on $M_{j+1}$ such that\\
(1) $\psi_{\zeta,\rho}$ is a quasi-plurisubharmonic function on a neighborhood of $\overline{M_j}$, smooth on $M_{j+1}\backslash E_{\zeta}(\psi)$, increasing with respect to $\zeta$ and $\rho$ on $M_j$ and converges to $\psi$ on $M_j$ as $\rho \to 0$,\\
(2) $\frac{\sqrt{-1}}{\pi}\partial\bar{\partial}\psi_{\zeta,\rho}\ge-\zeta\varpi-\delta\omega$ on $M_j$,\\
where $E_{\zeta}(\psi):=\{x\in M: v(\psi,x)\ge \zeta\}$ is the upper-level set of Lelong number and $\{\delta_{\rho}\}$ is an increasing family of positive numbers such that $\lim_{\rho\to 0}\delta_{\rho}=0$.

Let $\rho=\frac{1}{m}$. Let $\tilde{\delta}_m:=\delta_{\frac{1}{m}}$ and $\zeta=\tilde{\delta}_m$. Denote $\psi_m:=\psi_{\tilde{\delta}_m,\frac{1}{m}}$. Then we have a sequence of functions $\{\psi_m\}$ satisfying\\
(1') $\psi_m$ is quasi-plurisubharmonic function on $\overline{M_j}$, smooth on $M_{j+1}\backslash E_{m}(\psi)$, decreasing with respect to $m$ and converges to $\psi$ on $M_j$ as $m\to+\infty$,\\
(2') $\frac{\sqrt{-1}}{\pi}\partial\bar{\partial}\psi_{m}\ge-\tilde{\delta}_m\varpi-\tilde{\delta}_m\omega$ on $M_j$,\\
where $E_{m}(\psi)=\{x\in X:v(\psi,x)\ge\frac{1}{m}\}$ is the upper level set of Lelong number and $\{\tilde{\delta}_m\}$ is an decreasing family of positive numbers such that $\lim_{m\to+\infty}\tilde{\delta}_m=0$.

As $M_j$ is relatively compact in $M$, there exists a positive number $b\ge 1$ such that $b\omega\ge\varpi$ on $M_j$. Then condition (2') becomes\\
(2'') $\frac{\sqrt{-1}}{\pi}\partial\bar{\partial}\psi_{m}\ge-\tilde{\delta}_m\varpi-\tilde{\delta}_m\omega
\ge-2b\tilde{\delta}_m\omega$ on $M_j$.

Denote $h:=\varphi_{\alpha}+\psi$. Note that $h$ is a plurisubharmonic function on $M$. Denote $h_l:=\max\{h,-l\}$, where $l\in\mathbb{Z}^+$. Note that $h_l$ is a plurisubharmonic function.  As $h_l\ge -l$, we know $v(h_l,z)=0$ for all $z\in M$. By using the similarly discussion as above, we have a sequence of functions $\{h_{m',l}\}$ on $M_{j+1}$ such that\\
(i) $h_{m',l}$ is quasi-plurisubharmonic function on $\overline{M_j}$, smooth on $M_{j+1}$, decreasing with respect to $m'$ and converges to $h_l$ on $M_j$ as $m'\to+\infty$,\\
(ii) $\frac{\sqrt{-1}}{\pi}\partial\bar{\partial}h_{m',l}
\ge-2b\delta_{m',l}\omega$ on $M_j$,\\
where  $\{\delta_{m',l}\}$ is an decreasing family of positive numbers such that $\lim_{m'\to+\infty}\delta_{m',l}=0$.

\emph{From now on, we will fix the positive integer $l$ during our discussion until step 7.}

For fixed $l$, we can assume that $\tilde{\delta}_n$ and $\delta_{n,l}$ are the same sequence of variable $n\in \mathbb{Z}^+$, since we can replace them by the term $\max\{\tilde{\delta}_n,\delta_{n,l}\}$. We denote both $\tilde{\delta}_n$ and $\delta_{n,l}$ by $\tilde{\delta}_n$ for simplicity.

Let $\eta_m=\{\frac{t_0-T}{3m},\frac{t_0-T}{3m}\}$ and we have the function $M_{\eta_m}(\psi_m+T,2\log|F|)$. Denote $M_{\eta_m}:=M_{\eta_m}(\psi_m+T,2\log|F|)$ for simplicity.  Note that $\psi_m+T$ is a $2b\tilde{\delta}_m\omega$-plurisubharmonic function. As $F$ is a holomorphic function, $\omega$ is a K\"ahler form and $b\tilde{\delta}_m>0$, we know that $2\log|F|$ is a $2b\tilde{\delta}_m\omega$-plurisubharmonic function. It follows from Lemma \ref{regular of max} that $M_{\eta_m}$ is a $2b\tilde{\delta}_m\omega$-plurisubharmonic function, i.e.,
$$\frac{\sqrt{-1}}{\pi}\partial\bar{\partial}M_{\eta_m}\ge -2\pi b\tilde{\delta}_m\omega.$$

Denote $\Psi_m:=\psi_m-M_{\eta_m}(\psi_m+T,2\log|F|)$. Then $\Psi_m$ is smooth on $M_j\backslash E_m$. It is easy to verify that when $m\to+\infty$, $\Psi_m\to \Psi$. It follows from Lemma \ref{regular of max} that we know\par
(1) if $\psi_m+T\le 2\log|F|-\frac{2(t_0-T)}{3m}$ holds, we have $\Psi_m=\psi_m-2\log|F|$;\par
(2) if $\psi_m+T\ge 2\log|F|+\frac{2(t_0-T)}{3m}$ holds, we have $\Psi_m=-T$;\par
(3) if $2\log|F|-\frac{2(t_0-T)}{3m}<\psi_m+T< 2\log|F|+\frac{2(t_0-T)}{3m}$ holds, we have $\max\{\psi_m+T,2\log|F|\}\le M_{\eta_m}\le (\psi_m+T+\frac{t_0-T}{m})$ and hence $-T-\frac{t_0-T}{m}\le\Psi_m\le -T$.

Thus we have $\{\Psi_m<-t_0\}= \{\psi_m-2\log|F|<-t_0\}\subset \{\psi-2\log|F|<-t_0\}=\{\Psi<-t_0\}$. We also note that $\Psi_m\le-T$ on $M_{j+1}$.

\

\emph{Step 2: Recall some constructions. }

To simplify our notations, we denote $b_{t_0,B}(t)$ by $b(t)$ and $v_{t_0,B}(t)$ by $v(t)$.

Let $\epsilon \in (0,\frac{1}{8}B)$. Let $\{v_\epsilon\}_{\epsilon \in
(0,\frac{1}{8}B)}$ be a family of smooth increasing convex functions on $\mathbb{R}$, such
that:
\par
(1) $v_{\epsilon}(t)=t$ for $t\ge-t_0-\epsilon$, $v_{\epsilon}(t)=constant$ for
$t<-t_0-B+\epsilon$;\par
(2) $v_{\epsilon}{''}(t)$ are convergence pointwise
to $\frac{1}{B}\mathbb{I}_{(-t_0-B,-t_0)}$,when $\epsilon \to 0$, and $0\le
v_{\epsilon}{''}(t) \le \frac{2}{B}\mathbb{I}_{(-t_0-B+\epsilon,-t_0-\epsilon)}$
for ant $t \in \mathbb{R}$;\par
(3) $v_{\epsilon}{'}(t)$ are convergence pointwise to $b(t)$ which is a continuous
function on R when $\epsilon \to 0$ and $0 \le v_{\epsilon}{'}(t) \le 1$ for any
$t\in \mathbb{R}$.\par
One can construct the family $\{v_\epsilon\}_{\epsilon \in (0,\frac{1}{8}B)}$  by
 setting
\begin{equation}\nonumber
\begin{split}
v_\epsilon(t):=&\int_{-\infty}^{t}(\int_{-\infty}^{t_1}(\frac{1}{B-4\epsilon}
\mathbb{I}_{(-t_0-B+2\epsilon,-t_0-2\epsilon)}*\rho_{\frac{1}{4}\epsilon})(s)ds)dt_1\\
&-\int_{-\infty}^{-t_0}(\int_{-\infty}^{t_1}(\frac{1}{B-4\epsilon}
\mathbb{I}_{(-t_0-B+2\epsilon,-t_0-2\epsilon)}*\rho_{\frac{1}{4}\epsilon})(s)ds)dt_1-t_0,
\end{split}
\end{equation}
where $\rho_{\frac{1}{4}\epsilon}$ is the kernel of convolution satisfying
$\text{supp}(\rho_{\frac{1}{4}\epsilon})\subset
(-\frac{1}{4}\epsilon,{\frac{1}{4}\epsilon})$.
Then it follows that
\begin{equation}\nonumber
v_\epsilon{''}(t)=\frac{1}{B-4\epsilon}
\mathbb{I}_{(-t_0-B+2\epsilon,-t_0-2\epsilon)}*\rho_{\frac{1}{4}\epsilon}(t),
\end{equation}
and
\begin{equation}\nonumber
v_\epsilon{'}(t)=\int_{-\infty}^{t}(\frac{1}{B-4\epsilon}
\mathbb{I}_{(-t_0-B+2\epsilon,-t_0-2\epsilon)}*\rho_{\frac{1}{4}\epsilon})(s)ds.
\end{equation}

Let $\eta=s(-v_\epsilon(\Psi_m))$ and $\phi=u(-v_\epsilon(\Psi_m))$, where $s \in
C^{\infty}([T,+\infty))$ satisfies $s\ge \frac{1}{\delta}$ and $u\in C^{\infty}([T,+\infty))$, such that $s'(t)\neq 0$ for any $t$, $u''s-s''>0$
and $s'-u's=1$.
Let $\Phi=\phi+h_{m',l}+\delta M_{\eta_m}$. Denote $\tilde{h}:=e^{-\Phi}$.

\

\emph{Step 3: Solving $\bar{\partial}$-equation with error term. }

Set $B=[\eta \sqrt{-1}\Theta_{\widetilde{h}}-\sqrt{-1}\partial \bar{\partial}
\eta-\sqrt{-1}g\partial\eta \wedge\bar{\partial}\eta, \Lambda_{\omega}]$, where
$g$
is a positive function. We will determine $g$ by calculations. On $M_j\backslash E_m$, direct calculation shows that
\begin{equation}\nonumber
\begin{split}
\partial\bar{\partial}\eta=&
-s'(-v_{\epsilon}(\Psi_m))\partial\bar{\partial}(v_{\epsilon}(\Psi_m))
+s''(-v_{\epsilon}(\Psi_m))\partial(v_{\epsilon}(\Psi_m))\wedge
\bar{\partial}(v_{\epsilon}(\Psi_m)),\\
\eta\Theta_{\widetilde{h}}=&\eta\partial\bar{\partial}\phi+\eta\partial\bar{\partial}h_{m',l}+
\eta\partial\bar{\partial}(\delta M_{\eta_m})\\
=&su''(-v_{\epsilon}(\Psi_m))\partial(v_{\epsilon}(\Psi_m))\wedge
\bar{\partial}(v_{\epsilon}(\Psi_m))
-su'(-v_{\epsilon}(\Psi_m))\partial\bar{\partial}(v_{\epsilon}(\Psi_m))\\
+&s\partial\bar{\partial}h_{m',l}+
s\partial\bar{\partial}(\delta M_{\eta_m}).
\end{split}
\end{equation}
\par
Hence
\begin{equation}\nonumber
\begin{split}
&\eta \sqrt{-1}\Theta_{\widetilde{h}}-\sqrt{-1}\partial \bar{\partial}
\eta-\sqrt{-1}g\partial\eta \wedge\bar{\partial}\eta\\
=&s\sqrt{-1}\partial\bar{\partial}h_{m',l}+
s\sqrt{-1}\partial\bar{\partial}(\delta M_{\eta_m})\\
+&(s'-su')(v'_{\epsilon}(\Psi_m)\sqrt{-1}\partial\bar{\partial}(\Psi_m)+
v''_\epsilon(\psi_m)\sqrt{-1}\partial(\Psi_m)\wedge\bar{\partial}(\Psi_m))\\
+&[(u''s-s'')-gs'^2]\sqrt{-1}\partial(v_\epsilon(\Psi_m))\wedge\bar{\partial}(v_\epsilon(\Psi_m)),
\end{split}
\end{equation}
where we omit the term $-v_{\epsilon}(\Psi_m)$ in $(s'-su')(-v_{\epsilon}(\Psi_m))$ and $[(u''s-s'')-gs'^2](-v_{\epsilon}(\Psi_m))$ for simplicity.
\par
Let $g=\frac{u''s-s''}{s'^2}(-v_\epsilon(\Psi_m))$ and note that $s'-su'=1$,
$0\le v'_\epsilon(\Psi_m)\le0$. Then
\begin{equation}\label{calculation of curvature 1}
\begin{split}
&\eta \sqrt{-1}\Theta_{\widetilde{h}}-\sqrt{-1}\partial \bar{\partial}
\eta-\sqrt{-1}g\partial\eta \wedge\bar{\partial}\eta\\
=&s\sqrt{-1}\partial\bar{\partial}h_{m',l}+
s\sqrt{-1}\partial\bar{\partial}(\delta M_{\eta_m})
+v'_{\epsilon}(\Psi_m)\sqrt{-1}\partial\bar{\partial}(\Psi_m)+
v''_\epsilon(\psi_m)\sqrt{-1}\partial(\Psi_m)\wedge\bar{\partial}(\Psi_m)\\
=&v''_\epsilon(\psi_m)\sqrt{-1}\partial(\Psi_m)\wedge\bar{\partial}(\Psi_m)+
v'_{\epsilon}(\Psi_m)\sqrt{-1}\partial\bar{\partial}(\Psi_m)\\
&+s(\sqrt{-1}\partial\bar{\partial}h_{m',l}+2\pi b\tilde{\delta}_{m'}\omega)-2\pi b s \tilde{\delta}_{m'}\omega
+s(\sqrt{-1}\partial\bar{\partial}(\delta M_{\eta_m})+2\pi b\delta\tilde{\delta}_{m}\omega)-2\pi b s \delta\tilde{\delta}_{m}\omega\\
\ge&v''_\epsilon(\psi_m)\sqrt{-1}\partial(\Psi_m)\wedge\bar{\partial}(\Psi_m)+
v'_{\epsilon}(\Psi_m)\sqrt{-1}\partial\bar{\partial}(\Psi_m)\\
&+\frac{1}{\delta}(\sqrt{-1}\partial\bar{\partial}h_{m',l}+2\pi b\tilde{\delta}_{m'}\omega)
+\frac{1}{\delta}(\sqrt{-1}\partial\bar{\partial}(\delta M_{\eta_m})+2\pi b\delta\tilde{\delta}_{m}\omega)-2\pi b s(\tilde{\delta}_{m'}+ \delta\tilde{\delta}_{m})\omega\\
\end{split}
\end{equation}
Note that
\begin{equation}\label{calculation of curvature 2}
\begin{split}
&\delta v'_{\epsilon}(\Psi_m)\sqrt{-1}\partial\bar{\partial}(\Psi_m)
+(\sqrt{-1}\partial\bar{\partial}h_{m',l}+2\pi b\tilde{\delta}_{m'}\omega)
+(\sqrt{-1}\partial\bar{\partial}(\delta M_{\eta_m})+2\pi b\delta\tilde{\delta}_{m}\omega)\\
=&(1-v'_{\epsilon}(\Psi_m))(\sqrt{-1}\partial\bar{\partial}h_{m',l}+2\pi b\tilde{\delta}_{m'}\omega+\sqrt{-1}\partial\bar{\partial}(\delta M_{\eta_m})+2\pi b\delta\tilde{\delta}_{m}\omega)\\
&+v'_{\epsilon}(\Psi_m)(\sqrt{-1}\partial\bar{\partial}h_{m',l}+2\pi b\tilde{\delta}_{m'}\omega+\sqrt{-1}\partial\bar{\partial}(\delta M_{\eta_m})+2\pi b\delta\tilde{\delta}_{m}\omega)\\
&+v'_{\epsilon}(\Psi_m)(\partial\bar{\partial}(\delta\psi_m)-\partial\bar{\partial}(\delta M_{\eta_m}))\\
=&(1-v'_{\epsilon}(\Psi_m))(\sqrt{-1}\partial\bar{\partial}h_{m',l}+2\pi b\tilde{\delta}_{m'}\omega+\sqrt{-1}\partial\bar{\partial}(\delta M_{\eta_m})+2\pi b\delta\tilde{\delta}_{m}\omega)\\
&+v'_{\epsilon}(\Psi_m)(\sqrt{-1}\partial\bar{\partial}h_{m',l}+2\pi b\tilde{\delta}_{m'}\omega+\sqrt{-1}\partial\bar{\partial}(\delta \psi_m)+2\pi b\delta\tilde{\delta}_{m}\omega)\\
\ge& 0.
\end{split}
\end{equation}
It follows from inequality \eqref{calculation of curvature 1} and inequality \eqref{calculation of curvature 2} that
\begin{equation}\nonumber
\begin{split}
&\eta \sqrt{-1}\Theta_{\widetilde{h}}-\sqrt{-1}\partial \bar{\partial}
\eta-\sqrt{-1}g\partial\eta \wedge\bar{\partial}\eta\\
\ge&v''_\epsilon(\Psi_m)\sqrt{-1}\partial(\Psi_m)\wedge\bar{\partial}(\Psi_m)
-2\pi b s(\tilde{\delta}_{m'}+ \delta\tilde{\delta}_{m})\omega.
\end{split}
\end{equation}

By the constructions of $s(t)$, $v_{\epsilon}(t)$ and $\sup_m\sup_{M_j}\Psi_m\le-T$, we have $s(-v_{\epsilon}(\Psi_m))$ is uniformly bounded on $M_j$ with respect to $\epsilon$ and $m$. Let $\tilde{M}$ be the uniformly upper bound of $s(-v_{\epsilon}(\Psi_m))$ on $M_j$. Then on $M_j\backslash E_m$, we have

\begin{equation}\nonumber
\begin{split}
&\eta \sqrt{-1}\Theta_{\widetilde{h}}-\sqrt{-1}\partial \bar{\partial}
\eta-\sqrt{-1}g\partial\eta \wedge\bar{\partial}\eta\\
\ge&v''_\epsilon(\Psi_m)\sqrt{-1}\partial(\Psi_m)\wedge\bar{\partial}(\Psi_m)
-2\pi b \tilde{M}(\tilde{\delta}_{m'}+ \delta\tilde{\delta}_{m})\omega.
\end{split}
\end{equation}

Hence, for any $(n,1)$ form $\alpha$, we have
\begin{equation}
\label{curvature inequality}
\begin{split}
&\langle(B+2\pi b\tilde{M}(\tilde{\delta}_{m'}+ \delta\tilde{\delta}_{m})I)\alpha,\alpha\rangle_{\widetilde h}\\
\ge&\langle[v''_\epsilon(\Psi_m)\partial(\Psi_m)\wedge\bar{\partial}(\Psi_m),
\Lambda_{\omega}]\alpha,\alpha\rangle_{\widetilde h}\\
=&\langle(v''_\epsilon(\Psi_m)\bar{\partial}(\Psi_m)
\wedge(\alpha\llcorner(\bar{\partial}\Psi_m)^{\sharp})),\alpha\rangle_{\widetilde
h}.
\end{split}
\end{equation}
It follows from Lemma \ref{sempositive lemma} that $B+2\pi b\tilde{M}(\tilde{\delta}_{m'}+ \delta\tilde{\delta}_{m})I$ is semi-positive. Using the definition of contraction, Cauchy-Schwarz inequality
and inequality \eqref{curvature inequality}, we have
\begin{equation}
\label{cs inequality}
\begin{split}
|\langle
v''_\epsilon(\Psi_m)\bar{\partial}\Psi_m\wedge\gamma,\widetilde{\alpha}\rangle_
{\widetilde h}|^2=
&|\langle
v''_\epsilon(\Psi_m)\gamma,\widetilde{\alpha}\llcorner(\bar{\partial}\Psi_m)^{\sharp}
\rangle_{\widetilde h}|^2\\
\le&\langle
(v''_\epsilon(\Psi_m)\gamma,\gamma)
\rangle_{\widetilde h}
(v''_\epsilon(\Psi_m))|\widetilde{\alpha}\llcorner(\bar{\partial}\Psi_m)^{\sharp}|^2_{\widetilde
h}\\
=&\langle
(v''_\epsilon(\Psi_m)\gamma,\gamma)
\rangle_{\widetilde h}
\langle
(v''_\epsilon(\Psi_m))\bar{\partial}\Psi_m\wedge
(\widetilde{\alpha}\llcorner(\bar{\partial}\Psi_m)^{\sharp}),\widetilde{\alpha}
\rangle_{\widetilde h}\\
\le&\langle
(v''_\epsilon(\Psi_m)\gamma,\gamma)
\rangle_{\widetilde h}
\langle
(B+2\pi b\tilde{M}(\tilde\delta_{m'}+\delta\tilde\delta_m)I)\widetilde{\alpha},\widetilde{\alpha})
\rangle_{\widetilde h}
\end{split}
\end{equation}
for any $(n,0)$ form $\gamma$ and $(n,1)$ form $\widetilde{\alpha}$.

As $fF^{1+\delta}$ is holomorphic on $\{\Psi<-t_0\}$ and $\{\Psi_m<-t_0-\epsilon\}\subset \{\Psi_m<-t_0\}\subset\{\Psi<-t_0\}$, then $\lambda:=\bar{\partial}\big((1-v'_{\epsilon}(\Psi_m))fF^{1+\delta}\big)$ is well defined and smooth on $M_j\backslash E_m$.

Taking $\gamma=fF^{1+\delta}$, $\tilde{\alpha}=(B+2\pi b\tilde{M}(\tilde\delta_{m'}+\delta\tilde\delta_m)I)^{-1}(\bar{\partial}v'_{\epsilon}(\Psi_m))\wedge fF^{1+\delta}$. Then it follows from inequality \eqref{cs inequality} that
$$\langle
(B+2\pi b\tilde{M}(\tilde\delta_{m'}+\delta\tilde\delta_m)I)^{-1}\lambda,\lambda\rangle_
{\widetilde h}\le v''_\epsilon(\Psi_m)|fF^{1+\delta}|^2e^{-\Phi}.$$

Thus we have
$$\int_{M_j\backslash E_m}\langle
(B+2\pi b\tilde{M}(\tilde\delta_{m'}+\delta\tilde\delta_m)I)^{-1}\lambda,\lambda\rangle_
{\widetilde h}\le \int_{M_j\backslash E_m}v''_\epsilon(\Psi_m)|fF^{1+\delta}|^2e^{-\Phi}.$$

Recall that $e^{-\Phi}=e^{-\phi-h_{m',l}-\delta M_{\eta_m}}$. Note that $h_{m',l}\ge -l$ and $\delta M_{\eta_m}\ge \delta 2\log|F|$. By the construction of $\phi$, we know that $\sup\limits_{\overline{M_j}}e^{-\phi}<+\infty$.
Then $$\int_{M_j\backslash E_m}v''_\epsilon(\Psi_m)|fF^{1+\delta}|^2e^{-\Phi}
\le
\sup\limits_{\overline{M_j}}\left(e^{-\phi+l}|F|^2\right)\int_{\overline{M_j}}\frac{2}{B}
\mathbb{I}_{\{\Psi\le -t_0\}}|f|^2<+\infty.$$

By Lemma \ref{completeness}, $M_j\backslash E_m$ carries a complete K\"ahler metric. Then it follows from Lemma \ref{d-bar equation with error term} that there exists
$$u_{m,m',l,\epsilon,j}\in L^2(M_j\backslash E_m, K_M),$$
$$h_{m,m',l,\epsilon,j}\in L^2(M_j\backslash E_m, \wedge^{n,1}T^*M)$$
such that $\bar\partial u_{m,m',l,\epsilon,j}+\sqrt{2\pi b\tilde{M}(\tilde\delta_{m'}+\delta\tilde\delta_m)}h_{m,m',l,\epsilon,j}=\lambda$ holds on $M_j\backslash E_m$, and

\begin{equation}\nonumber
\begin{split}
&\int_{M_j \backslash E_m}
\frac{1}{\eta+g^{-1}}|u_{m,m',l,\epsilon,j}|^2 e^{-\Phi}+
\int_{M_j \backslash E_m}
|h_{m,m',l,\epsilon,j}|^2 e^{-\Phi}\\
\le&
\int_{M_j \backslash E_m}
\langle
(B+2\pi b\tilde{M}(\tilde \delta_{m'}+\delta\tilde\delta_m)I)^{-1}\lambda,\lambda
\rangle_{\widetilde h}\\
\le&
\int_{M_j \backslash E_m}
v''_{\epsilon}(\Psi_m)|fF^{1+\delta}|^2 e^{-\Phi}
<+\infty.
\end{split}
\end{equation}

Assume that we can choose
$\eta$ and $\phi$ such that
$(\eta+g^{-1})^{-1}=e^{v_\epsilon(\Psi_m)}e^{\phi}c(-v_\epsilon(\Psi_m))$. Then we have

\begin{equation}
\label{estimate 1}
\begin{split}
&\int_{M_j \backslash E_m}
|u_{m,m',l,\epsilon,j}|^2 e^{v_\epsilon(\Psi_m)-h_{m',l}-\delta M_{\eta_m}}c(-v_\epsilon(\Psi_m))+
\int_{M_j \backslash E_m}
|h_{m,m',l,\epsilon,j}|^2 e^{-\phi-h_{m',l}-\delta M_{\eta_m}}\\
\le&
\int_{M_j \backslash E_m}
v''_{\epsilon}(\Psi_m)|fF^{1+\delta}|^2 e^{-\phi-h_{m',l}-\delta M_{\eta_m}}
<+\infty.
\end{split}
\end{equation}

By the construction of $v_{\epsilon}(t)$ and $c(t)e^{-t}$ is decreasing with respect to $t$, we know $c(-v_{\epsilon}(\Psi_m))e^{v_{\epsilon}(\Psi_m)}$ has a positive lower bound on $M_j\Subset M$. By the constructions of $v_{\epsilon}(t)$ and $u$, we know $e^{-\phi}=e^{-u(-v_{\epsilon}(\Psi_m))}$ has a positive lower bound on $M_j\Subset M$. By the upper semi-continuity of $M_{\eta_m}$, we know $e^{-\delta M_{\eta_m}}$ has a positive lower bound on $M_j\Subset M$. Note that $h_{m',l}$ is smooth on $M_j\Subset M$. Hence it follows from inequality \eqref{estimate 1} that

$$u_{m,m',l,\epsilon,j}\in L^2(M_j, K_M),$$
$$h_{m,m',l,\epsilon,j}\in L^2(M_j, \wedge^{n,1}T^*M).$$
It follows from Lemma \ref{extension of equality} that we know
\begin{equation}
\label{d-bar realation u,h,lamda 1}
\bar\partial u_{m,m',l,\epsilon,j}+\sqrt{2\pi b\tilde{M}(\tilde\delta_{m'}+\delta\tilde\delta_m)}h_{m,m',l,\epsilon,j}=\lambda
\end{equation}
holds on $M_j$. And we have
\begin{equation}
\label{estimate 2}
\begin{split}
&\int_{M_j}
|u_{m,m',l,\epsilon,j}|^2 e^{v_\epsilon(\Psi_m)-h_{m',l}-\delta M_{\eta_m}}c(-v_\epsilon(\Psi_m))+
\int_{M_j}
|h_{m,m',l,\epsilon,j}|^2 e^{-\phi-h_{m',l}-\delta M_{\eta_m}}\\
\le&
\int_{M_j}
v''_{\epsilon}(\Psi_m)|fF^{1+\delta}|^2 e^{-\phi-h_{m',l}-\delta M_{\eta_m}}
<+\infty.
\end{split}
\end{equation}

\

\emph{Step 4: Letting $m\to+\infty$. }

Note that $\sup_m\sup_{M_j}e^{-\phi}=\sup_m\sup_{M_j}e^{-u(-v_{\epsilon}(\Psi_m))}<+\infty$, $e^{-h_{m',l}}\le e^{l}$ and $e^{-\delta M_{\eta_m}}\le e^{-\delta 2\log|F|}$. As $\{\Psi_m<-t_0-\epsilon\}\subset \{\Psi_m<-t_0\}\subset \{\Psi<-t_0\}$, we have
$$v''_{\epsilon}(\Psi_m)|fF^{1+\delta}|^2 e^{-\phi-h_{m',l}-\delta M_{\eta_m}}\le \frac{2}{B}\left(\sup_m\sup_{M_j}e^{-\phi}\right)e^{l}\left(\sup_{M_j}|F|^2\right)\mathbb{I}_{\{\Psi<-t_0\}}|f|^2$$
holds on $M_j$. It follows from $\int_{\{\Psi<-t_0\}\cap \overline{M_j}}|f|^2<+\infty$ and dominated convergence theorem that
\begin{equation}\nonumber
\begin{split}
&\lim_{m\to+\infty}\int_{M_j}
v''_{\epsilon}(\Psi_m)|fF^{1+\delta}|^2 e^{-\phi-h_{m',l}-\delta M_{\eta_m}}\\
=&\int_{M_j}v''_{\epsilon}(\Psi)|fF^{1+\delta}|^2 e^{-u(-v_{\epsilon}(\Psi))-h_{m',l}-\delta \max{\{\psi+T,2\log|F|\}}}.
\end{split}
\end{equation}
Note that
$$v''_{\epsilon}(\Psi)|fF^{1+\delta}|^2 e^{-u(-v_{\epsilon}(\Psi))-h_{m',l}-\delta \max{\{\psi+T,2\log|F|\}}}\le \frac{2}{B}\left(\sup_{M_j}e^{-u(-v_{\epsilon}(\Psi))}|F|^2\right)e^{l}\mathbb{I}_{\{\Psi<-t_0\}}|f|^2$$
holds on $M_j$. We have
$$\int_{M_j}v''_{\epsilon}(\Psi)|fF^{1+\delta}|^2 e^{-u(-v_{\epsilon}(\Psi))-h_{m',l}-\delta \max{\{\psi+T,2\log|F|\}}}<+\infty.$$

Note that $\inf_{m}\inf_{M_j}c(-v_{\epsilon}(\Psi_m))e^{-v_{\epsilon}(\Psi_m)-h_{m',l}}>0$.
It follows from Lemma \ref{regular of max} that $M_{\eta_m}\le \max{\{\psi_m+T,2\log|F|\}}+\frac{t_0-T}{3m}\le \max{\{\psi_m+T,2\log|F|\}}+t_0-T\le \max{\{\psi_1+T,2\log|F|\}}+t_0-T$. As $\psi_1$ is a quasi-plurisubharmonic function on $\overline{M_j}$, we know $\max{\{\psi_1+T,2\log|F|\}}$ is upper semi-continuous function on $M_j$. Hence
\begin{equation}
\label{estimate for M eta}
\inf_{m}\inf_{M_j}e^{-M_{\eta_m}}\ge \inf_{M_j}e^{-\max{\{\psi_1+T,2\log|F|\}}-t_0}>0.
\end{equation}
 Then it follows from inequality \eqref{estimate 2} that
$$\sup_{m}\int_{M_j}|u_{m,m',l,\epsilon,j}|^2<+\infty.$$

Therefore the solutions $u_{m,m',l,\epsilon,j}$ are uniformly bounded in $L^2$ norm
with respect to $m$ on $M_j$. Since the closed unit ball of the Hilbert space is
weakly compact, we can extract a subsequence $u_{m_1,m',l,\epsilon,j}$ weakly
convergent to $u_{m',l,\epsilon,j}$ in $L^2(M_j,K_M)$ as $m_1\to+\infty$.

Note that $\sup_m\sup_{M_j}e^{v_{\epsilon}(\Psi_m)}c(-v_{\epsilon}(\Psi_m))e^{-h_{m',l}}< +\infty$. As $M_{\eta_m}\ge \max\{\psi_m+T,2\log|F|\}\ge 2\log|F|$ and $F$ has no zero points on $M$, we have $\sup_m\sup_{M_j}e^{-M_{\eta_m}}\le \sup_{M_j}\frac{1}{|F|^2}<+\infty$. Hence we know
$$\sup_m\sup_{M_j}e^{v_{\epsilon}(\Psi_m)}c(-v_{\epsilon}(\Psi_m))e^{-h_{m',l}-\delta M_{\eta_m}}<+\infty.$$
It follows from Lemma \ref{weakly convergence} that we know $u_{m_1,m',l,\epsilon,j}\sqrt{e^{v_{\epsilon}(\Psi_{m_1})}c(-v_{\epsilon}(\Psi_{m_1}))e^{-h_{m',l}-\delta M_{\eta_{m_1}}}}$ weakly convergent to $u_{m',l,\epsilon,j}\sqrt{e^{v_{\epsilon}(\Psi)}c(-v_{\epsilon}(\Psi))e^{-h_{m',l}-\delta \max\{\psi+T,2\log|F|\}}}$. Hence we have

\begin{equation}
\label{estimate for m'}
\begin{split}
&\int_{M_j}
|u_{m',l,\epsilon,j}|^2 e^{v_\epsilon(\Psi)-h_{m',l}-\delta \max\{\psi+T,2\log|F|\}}c(-v_\epsilon(\Psi))\\
\le&\liminf_{m_1\to+\infty}\int_{M_j}
|u_{m_1,m',l,\epsilon,j}|^2 e^{v_\epsilon(\Psi_{m_1})-h_{m',l}-\delta M_{\eta_{m_1}}}c(-v_\epsilon(\Psi_{m_1}))\\
\le&\liminf_{m_1\to+\infty}
\int_{M_j}
v''_{\epsilon}(\Psi_{m_1})|fF^{1+\delta}|^2 e^{-u(-v_{\epsilon}(\Psi_{m_1}))-h_{m',l}-\delta M_{\eta_{m_1}}}\\
\le & \int_{M_j}v''_{\epsilon}(\Psi)|fF^{1+\delta}|^2 e^{-u(-v_{\epsilon}(\Psi))-h_{m',l}-\delta \max{\{\psi+T,2\log|F|\}}}
<+\infty.
\end{split}
\end{equation}

Note that $\inf_{m_1}\inf_{M_j}e^{-u(-v_{\epsilon}(\Psi_{m_1}))-h_{m',l}}>0$. Combining inequality \eqref{estimate 2} and inequality \eqref{estimate for M eta} we know
$$\sup_{m_1}\int_{M_j}
|h_{m_1,m',l,\epsilon,j}|^2 <+\infty.$$

Since the closed unit ball of the Hilbert space is
weakly compact, we can extract a subsequence of $\{h_{m_1,m',l,\epsilon,j}\}$ (also denote by $h_{m_1,m',l,\epsilon,j}$) weakly
convergent to $h_{m',l,\epsilon,j}$ in $L^2(M_j,\wedge^{n,1}T^*M)$ as $m_1\to+\infty$.

Note that $\sup_{m_1}\sup_{M_j}e^{-u(-v_{\epsilon}(\Psi_{m_1}))-h_{m',l}}<+\infty$ and $\sup_{m_1}\sup_{M_j}e^{-M_{\eta_{m_1}}}\le \sup_{M_j}\frac{1}{|F|^2}<+\infty$. We know $$\sup_{m_1}\sup_{M_j}e^{-u(-v_{\epsilon}(\Psi_{m_1}))-h_{m',l}-\delta M_{\eta_{m_1}}}<+\infty.$$
It follows from Lemma \ref{weakly convergence} that we have
$h_{m_1,m',l,\epsilon,j}\sqrt{e^{-u(-v_{\epsilon}(\Psi_{m_1}))-h_{m',l}-\delta M_{\eta_{m_1}}}}$ is weakly convergent to
$h_{m',l,\epsilon,j}\sqrt{e^{-u(-v_{\epsilon}(\Psi))-h_{m',l}-\delta \max{\{\psi+T,2\log|F|\}}}}$. Hence we have

\begin{equation}
\label{estimate for h m'}
\begin{split}
&\int_{M_j}
|h_{m',l,\epsilon,j}|^2e^{-u(-v_{\epsilon}(\Psi))-h_{m',l}-\delta \max{\{\psi+T,2\log|F|\}}}\\
\le&\liminf_{m_1\to+\infty}\int_{M_j}
|h_{m_1,m',l,\epsilon,j}|^2 e^{-u(-v_{\epsilon}(\Psi_{m_1}))-h_{m',l}-\delta M_{\eta_{m_1}}}\\
\le&\liminf_{m_1\to+\infty}
\int_{M_j}
v''_{\epsilon}(\Psi_{m_1})|fF^{1+\delta}|^2 e^{-u(-v_{\epsilon}(\Psi_{m_1}))-h_{m',l}-\delta M_{\eta_{m_1}}}\\
\le & \int_{M_j}v''_{\epsilon}(\Psi)|fF^{1+\delta}|^2 e^{-u(-v_{\epsilon}(\Psi))-h_{m',l}-\delta \max{\{\psi+T,2\log|F|\}}}
<+\infty.
\end{split}
\end{equation}

Replace $m$ by $m_1$ in \eqref{d-bar realation u,h,lamda 1} and let $m_1\to +\infty$, we have

\begin{equation}
\label{d-bar realation u,h,lamda 2}
\bar\partial u_{m',l,\epsilon,j}+\sqrt{2\pi b\tilde{M}\tilde\delta_{m'}}h_{m',l,\epsilon,j}
=\bar\partial\left((1-v'_{\epsilon}(\Psi))fF^{1+\delta}\right).
\end{equation}

\

\emph{Step 5: Letting $m'\to+\infty$. }

When $\Psi<-t_0-\epsilon<-t_0$, we have $\psi-2\log|F|<-T$ and then $\max{\{\psi+T,2\log|F|\}}=2\log|F|$. Hence
$$\int_{M_j}v''_{\epsilon}(\Psi)|fF^{1+\delta}|^2 e^{-u(-v_{\epsilon}(\Psi))-h_{m',l}-\delta \max{\{\psi+T,2\log|F|\}}}=\int_{M_j}v''_{\epsilon}(\Psi)|fF|^2 e^{-u(-v_{\epsilon}(\Psi))-h_{m',l}}.$$

Note that
 $$v''_{\epsilon}(\Psi)|fF|^2 e^{-u(-v_{\epsilon}(\Psi))-h_{m',l}}\le \frac{2}{B}\left(\sup_{M_j}e^{-u(-v_{\epsilon}(\Psi))+l}|F|^2\right)\mathbb{I}_{\{\Psi<-t_0\}}|f|^2.$$

It follows from $\int_{\{\Psi<-t_0\}\cap \overline{M_j}}|f|<+\infty$ and dominated convergence theorem that
\begin{equation}\nonumber
\begin{split}
&\lim_{m'\to+\infty}\int_{M_j}
v''_{\epsilon}(\Psi)|fF|^2 e^{-u(-v_{\epsilon}(\Psi))-h_{m',l}}\\
=&\int_{M_j}v''_{\epsilon}(\Psi)|fF|^2 e^{-u(-v_{\epsilon}(\Psi))-h_{l}}<+\infty.
\end{split}
\end{equation}

Note that $h_{m',l}\le h_{1,l}$ for any $m'$ and $h_{1,l}$ is quasi-plurisubharmonic function on $\overline{M_{j}}$. Then $$\inf_{m'}\inf_{M_j}e^{v_\epsilon(\Psi)-h_{m',l}-\delta \max\{\psi+T,2\log|F|\}}c(-v_\epsilon(\Psi))\ge C_j \inf_{m'}\inf_{M_j}e^{-h_{1,l}}>0,$$
where $C_j:=\inf_{M_j}e^{v_\epsilon(\Psi)-\delta \max\{\psi+T,2\log|F|\}}c(-v_\epsilon(\Psi))$ is a positive number.

It follows from inequality \eqref{estimate for m'} that
$$\sup_{m'}\int_{M_j}|u_{m',l,\epsilon,j}|^2<+\infty.$$

 Since the closed unit ball of the Hilbert space is
weakly compact, we can extract a subsequence $u_{m'',l,\epsilon,j}$ weakly
convergent to $u_{l,\epsilon,j}$ in $L^2(M_j,K_M)$ as $m''\to+\infty$.

Note that
$$\sup_{m''}\sup_{M_j}e^{v_\epsilon(\Psi)-h_{m'',l}-\delta \max\{\psi+T,2\log|F|\}}c(-v_\epsilon(\Psi))
\le \left(\sup_{M_j}\frac{e^{v_\epsilon(\Psi)}c(-v_\epsilon(\Psi))}{|F|^{2\delta}}\right)e^l.$$
It follows from Lemma \ref{weakly convergence} that we have $u_{m'',l,\epsilon,j}\sqrt{e^{v_\epsilon(\Psi)-h_{m'',l}-\delta \max\{\psi+T,2\log|F|\}}c(-v_\epsilon(\Psi))}$ weakly convergent to $u_{l,\epsilon,j}\sqrt{e^{v_\epsilon(\Psi)-h_{l}-\delta \max\{\psi+T,2\log|F|\}}c(-v_\epsilon(\Psi))}$.

It follows from inequality \eqref{estimate for m'} that

\begin{equation}
\label{estimate for m''}
\begin{split}
&\int_{M_j}
|u_{l,\epsilon,j}|^2 e^{v_\epsilon(\Psi)-h_{l}-\delta \max\{\psi+T,2\log|F|\}}c(-v_\epsilon(\Psi))\\
\le&\liminf_{m_1\to+\infty}\int_{M_j}
|u_{m'',l,\epsilon,j}|^2 e^{v_\epsilon(\Psi)-h_{m'',l}-\delta \max\{\psi+T,2\log|F|\}}c(-v_\epsilon(\Psi))\\
\le&\liminf_{m_1\to+\infty}
\int_{M_j}
v''_{\epsilon}(\Psi)|fF|^2 e^{-u(-v_{\epsilon}(\Psi))-h_{m'',l}}\\
=&\int_{M_j}v''_{\epsilon}(\Psi)|fF|^2 e^{-u(-v_{\epsilon}(\Psi))-h_{l}}<+\infty.
\end{split}
\end{equation}

As $\inf_{m''}\inf_{M_j}e^{-u(-v_{\epsilon}(\Psi))-h_{m'',l}-\delta\max\{\psi+T,2\log|F|\}}\ge \tilde{C}_j\inf_{M_j}e^{-h_{1,l}}>0$, where $\tilde{C}_j=\inf_{M_j}e^{-u(-v_{\epsilon}(\Psi))-\delta\max\{\psi+T,2\log|F|\}}$ is a positive number. Then it follows from inequality \eqref{estimate for h m'} that we know
$$\sup_{m''}\int_{M_j}|h_{m'',l,\epsilon,j}|^2<+\infty.$$

Since the closed unit ball of the Hilbert space is
weakly compact, we can extract a subsequence of $\{h_{m'',l,\epsilon,j}\}$ (also denote by $h_{m'',l,\epsilon,j}$) weakly
convergent to $h_{l,\epsilon,j}$ in $L^2(M_j,\wedge^{n,1}T^*M)$ as $m''\to+\infty$. Then it follows from Lemma \ref{weakly convergence} and $\lim_{m''\to +\infty}\delta_{m''}=0$ that $\sqrt{2\pi b\tilde{M}\tilde\delta_{m''}}h_{m'',l,\epsilon,j}$ is weakly convergent to 0.

Replace $m'$ by $m''$ in \eqref{d-bar realation u,h,lamda 2} and let $m''$ goes to $+\infty$, we have
\begin{equation}
\label{d-bar realation u,h,lamda 3}
\bar\partial u_{l,\epsilon,j}
=\bar\partial\left((1-v'_{\epsilon}(\Psi))fF^{1+\delta}\right).
\end{equation}

Denote $F_{l,\epsilon,j}:=-u_{l,\epsilon,j}+(1-v'_{\epsilon}(\Psi))fF^{1+\delta}$. It follows from \eqref{d-bar realation u,h,lamda 3} and inequality \eqref{estimate for m''} that we know $F_{l,\epsilon,j}$ is a holomorphic $(n,0)$ form on $M_j$ and

\begin{equation}
\label{estimate for F lej}
\begin{split}
&\int_{M_j}
|F_{l,\epsilon,j}-(1-v'_{\epsilon}(\Psi))fF^{1+\delta}|^2 e^{v_\epsilon(\Psi)-h_{l}-\delta \max\{\psi+T,2\log|F|\}}c(-v_\epsilon(\Psi))\\
\le&\int_{M_j}v''_{\epsilon}(\Psi)|fF|^2 e^{-u(-v_{\epsilon}(\Psi))-h_{l}}<+\infty.
\end{split}
\end{equation}

\

\emph{Step 6: Letting $\epsilon\to 0$. }

Note that
$$v''_{\epsilon}(\Psi)|fF|^2e^{-u(-v_{\epsilon})(\Psi)-h_l}\le \frac{2}{B}\sup_{\epsilon}\sup_{M_j}\left(e^{-u(-v_{\epsilon}(\Psi))+l}|F|^2\right)\mathbb{I}_{\{\Psi<-t_0\}}|f|^2.$$
It follows from $\int_{\overline{M_j}\cap \{\Psi<-t_0\}}|f|^2<+\infty$ and dominated convergence theorem that
\begin{equation}\label{estiamte for RHS 1}
\begin{split}
&\lim_{\epsilon\to 0}\int_{M_j}
v''_{\epsilon}(\Psi)|fF|^2 e^{-u(-v_{\epsilon}(\Psi))-h_{l}}\\
=&\int_{M_j}\frac{1}{B}\mathbb{I}_{\{-t_0-B<\Psi<-t_0\}}|fF|^2 e^{-u(-v(\Psi))-h_{l}}\\
\le & \left(\sup_{M_j}e^{-u(-v(\Psi))}\right)\int_{M_j}\frac{1}{B}\mathbb{I}_{\{-t_0-B<\Psi<-t_0\}}|fF|^2 e^{-h_{l}}
\end{split}
\end{equation}

When $\Psi<-t_0$, we know that $ \max\{\psi+T,2\log|F|\}= 2\log|F|$. Note that $h+\delta \max\{\psi+T,2\log|F|\}=\varphi+\Psi=\varphi_{\alpha}+(1+\delta)\max\{\psi+T,2\log|F|\}+\Psi$. Hence
\begin{equation}\label{estiamte for RHS 2}
\begin{split}
&\int_{M_j}\frac{1}{B}\mathbb{I}_{\{-t_0-B<\Psi<-t_0\}}|fF|^2 e^{-h_{l}}\\
=&\int_{M_j}\frac{1}{B}\mathbb{I}_{\{-t_0-B<\Psi<-t_0\}}|fF^{1+\delta}|^2 e^{-h_{l}-\delta 2\log|F|}\\
\le &\int_{M_j}\frac{1}{B}\mathbb{I}_{\{-t_0-B<\Psi<-t_0\}}|fF^{1+\delta}|^2 e^{-h-\delta \max\{\psi+T,2\log|F|\}}\\
=&\int_{M_j}\frac{1}{B}\mathbb{I}_{\{-t_0-B<\Psi<-t_0\}}|f|^2 e^{-\varphi_{\alpha}-\Psi}<+\infty.
\end{split}
\end{equation}

Note that $$\inf_{\epsilon}\inf_{M_j}e^{v_\epsilon(\Psi)-h_{l}-\delta \max\{\psi+T,2\log|F|\}}c(-v_\epsilon(\Psi))>0.$$
We have
$$\sup_{\epsilon}\int_{M_j}|F_{l,\epsilon,j}-(1-v'_{\epsilon}(\Psi))fF^{1+\delta}|^2<+\infty.$$

We also note that
$$\sup_{\epsilon}\int_{M_j}|(1-v'_{\epsilon}(\Psi))fF^{1+\delta}|^2\le (\sup_{M_j}|F|^{2(1+\delta)})\int_{\overline{M_j}\cap \{\Psi<-t_0\}}|f|^2<+\infty.$$
Then we know that
$$\sup_{\epsilon}\int_{M_j}|F_{l,\epsilon,j}|^2<+\infty,$$
and there exists a subsequence of $\{F_{l,\epsilon,j}\}$ (also denoted by $\{F_{l,\epsilon,j}\}$) compactly convergent to a holomorphic $(n,0)$ form $F_{l,j}$ on $M_j$. It follows from Fatou's Lemma and inequalities \eqref{estimate for F lej}, \eqref{estiamte for RHS 1}, \eqref{estiamte for RHS 2} that

\begin{equation}
\label{estimate for F lj}
\begin{split}
&\int_{M_j}
|F_{l,j}-(1-b(\Psi))fF^{1+\delta}|^2 e^{v(\Psi)-h_{l}-\delta \max\{\psi+T,2\log|F|\}}c(-v(\Psi))\\
\le &\liminf_{\epsilon\to 0}\int_{M_j}
|F_{l,\epsilon,j}-(1-v'_{\epsilon}(\Psi))fF^{1+\delta}|^2 e^{v_\epsilon(\Psi)-h_{l}-\delta \max\{\psi+T,2\log|F|\}}c(-v_\epsilon(\Psi))\\
\le& \liminf_{\epsilon\to 0}\int_{M_j}v''_{\epsilon}(\Psi)|fF|^2 e^{-u(-v_{\epsilon}(\Psi))-h_{l}}\\
\le &\left(\sup_{M_j}e^{-u(-v(\Psi))}\right)\int_{M_j}\frac{1}{B}\mathbb{I}_{\{-t_0-B<\Psi<-t_0\}}|fF|^2 e^{-h_{l}}\\
\le& \left(\sup_{M_j}e^{-u(-v(\Psi))}\right)
\int_{M_j}\frac{1}{B}\mathbb{I}_{\{-t_0-B<\Psi<-t_0\}}|f|^2 e^{-\varphi_{\alpha}-\Psi}<+\infty.
\end{split}
\end{equation}

\

\emph{Step 7: Letting $l\to+\infty$. }

It follows from $h_{l}\le h_1$ for any $l\in \mathbb{Z}^+$ and $h_1$ is a plurisubharmonic function on $M$ that
$$\inf\inf_{M_j} e^{v(\Psi)-h_{l}-\delta \max\{\psi+T,2\log|F|\}}c(-v(\Psi))\ge \hat{C}_j\inf_{M_j}e^{-h_1}>0,$$
where $\hat{C}_j=\inf_{M_j} e^{v(\Psi)-\delta \max\{\psi+T,2\log|F|\}}c(-v(\Psi))$ is a positive number. By inequality \eqref{estimate for F lj}, we have
$$\sup_l \int_{M_j}
|F_{l,j}-(1-b(\Psi))fF^{1+\delta}|^2 <+\infty.$$
Note that
$$\sup_{l}\int_{M_j}
|(1-b(\Psi))fF^{1+\delta}|^2\le\left(\sup_{M_j}|F|^{2(1+\delta)}\right)\int_{\overline{M_j}\cap \{\Psi<-t_0\}}|f|^2<+\infty. $$
Hence we know that
$$\sup_l \int_{M_j}
|F_{l,j}|^2 <+\infty,$$
and there exists a subsequence of $\{F_{l,j}\}$ (also denoted by $\{F_{l,j}\}$) compactly convergent to a holomorphic $(n,0)$ form $F_{j}$ on $M_j$. It follows from Fatou's Lemma and inequality \eqref{estimate for F lj} that

\begin{equation}
\label{estimate for F j}
\begin{split}
&\int_{M_j}
|F_{j}-(1-b(\Psi))fF^{1+\delta}|^2 e^{v(\Psi)-\varphi-\Psi}c(-v(\Psi))\\
=&\int_{M_j}
|F_{j}-(1-b(\Psi))fF^{1+\delta}|^2 e^{v(\Psi)-h-\delta \max\{\psi+T,2\log|F|\}}c(-v(\Psi))\\
\le&\liminf_{l\to+\infty}\int_{M_j}
|F_{l,j}-(1-b(\Psi))fF^{1+\delta}|^2 e^{v(\Psi)-h_{l}-\delta \max\{\psi+T,2\log|F|\}}c(-v(\Psi))\\
\le& \left(\sup_{M_j}e^{-u(-v(\Psi))}\right)
\int_{M_j}\frac{1}{B}\mathbb{I}_{\{-t_0-B<\Psi<-t_0\}}|f|^2 e^{-\varphi_{\alpha}-\Psi}<+\infty.
\end{split}
\end{equation}

\

\emph{Step 8: Letting $j\to+\infty$. }

It is easy to see that
\begin{equation}
\label{estimate for RHS last}
\begin{split}
 &\left(\sup_{M_j}e^{-u(-v(\Psi))}\right)
\int_{M_j}\frac{1}{B}\mathbb{I}_{\{-t_0-B<\Psi<-t_0\}}|f|^2 e^{-\varphi_{\alpha}-\Psi}\\
\le&
\left(\sup_{M}e^{-u(-v(\Psi))}\right)
\int_{M}\frac{1}{B}\mathbb{I}_{\{-t_0-B<\Psi<-t_0\}}|f|^2 e^{-\varphi_{\alpha}-\Psi}<+\infty.
\end{split}
\end{equation}

For fixed $j$, as $e^{v(\Psi)-\varphi-\Psi}c(-v(\Psi))$ has a positive lower bound on any $\overline{M_j}$, we have for $j_1>j$,
$$\sup_{j_1}\int_{M_j}
|F_{j_1}-(1-b(\Psi))fF^{1+\delta}|^2<+\infty.$$

Combining with
$$\int_{M_j}|(1-b(\Psi))fF^{1+\delta}|^2\le \left(\sup_{M_j}|F^{1+\delta}|^2\right)\int_{M_j\cap \{\Psi<-t_0\}}|f|^2<+\infty,$$
we know that for $j_1>j,$
$\int_{M_j}
|F_{j_1}|^2$ is uniformly bounded with respect to $j_1$.

By diagonal method, there exists a subsequence of ${\{F_j\}}$ (also denoted by ${\{F_j\}}$) compactly convergent to a holomorphic $(n,0)$ form $\tilde{F}$ on $M$. Then it follows from Fatou's Lemma, inequality \eqref{estimate for F j} and inequality \eqref{estimate for RHS last} that

\begin{equation}
\label{estimate for F}
\begin{split}
&\int_{M}
|\tilde{F}-(1-b(\Psi))fF^{1+\delta}|^2 e^{v(\Psi)-\varphi-\Psi}c(-v(\Psi))\\
\le&\liminf_{j\to +\infty}
\int_{M_j}
|F_{j}-(1-b(\Psi))fF^{1+\delta}|^2 e^{v(\Psi)-\varphi-\Psi}c(-v(\Psi))\\
\le& \liminf_{j\to +\infty}\left(\sup_{M_j}e^{-u(-v(\Psi))}\right)
\int_{M_j}\frac{1}{B}\mathbb{I}_{\{-t_0-B<\Psi<-t_0\}}|f|^2 e^{-\varphi_{\alpha}-\Psi}\\
\le & \left(\sup_{M}e^{-u(-v(\Psi))}\right)
\int_{M}\frac{1}{B}\mathbb{I}_{\{-t_0-B<\Psi<-t_0\}}|f|^2 e^{-\varphi_{\alpha}-\Psi}<+\infty.
\end{split}
\end{equation}

\

\emph{Step 9: ODE System.}

Now we want to find $\eta$ and $\phi$ such that
$(\eta+g^{-1})=e^{-v_\epsilon(\Psi_m)}e^{-\phi}\frac{1}{c(-v_{\epsilon}(\Psi_m))}$.
As $\eta=s(-v_{\epsilon}(\Psi_m))$ and $\phi=u(-v_{\epsilon}(\Psi_m))$, we have
$(\eta+g^{-1})e^{v_\epsilon(\Psi_m)}e^{\Phi}=\left((s+\frac{s'^2}{u''s-s''})e^{-t}e^u\right)\circ(-v_\epsilon(\Psi_m))$.\\

Summarizing the above discussion about $s$ and $u$, we are naturally led to a system of
ODEs:
\begin{equation}
\begin{split}
&1)(s+\frac{s'^2}{u''s-s''})e^{u-t}=\frac{1}{c(t)},\\
&2)s'-su'=1,
\end{split}
\label{ODE SYSTEM}
\end{equation}
when $t\in(T,+\infty)$.

We  solve the ODE system \eqref{ODE SYSTEM} and get
$u(t)=-\log(\frac{1}{\delta}c(T)e^{-T}+\int^t_T c(t_1)e^{-t_1}dt_1)$ and
$s(t)=\frac{\int^t_T(\frac{1}{\delta}c(T)e^{-T}+\int^{t_2}_T c(t_1)e^{-t_1}dt_1)dt_2+\frac{1}{\delta^2}c(T)e^{-T}}{\frac{1}{\delta}c(T)e^{-T}+\int^t_T
c(t_1)e^{-t_1}dt_1}$.

It follows that $s\in C^{\infty}([T,+\infty))$ satisfies
$s \ge \frac{1}{\delta}$ and $u\in C^{\infty}([T,+\infty))$ satisfies
$u''s-s''>0$.

As $u(t)=-\log(\frac{1}{\delta}c(T)e^{-T}+\int^t_T c(t_1)e^{-t_1}dt_1)$ is decreasing with respect to t, then
it follows from $-T \ge v(t) \ge \max\{t,-t_0-B_0\} \ge -t_0-B_0$, for any $t \le 0$
that
\begin{equation}
\begin{split}
\sup\limits_{M}e^{-u(-v(\psi))} \le
\sup\limits_{t\in[T,t_0+B]}e^{-u(t)}=\frac{1}{\delta}c(T)e^{-T}+\int^{t_0+B}_T c(t_1)e^{-t_1}dt_1.
\end{split}
\end{equation}

Combining with inequality \eqref{estimate for F}, we have

\begin{equation}\nonumber
\begin{split}
&\int_{M}
|\tilde{F}-(1-b(\Psi))fF^{1+\delta}|^2 e^{v(\Psi)-\varphi-\Psi}c(-v(\Psi))\\
\le & \left(\frac{1}{\delta}c(T)e^{-T}+\int^{t_0+B}_T c(t_1)e^{-t_1}dt_1\right)
\int_{M}\frac{1}{B}\mathbb{I}_{\{-t_0-B<\Psi<-t_0\}}|f|^2 e^{-\varphi_{\alpha}-\Psi}<+\infty.
\end{split}
\end{equation}

Lemma \ref{L2 method} is proved.


\vspace{.1in} {\em Acknowledgements}.
The first author and the second author were supported by National Key R\&D Program of China 2021YFA1003103. The first author was supported by NSFC-11825101, NSFC-11522101 and NSFC-11431013.

\bibliographystyle{references}
\bibliography{xbib}

\end{document}